\documentclass[11pt]{article}
\usepackage{amsmath,amssymb,amsthm,amsfonts,mathrsfs}
\usepackage{colonequals,graphics,xcolor,yfonts}
\usepackage{enumitem}
\setlist{
  listparindent=\parindent,
  parsep=3pt,
}
\usepackage{stmaryrd}
\usepackage[alphabetic]{amsrefs}
\usepackage[all,cmtip]{xy}
\usepackage[colorlinks,anchorcolor=blue,citecolor=blue,linkcolor=blue,urlcolor=blue,bookmarksdepth=2,bookmarksopen=true]{hyperref}
\usepackage{fullpage}
\usepackage{multicol}
\usepackage{ulem}

\usepackage{fancyhdr}
\fancypagestyle{titlepage}
{
	\fancyhf{}

	\fancyfoot[l]{
	\href{https://mathscinet.ams.org/mathscinet/msc/msc2020.html}
		{\emph{2020 Mathematics Subject Classification}}
		14E07, 14G15, 20B30.
		\\
		\emph{Keywords}: bijective Cremona transformations, birational permutations, parity
	}
}

\usepackage{titlesec}
\titleformat{\subsubsection}[runin]
  {\bf}{\thesubsubsection}{1em}{}

\usepackage{tikz}
\usetikzlibrary{cd,arrows}
\tikzset{>=stealth}
\tikzcdset{arrow style=tikz}
\tikzset{link/.style={column sep=1.8cm,row sep=0.16cm}}
\tikzset{link3/.style={column sep=1cm,row sep=0.18cm}} % small relation diagrams 

\newcommand{\red}[1]{\textcolor{red}{#1}}
\newcommand{\blue}[1]{\textcolor{blue}{#1}}

\AtBeginDocument{%
	\def\MR#1{}
}

\newcommand{\bZ}{\mathbb{Z}}    %Integers
\newcommand{\bQ}{\mathbb{Q}}    %Rational numbers
\newcommand{\bR}{\mathbb{R}}    %Real numbers
\newcommand{\bF}{\mathbb{F}}    %Finite field

\newcommand{\bA}{\mathbb{A}}    %Affine space
\newcommand{\cC}{\mathcal{C}}   %Moduli space
\newcommand{\bP}{\mathbb{P}}    %Projective space
\newcommand{\cP}{\mathcal{P}}    %Projective bundle over a finite set
\newcommand{\Alt}{\mathrm{Alt}} %Alternating group
\newcommand{\Aut}{\mathrm{Aut}} %Automorphism group
\newcommand{\Bir}{\mathrm{Bir}} %Group of birational self-maps
\newcommand{\BBir}{\mathrm{BBir}}%Group of birational permutations
\newcommand{\Cr}{\mathrm{Cr}}   %Cremona group
\newcommand{\BCr}{\mathrm{BCr}} %Cremona group without k-base-point
\newcommand{\Bl}{\mathrm{Bl}}   %Blow-up
\newcommand{\Bs}{\mathrm{Bs}}   %Indeterminacy locus
\newcommand{\Gal}{\mathrm{Gal}} %Galois group
\newcommand{\GL}{\mathrm{GL}}   %General linear group
\newcommand{\PGL}{\mathrm{PGL}} %Projective transformations
\newcommand{\Pic}{\mathrm{Pic}} %Picard group
\newcommand{\Sym}{\mathrm{Sym}} %Symmetric group

\newcommand{\ord}{\mathrm{ord}} %The order of a group element

\newtheorem*{thm*}{Theorem}
\newtheorem*{prop*}{Proposition}
\newtheorem*{cor*}{Corollary}
\newtheorem{thm}{Theorem}[section]
\newtheorem{prop}[thm]{Proposition}
\newtheorem{cor}[thm]{Corollary}
\newtheorem{lemma}[thm]{Lemma}

\numberwithin{equation}{section}

\theoremstyle{definition}

\newtheorem{eg}[thm]{Example}

\newtheorem{rmk}[thm]{Remark}

\begin{document}
\title{Bijective Cremona transformations of the plane}
\author{Shamil Asgarli
	\and Kuan-Wen Lai
	\and Masahiro Nakahara
	\and Susanna Zimmermann}
\date{}

\maketitle

%Contact information
\newcommand{\ContactInfo}{{
% additional braces for segregating \footnotesize
\footnotesize

\bigskip
\noindent S.~Asgarli,
\textsc{Department of Mathematics\\
University of British Columbia\\
Vancouver, BC V6T1Z2, Canada}\par\nopagebreak
\noindent\textsc{Email:} \noindent\texttt{sasgarli@math.ubc.ca}

\bigskip
\noindent K.-W.~Lai,
\textsc{Mathematisches Institut der Universit\"at Bonn\\
Endenicher Allee 60, 53121 Bonn, Deutschland}\par\nopagebreak
\noindent\textsc{Email:} \texttt{kwlai@math.uni-bonn.de}

\bigskip
\noindent M.~Nakahara,
\textsc{Department of Mathematics\\
University of Washington\\
Seattle, WA 98195, USA}\par\nopagebreak
\noindent\textsc{Email:} \noindent\texttt{mn75@uw.edu}

\bigskip
\noindent S.~Zimmermann,
\textsc{Univ Angers, CNRS, LAREMA, SFR MATHSTIC, F-49000 Angers, France}\par\nopagebreak
\noindent\textsc{Email:} \noindent\texttt{susanna.zimmermann@univ-angers.fr}
}}

\thispagestyle{titlepage}

%--------------------Abstract
\begin{abstract}
We study the birational self-maps of the projective plane over finite fields that induce permutations on the set of rational points. As a main result, we prove that no odd permutation arises over a non-prime finite field of characteristic two, which completes the investigation initiated by Cantat about which permutations can be realized this way. Main ingredients in our proof include the invariance of parity under groupoid conjugations by birational maps, and a list of generators for the group of such maps.
\end{abstract}

\setcounter{tocdepth}{2}
\tableofcontents

%--------------------Introduction
\section{Introduction}
\label{sect:introduction}

We call a birational self-map of a variety a \emph{birational permutation} if both the map and its inverse are defined at all rational points on the variety. In particular, such a map induces a bijection on the set of rational points. Over a finite field, the rational points form a finite set, so such a bijection induces a permutation in the usual sense. Fixing a variety and a finite ground field, what kind of permutations on the rational points can be realized this way?

In this paper, we focus on the birational self-maps of a projective space $\bP^n$, that is, the \emph{Cremona transformations}. They form a group $\Cr_n(k)$ where $k$ is the ground field. We say that a Cremona transformation is \emph{bijective} if it is a birational permutation. Clearly, bijective elements form a subgroup $\BCr_n(k)\subset\Cr_n(k)$.
When $k = \bF_q$, the finite field of $q$ elements, the actions of bijective elements on the set of $\bF_q$-points determines a group homomorphism
$$\xymatrix{
    \sigma_{q}\colon\BCr_n(\bF_q)\ar[r] & \Sym(\bP^n(\bF_q))
}$$
where $\Sym(\bP^n(\bF_q))$ is the symmetric group of the set $\bP^n(\bF_q)$. Let $\Alt(\bP^n(\bF_q))\subset\Sym(\bP^n(\bF_q))$ be the alternating subgroup, which consists of even permutations. In the case $n=2$, it is known that the image of $\sigma_q$ satisfies
\begin{itemize}
    \item $\mathrm{Im}(\sigma_q) = \Sym(\bP^2(\bF_q))$ if $q$ is odd or $q = 2$,
    \item $\mathrm{Im}(\sigma_q)\supset\Alt(\bP^2(\bF_q))$ if $q = 2^m\geq4$.
\end{itemize}
This result was mainly proved by Cantat \cite{Can09}, but the original proof has a minor gap. In Section~\ref{sect:realizingPermutations}, we review Cantat's construction and fill in the gap with a theorem by Cohen (Theorem~\ref{thm:cohen}) about primitive roots of $\bF_{q^2}$.

The main focus of this paper is the case $q = 2^m\geq 4$. We prove that:

\begin{thm}
\label{mainthm:intro}
For $q=2^m\geq4$, the group $\BCr_2(\bF_q)$ produces only even permutations on $\bP^2(\bF_q)$. As a result, we have $\mathrm{Im}(\sigma_q)=\Alt(\bP^2(\bF_q))$.
\end{thm}

Our proof for Theorem~\ref{mainthm:intro} relies on being able to transfer the parity problem from one surface to another. Let $\Bir_k(X)$ denote the group of birational self-maps of a variety $X$ over a field $k$. In the same spirit of the notation $\BCr_n(k)$, we denote by $\BBir_k(X)\subset\Bir_k(X)$ the subgroup of birational permutations. For surfaces over $\bF_q$, where $q = 2^m\geq4$, the parity of a birational permutation is invariant under groupoid conjugations by birational maps in the following sense:

\begin{thm}
\label{thm:parityBirMapIntro}
Let $X$ and $Y$ be smooth surfaces over $\bF_q$, where $q=2^m\geq4$, together with two birational permutations $\alpha\in\BBir_{\bF_q}(X)$ and $\beta\in\BBir_{\bF_q}(Y)$. Suppose that there exists a birational map $h\colon X\dashrightarrow Y$ such that $\alpha = h^{-1}\beta h$, i.e., the following diagram commutes:
$$\xymatrix{
    X\ar@{-->}[d]_-h\ar@{-->}[rr]^-{\alpha}
    && X\ar@{-->}[d]^-h\\
    Y\ar@{-->}[rr]^-{\beta}
    && Y.
}$$
Then the permutations induced by $\alpha$ on $X(\bF_q)$ and $\beta$ on $Y(\bF_q)$ have the same parity.
\end{thm}

Throughout the paper, we will call the groupoid conjugation demonstrated in Theorem~\ref{thm:parityBirMapIntro} simply as ``conjugation''. Our next result studies birational permutations on conic bundles over $\bP^1$, del Pezzo surfaces, and bijective Cremona maps on $\bP^2$ of finite order. As an application of Theorem~\ref{thm:parityBirMapIntro}, we obtain:

\begin{thm}
\label{thm:evenPermutationsIntro}
Over $\bF_q$, $q = 2^m\geq 4$, a birational permutation on a smooth surface induces an even permutation on the set of $\bF_q$-points if it is conjugate to
\begin{itemize}
    \item a birational permutation on a conic bundle over $\bP^1$ preserving the fiber class,
    \item an automorphism of a rational del Pezzo surface, or
    \item an element of $\BCr_2(\bF_q)$ of finite order.
\end{itemize}
\end{thm}

To complete the proof of Theorem~\ref{mainthm:intro}, we first produce a list of generators for the bijective Cremona group, and then show that every generator is a composition of maps described as in Theorem~\ref{thm:evenPermutationsIntro}. We state the result on the generators below and refer the reader to Lemma~\ref{lem:generating_set_BCr} for the complete list.

\begin{thm}
\label{thm:generatorIntro}
Let $k$ be a perfect field and $\mathbf{T}\subset\Cr_2(k)$ be the set of generators for $\Cr_2(k)$ given by Iskovskikh \cite{Isk91}. Then $\mathbf{T}\cap\BCr_2(k)$ forms a set of generators for $\BCr_2(k)$.
\end{thm}

\begin{rmk}
\label{rmk:history}
The first version of this paper was announced on the arXiv in 2019, where Theorem~\ref{mainthm:intro} remained as a conjecture. In that version, we proved that all but the \emph{quintic transformations} among the generators in Theorem~\ref{thm:generatorIntro} induce only even permutations, and verified with Magma \cite{Magma} that the quintic transformations over $\bF_q$ for $q=4,8,16$ are all even. In June 2021, we communicated with Julia Schneider on the \emph{central symmetry} of a relation diagram of Sarkisov links, which allowed us to attack the quintic transformations and prove our conjecture.

In parallel to our work on the quintic transformations, we learned that Genevois, Lonjou, and Urech \cite{GLU21} also came up with a proof for Theorem~\ref{mainthm:intro} based on our Theorem~\ref{thm:evenPermutationsIntro} and the main theorem of \cite{LS21} with a more combinatorial approach. In fact, they observed that parity can still be defined for a birational self-map on a smooth rational surface over $\bF_{q}$, $q=2^m\geq 4$, even if the map is not bijective, which allowed them to prove Theorem~\ref{mainthm:intro} not only for $\bP^2$ but also for all smooth rational surfaces.
\end{rmk}

\subsubsection*{Organization of the paper}
In Section~\ref{sect:realizingPermutations}, we discuss the realizability of all permutations on the rational points in the plane over finite fields of odd characteristics and $\bF_2$. We study the parity problem over a non-prime field of characteristic~$2$ throughout Sections~\ref{sect:bir-inv-parity}--\ref{sect:no-odd-permutation}, where we assume that $k = \bF_q$ with $q = 2^m\geq4$ unless otherwise specified. In Section~\ref{sect:bir-inv-parity}, we begin with the analysis of the parities induced by linear transformations and then prove Theorem~\ref{thm:parityBirMapIntro}. In Section~\ref{sect:birratsurf}, we study the birational permutations on certain rational surfaces and prove Theorem~\ref{thm:evenPermutationsIntro}. In Section~\ref{sect:no-odd-permutation}, we exhibit a list of generators for $\BCr_2(k)$ when $k$ is a perfect field and prove Theorem~\ref{thm:generatorIntro}. Then we analyze whether each generator induces an even permutation and deduce Theorem~\ref{mainthm:intro}. In Section~\ref{sect:basicProperties}, we answer a few questions about $\BCr_2(k)$ as a subgroup of $\Cr_2(k)$, which include whether it is finitely generated, what its index is, and whether it is a normal subgroup.

\subsubsection*{Acknowledgements}
We thank Brendan Hassett for suggesting us the problem in the present paper. We also thank Zinovy Reichstein for a quick proof that $\mathrm{BCr}_2(k)$ is not finitely generated when $k$ is uncountable. We thank Julia Schneider for discussing with us on the key ideas that allowed us to attack the quintic transformations. Before we are able to prove our conjecture, Lian Duan assisted us designing a Magma code that can compute efficiently the parities of all possible quintic transformations over $\bF_q$ for $q=4,8,16$. We are very grateful for his generous help. Finally, we thank the anonymous referee for their valuable suggestions. During this project, the first author was partially supported by funds from NSF Grant DMS-1701659. The second author is supported by the ERC Synergy Grant ERC-2020-SyG-854361-HyperK. The third author was supported by EPSRC grant EP/R021422/2. The last author was supported by FIBALGA ANR-18-CE40-0003-01, PEPS 2019 ``JC/JC'' and \'Etoiles Montantes de la R\'egion Pays de la Loire.

%--------------------Realizing arbitrary permutations
\section{Realizing arbitrary permutations}
\label{sect:realizingPermutations}

\begin{thm}[\cite{Can09}]
\label{thm:realization}
The image of the homomorphism
$\xymatrix{
    \sigma_q\colon\BCr_2(\bF_q) \to \Sym(\bP^2(\bF_q))
}$
satisfies
\begin{itemize}
    \item $\mathrm{Im}(\sigma_q) = \Sym(\bP^2(\bF_q))$ if $q$ is odd or $q = 2$, and
    \item $\mathrm{Im}(\sigma_q)\supset\Alt(\bP^2(\bF_q))$ if $q = 2^m\geq4$.
\end{itemize}
\end{thm}

Cantat's proof of Theorem~\ref{thm:realization} is built upon a property about the subgroups of $\Sym(\bP^{n}(\bF_q))$ that contain $\mathrm{PSL}_{n+1}(\bF_q)$: The elements in $\Sym(\bP^n(\bF_q))$ which preserve the collinearity, i.e., map collinear points to collinear points, are called \emph{collineations}. They form a subgroup
$$
    \mathrm{P\Gamma L}_n(\bF_q)\subset\Sym(\bP^n(\bF_q))
$$
which contains $\mathrm{PSL}_{n+1}(\bF_q)$.

\begin{thm}[\cites{Bha81, KM74, Lis75, Pog74}]
\label{thm:groupTheory}
Let $G\subset\Sym(\bP^n(\bF_q))$ be a subgroup.
If $G$ contains $\mathrm{PSL}_{n+1}(\bF_q)$,
then either $G\subset\mathrm{P\Gamma L}_n(\bF_q)$ or $G\supset \mathrm{Alt}(\bP^n(\bF_q))$.
\end{thm}

Applying this result to the image $\sigma_q(\BCr_2(\bF_q))$, Cantat proved that $\sigma_{q}$ is surjective by constructing an element $f\in\BCr_2(\bF_q)$ which
\begin{itemize}[itemsep=-3pt,topsep=5pt]
    \item does not preserve the collinearity on $\bP^2(\bF_q)$, and
    \item induces an odd permutation on $\bP^2(\bF_q)$.
\end{itemize}
Our main goal in this section is to exhibit the construction of $f$ explicitly using input from the theory of primitive roots by Cohen.

%----------Special birational maps on a quadric surface
\subsection{Special birational maps on a quadric surface}
\label{subsect:spbirmap_quadric}

We first recall a key construction in \cite{Can09}*{Section~3}. Fix a smooth quadric $Q$ and a line $L$ in $\bP^{3}$, both defined over $\bF_q$, such that $L$ meets $Q$ in a pair of conjugate points over the extension $\bF_{q^2}/\bF_q$. The projection from $L$ induces a rational map
$\pi_L\colon Q\dashrightarrow\bP^{1}$ fibered in the conics cut out by the planes containing $L$. Assume further that there exists an $\bF_q$-point $P_0$ in the base $\bP^1$ over which the fiber $C_0\colonequals\pi_L^{-1}(P_0)$ is smooth.

This setting implies that every degenerate fiber over $\bF_q$ is a union of distinct lines $L_1\cup L_2$ conjugate to each other over $\bF_{q^2}/\bF_q$, on which the node $P\colonequals L_1\cap L_2$ appears as the only $\bF_q$-point. The projection from $P$ defines a birational map $\pi_P\colon Q\dashrightarrow\bP^{2}$. Let us organize these maps into a diagram:
\begin{equation}
\label{eqn:projFromLP}
\begin{aligned}
\xymatrix{
    Q
    \ar@{-->}[d]_-{\pi_L:\text{ conic fibration}}
    \ar@{-->}[r]_-\sim^-{\pi_P} & \bP^{2}\\
    \;\bP^{1}. &
}
\end{aligned}
\end{equation}
Cantat's construction of a desired $f\in\BCr_2(\bF_q)$ can be divided into two parts:
\begin{enumerate}[label=(\arabic*)]
    \item\label{enum:birQuadric}
    Constructing a birational self-map $g$ on $Q$ that preserves the fiber structure, acts as a prescribed odd permutation on $C_0(\bF_q)$ and as the identity on $\bF_q$-points of the other fibers.
    \item\label{enum:birProj}
    Descending $g$ down to $\bP^2$ as $f\colonequals\pi_P\circ g\circ\pi_P^{-1}$, then showing that $f$ induces an odd permutation on the $\bF_q$-points and does not preserve collinearity.
\end{enumerate}

\begin{eg}
\label{eg:projFromLP_oddChar}
Assume that $q$ is odd. Let $[x:y:z:w]$ be a system of homogeneous coordinates on $\bP^{3}$. Choose a non-square $t\in\bF_q$, namely, $t\neq s^2$ for all $s\in\bF_q$. Then the data
\begin{gather*}
	Q\colonequals\left\{
		x^2 - ty^2 + z^2 = w^2
	\right\}\subset\bP^{3},
	\quad
	L\colonequals\left\{
	    z = w = 0
	\right\}\subset\bP^{3},
\end{gather*}
and $P\colonequals[0:0:1:1]\in Q$ provide an example of (\ref{eqn:projFromLP}). Here the projection map is explicitly given by $\pi_L([x:y:z:w]) = [z:w]$, and the degenerate fiber through $P$ is defined as $x^2 - ty^2 = 0$ on the plane parametrized by the map,
$$
	\bP^2\hookrightarrow\bP^{3}:
	[x:y:u]\mapsto[x:y:u:u].
$$
For a smooth fiber over $\bF_q$, one can choose
\begin{equation}
\label{eqn:centralConic}
    C_0\colonequals
    \pi_L^{-1}([0:1])
    = \left\{
        x^2 - ty^2 = w^2
    \right\}\subset Q.
\end{equation}
Note that $C_0$ lies on the plane $\{z=0\}$.
\end{eg}

Let us construct the map $g$ as in \ref{enum:birQuadric} in the case of odd characteristics using Example~\ref{eg:projFromLP_oddChar}. (The case of characteristic $2$ will be discussed in \S\ref{subsubsect:constrInChar2}.) The process starts by constructing a suitable automorphism on the smooth fiber $C_0$ in \eqref{eqn:centralConic} and then extend it to $Q$. Consider the automorphism on the plain $\{z=0\}$:
$$
    \bP^2\to\bP^2:
    [x:y:w]\mapsto[\alpha x+t\beta y:\beta x+\alpha y:w],
$$
where the parameter $(\alpha,\beta)$ is a point on the affine conic
$$
    S^\circ\colonequals\left\{
        \alpha^2 - t\beta^2 = 1
    \right\}
    \subset\bA^2.
$$
Note that this is the identity map when $(\alpha,\beta) = (1,0)$. For each $(\alpha,\beta)\in S^\circ$, the formula induces an automorphism $g_0\colon C_0\xrightarrow{\sim} C_0$ as one can verify that
\begin{equation}
\label{eqn:preserveConic}
	(\alpha x + t\beta y)^2 - t(\beta x + \alpha y)^2
	= x^2 - ty^2.
\end{equation}

\begin{rmk}
The map $g_0$ can be expressed as
$$
    g_0\colon C_0\xrightarrow{\sim}C_0:
    [x:y:w]\mapsto
    [\alpha x+t\beta y:\beta x+\alpha y:\gamma w]
$$
where $[\alpha:\beta:\gamma]\in\bP^2$ is any $\bF_q$-point on the (projective) conic
$$
    S\colonequals\{\alpha^2 - t\beta^2 = \gamma^2\}\subset\bP^2.
$$
Note that every $\bF_q$-point on $S$ has $\gamma\neq0$ since $t\in\bF_q$ is a non-square. Due to this, we assume that $\gamma = 1$ for the convenience of computation.
\end{rmk}

In the following, we exhibit how to extend $g_0$ to the whole quadric $Q$ as a birational permutation that fixes the $\bF_q$-points not lying on $C_0$. The method is built upon the following lemma about interpolations. Although we only need the case $n=1$ for our purposes, we present the proof of the general case as it is not any harder.

%Interpolation lemma
\begin{lemma}
\label{lemma:interpolation}
Let $\bF_q$ be a finite field. Fix any $P_0\in\bP^n(\bF_q)$ and $P_1, P_2\in\bP^1(\bF_q)$ such that $P_1\neq P_2$. Then there exists a rational map $h\colon\bP^n\dashrightarrow\bP^1$ over $\bF_q$ such that
\begin{itemize}
    \item $h(P_0) = P_1$,
    \item $h(P) = P_2$ for all $P\in\bP^n(\bF_q)\setminus\{P_0\}$.
\end{itemize}
\end{lemma}

\begin{proof}
For every $P\in\bF_q^{n+1}\setminus\{0\}$, there exists a homogeneous polynomial $f_{P}\in \bF_q[x_0, ..., x_n]$ such that for each $P'\in\bF_q^{n+1}\setminus\{0\}$,
$$
    f_{P}(P') = \begin{cases}
    1 &\text{if }P'=\lambda P\text{ for }\lambda\in\bF_{q}^{\ast} \\
    0 &\text{otherwise}
\end{cases}
$$
Indeed, we may assume that $P=(1,0,...,0)$ after applying a $\GL_{n+1}(\bF_q)$-action, in which case the polynomial
$$
    f_{P} = x_0^{q-1} \prod_{i=1}^{n} (x_0^{q-1}-x_{i}^{q-1})
$$
satisfies the desired property. (The function $f_{P}$ serves the role of the Dirac delta function.) Next, consider the homogeneous polynomial
$$
    f\colonequals
    \frac{1}{q-1}\sum_{P\in\bF_q^{n+1}} f_P.
$$
Then $f(P)=1$ for every $P\in\bF_q^{n+1}\setminus\{0\}$. In order to prove the lemma, let us write $P_1=[\alpha:\beta]$, $P_2=[\gamma:\delta]$, and lift $P_0\in\bP^n(\bF_q)$ to $\widetilde{P}_0\in\bF_q^{n+1}$. Consider $h: \bP^n\dashrightarrow \bP^1$ defined by
$$
    h(P) = [
        \gamma f(P)+(\alpha-\gamma)f_{\widetilde{P}_0}(P) :
        \delta f(P)+(\beta-\delta)f_{\widetilde{P}_0}(P)
    ].
$$
Then $h$ is well-defined, and has the desired interpolation property. \end{proof}

\begin{prop}
\label{prop:extension}
For every $(\alpha_0,\beta_0)\in S^\circ(\bF_q)$, the automorphism
$$
    g_0\colon C_0\xrightarrow{\sim}C_0:
    [x:y:w]\mapsto[\alpha_0x+t\beta_0y:\beta_0x+\alpha_0y:w].
$$
extends to a birational self-map $g\colon Q\dashrightarrow Q$ that preserves the fibration $\pi_L\colon Q\dashrightarrow\bP^{1}$ and satisfies
\begin{itemize}
    \item $g|_{C_0} = g_0$,
    \item $g|_{C} = \mathrm{id}$ for all $\bF_q$-fibers $C\neq C_0$ of $\pi_L$.
\end{itemize}
(This element $g$ can be viewed as the group version of the Dirac delta function.)
\end{prop}

\begin{proof}
Let $\zeta$ be an affine coordinate on the base $\bP^1$ of the fibration $\pi_L$. We identify $S^\circ$ as an open subset of $\bP^1$ via the stereographic projection from $(-1,0)\in S^\circ$:
$$
    S^\circ\hookrightarrow\bP^1:
    (\alpha,\beta)\mapsto\zeta = \frac{\beta}{1+\alpha}.
$$
Let $\zeta_0\in\bP^1$ denote the image of $(\alpha_0,\beta_0)\in S^\circ$ under this map. Note that $(1,0)\in S^\circ$ is mapped to $0\in\bP^1$. Note also that we can recover $\alpha$ and $\beta$ by
\begin{equation}
\label{eqn:inverseStereo}
    \alpha = \frac{1+t\zeta^2}{1-t\zeta^2},
    \quad
    \beta = \frac{2\zeta}{1-t\zeta^2}.
\end{equation}

Let $P_0\colonequals[0:1] = \pi_L(C_0)\in\bP^{1}$. By Lemma~\ref{lemma:interpolation}, there exists a rational function
$\zeta = h(z,w)$ on the base $\bP^1$ over $\bF_q$ such that $h(P_0) = \zeta_0$ and $h(P) = 0$ for all $P\in\bP^{1}(\bF_q)\setminus\{P_0\}$.
Substituting it into (\ref{eqn:inverseStereo}), we obtain two rational functions
\[
    \alpha(z,w) = \frac{1+th(z,w)^2}{1-t h(z,w)^2},
    \quad
    \beta(z,w) = \frac{2h(z,w)}{1-th(z,w)^2},
\]
which determine a birational self-map on $Q$ via the inhomogeneous formula:
\[
    g\colon Q\dashrightarrow Q:
    [x:y:z:w]\mapsto
    [\alpha x+t\beta y:\beta x+\alpha y:z:w].
\]
Note that this is well-defined due to the same computation as (\ref{eqn:preserveConic}). By construction, we have
\begin{itemize}
    \item $(\alpha(P_0),\beta(P_0)) = (\alpha_0,\beta_0)$,
    \item $(\alpha(P),\beta(P)) = (1,0)$ for all $P\in\bP^{1}\setminus\{P_0\}$,
\end{itemize}
which respectively implies that $g|_{C_0} = g_0$ and that $g|_C = \mathrm{id}$ for all $\bF_q$-fibers $C\neq C_0$.
\end{proof}

%----------Odd permutations on the smooth fiber
\subsection{Odd permutations on the smooth fiber}
\label{subsect:OddOnSmoothFiber}

Let us retain the notation from the previous section. Our goal here is to find a $g_0$ which acts transitively on $C_0(\bF_q)$ and thus induces an odd permutation. Note that, as $C_0\cong\bP^1$, it is not hard to find an automorphism on $C_0$ which induces an odd permutation on the $\bF_q$-points. However, it is not obvious that every such automorphism can be extended to $Q$ while keeping control on the induced permutation on the other $\bF_q$-points. In the following, we identify $\bF_{q^2}\cong\bF_q\oplus\sqrt{t^{-1}}\bF_q$ and view $C_0\cong\bP^1$ as the projectivization
$$
    C_0\cong\bP(\bF_q\oplus\sqrt{t^{-1}}\bF_q)
    \cong\bP(\bF_{q^2}).
$$

\begin{lemma}
\label{lemma:multiplicativeAuto}
Assume that $g_0$ is not the identity map, that is, $\alpha\neq 1$. Then, under a suitable choice of isomorphism $C_0\cong\bP(\bF_{q^2})$, the action of $g_0$ can be obtained as the multiplication on $\bF_{q^2}$ by the element
\begin{equation}
\label{eqn:multiplicativeAuto}
    \beta + (\alpha - 1)\sqrt{t^{-1}}\in\bF_{q^2}
\end{equation}
where $\alpha, \beta\in\bF_q$ satisfy $\alpha^2 - t\beta^2 = 1$.
\end{lemma}

\begin{proof}
First we identify $C_0$ with $\bP^1$ using the stereographic projection from $[-1:0:1]\in C_0$.
On the affine chart $w=1$, this map can be defined as
$$
    \theta\colon C_0\xrightarrow{\sim}\bP^1:
    (x,y)\mapsto\zeta = \frac{y}{1+x}
$$
where $\zeta$ is an affine coordinate on $\bP^1$.
Its inverse $\theta^{-1}\colon\bP^1\xrightarrow{\sim} C_0$ is
$$
    x = \frac{1+t\zeta^2}{1-t\zeta^2},\quad
    y = \frac{2\zeta}{1-t\zeta^2}.
$$
We claim that $g_\theta\colonequals\theta\circ g_0\circ\theta^{-1}\colon\bP^1\xrightarrow{\sim}\bP^1$ is given by the formula
\begin{equation}
\label{eqn:holomorphhicAuto}
    g_\theta(\zeta)
    = \frac{\beta\zeta + t^{-1}(\alpha-1)}{(\alpha-1)\zeta + \beta}.
\end{equation}
Indeed, as $g_0(x,y) = (\alpha x + t\beta y, \beta x + \alpha y)$ in the affine coordinates,
a straightforward computation shows that
\begin{gather*}
    g_\theta(\zeta)
    = \frac{\beta x(\zeta) + \alpha y(\zeta)}
        {1 + \alpha x(\zeta) + t\beta y(\zeta)}
    = \frac{
        \beta\left(\frac{1+t\zeta^2}{1-t\zeta^2}\right)
        +\alpha\left(\frac{2\zeta}{1-t\zeta^2}\right)
        }{
        1 + \alpha\left(\frac{1+t\zeta^2}{1-t\zeta^2}\right)
        + t\beta\left(\frac{2\zeta}{1-t\zeta^2}\right)
        }\\
    = \frac{
        \beta(1+t\zeta^2)
        +\alpha(2\zeta)
        }{
        (1-t\zeta^2) + \alpha(1+t\zeta^2)
        + t\beta(2\zeta)
        }
    = \frac{
        t\beta\zeta^2 + 2\alpha\zeta + \beta
        }{
        t(\alpha-1)\zeta^2 + 2t\beta\zeta + (\alpha+1)
        }.
\end{gather*}
Using the quadratic formula and the fact that $\alpha^2 - t\beta^2 = 1$, the numerator and denominator can be decomposed into linear terms:
$$
    g_\theta(\zeta) = \frac{
        t\beta(\zeta + \frac{\alpha - 1}{t\beta})
            (\zeta + \frac{\alpha + 1}{t\beta})
        }{
        t(\alpha - 1)(\zeta + \frac{\alpha + 1}{t\beta})^2
        }
        = \frac{
        t\beta(\zeta + \frac{\alpha - 1}{t\beta})
        }{
        t(\alpha-1)(\zeta + \frac{\alpha + 1}{t\beta})
        }
$$
which can be further simplified as
$$
    g_\theta(\zeta) = \frac{
        t\beta\zeta + (\alpha - 1)
        }{
        t(\alpha - 1)\zeta + \frac{(\alpha^2 - 1)}{\beta}
        }
        = \frac{
        t\beta\zeta + (\alpha - 1)
        }{
        t(\alpha - 1)\zeta + t\beta
        }
        = \frac{
        \beta\zeta + t^{-1}(\alpha - 1)
        }{
        (\alpha - 1)\zeta + \beta
        },
$$
as claimed. Under the identification $\bP^1\cong\bP(\bF_q\oplus\sqrt{t^{-1}}\bF_q)$, formula (\ref{eqn:holomorphhicAuto}) can be rewritten as
$$
    g_\theta = \begin{pmatrix}
    \beta & t^{-1}(\alpha-1)\\
    \alpha-1 & \beta
    \end{pmatrix}\in\PGL_2(\bF_q).
$$
This matrix acts on $\bF_{q^2}\cong\bF_q\oplus\sqrt{t^{-1}}\bF_q$ as the multiplication by $\beta + (\alpha - 1)\sqrt{t^{-1}}$, which completes the proof.
\end{proof}

Due to this lemma, to find $g_0$ that acts on $C_0(\bF_q)$ transitively, it is sufficient to find a primitive root of $\bF_{q^2}$ of the form (\ref{eqn:multiplicativeAuto}).
To attain this, we use the following result by Cohen:

\begin{thm}[\cite{Coh83}*{Theorem~1.1}]
\label{thm:cohen}
Let $\{\theta_1,\theta_2\}$ be a basis of $\bF_{q^2}$ over $\bF_q$ and let $a_1$ be a non-zero member of $\bF_q$.
Then there exists a primitive root of $\bF_{q^2}$ of the form $a_1\theta_1+a_2\theta_2$ for some $a_2\in\bF_{q}$.
\end{thm}

\begin{cor}
\label{cor:specialPrimitiveRoot}
There exists a primitive root of $\bF_{q^2}$ of the form
$$
    \beta + (\alpha - 1)\sqrt{t^{-1}}\in\bF_{q^2}^\ast
$$
where $\alpha,\beta\in\bF_q$ satisfy $\alpha^2 - t\beta^2 = 1$.
\end{cor}
\begin{proof}
By applying Theorem~\ref{thm:cohen} to the basis $\left\{1,\sqrt{t^{-1}}\right\}$, we find $c\in\bF_q$ such that
$$
    \xi\colonequals c - \frac{t}{2}\sqrt{t^{-1}}\in\bF_{q^2}
$$
is a primitive root of $\bF_{q^2}$. We claim that $\xi^{-1}$ can be expressed as the required form. Let us write $\xi^{-1} = \beta + (\alpha - 1)\sqrt{t^{-1}}$, then
$$
    \xi = \frac{\beta}{\beta^2 - t^{-1}(\alpha - 1)^2}
    - \frac{\alpha - 1}{\beta^2 - t^{-1}(\alpha - 1)^2}\sqrt{t^{-1}}.
$$
Equating the coefficients of $\sqrt{t^{-1}}$ in the above two expressions for $\xi$, we obtain
$$
    \frac{t}{2} =  \frac{\alpha - 1}{\beta^2 - t^{-1}(\alpha - 1)^2}
$$
which implies that
$
    (\alpha - 1)^2 - t\beta^2 = -2(\alpha - 1),
$
thus $\alpha^2 - t\beta^2 = 1$, as required.
\end{proof}

%----------Induced actions on the projective plane
\subsection{Induced actions on the projective plane}
\label{subsect:actionOnP2}

Here we complete the proof of Theorem~\ref{thm:realization}. We will first treat the case when $q$ is odd using what we have established in the previous sections. The case $q=2$ will be treated separately with a similar strategy, where we will also prove that the image of $\sigma_{q}$ contains $\Alt(\bP^2(\bF_{q}))$ for $q=2^m\geq 4$.

%-----Realizing arbitrary permutations for odd q
\subsubsection{Proof of Theorem~\ref{thm:realization} for odd $q$}
\label{subsubsect:realization_q-odd}

Proposition~\ref{prop:extension} and Corollary~\ref{cor:specialPrimitiveRoot} imply the existence of a birational self-map $g\colon Q\dashrightarrow Q$ acting transitively on the $\bF_q$-points of a smooth fiber $C_0$ and leaving all the other $\bF_q$-fibers fixed. Recall that $\pi_P\colon Q\dashrightarrow\bP^2$ is the projection from the node $P$ of a degenerate fiber of the fibration $\pi_L\colon Q\dashrightarrow\bP^1$. In particular, it has $P$ as the only indeterminacy point and contracts the two branches of the degenerate fiber. In particular, it maps the smooth fiber $C_0$ isomorphically onto a smooth conic $C\colonequals\pi_P(C_0)\subset\bP^2$.

\begin{prop}
\label{prop:inducedmapP2}
The composition $f \colonequals \pi_{P}\circ g \circ \pi_{P}^{-1} \colon\bP^2\dashrightarrow \bP^2$ satisfies the following properties:
\begin{enumerate}[label=\textup{(\arabic*)}]
\item $f\in \BCr_2(\bF_q)$.
\item $f$ fixes all the $\bF_q$-points away from the conic $C$.
\item $f$ acts transitively on $C(\bF_q)$ and thus permutes as a $(q+1)$-cycle.
\item\label{prop:inducedmapP2:colli} There exists a triple of collinear points $P_1, P_2, P_3 \in \bP^2(\bF_q)$ such that $f(P_1), f(P_2), f(P_3)$ are not collinear. 
\end{enumerate}
In particular, the induced permutation on $\bP^2(\bF_q)$ by $f$ does not preserve collineation, and moreover, induces a $(q+1)$-cycle, and hence has odd sign as $q$ is odd.
\end{prop}

\begin{proof}
Let us prove the statements one-by-one.
\begin{enumerate}[label=\textup{(\arabic*)}]
\item We have the commutative diagram:
$$\xymatrix{
    & \Bl_{P} Q \ar[ld]^{\widetilde{\pi}_{P}}\ar[d] \ar@{-->}[r]^{\widetilde{g}} & \Bl_{P}Q \ar[d] \ar[rd]^{\widetilde{\pi}_{P}} & \\
    \bP^{2} \ar@{-->}@/^1pc/[ru]^{\widetilde{\pi}_{P}^{-1}} \ar@{-->}[r]_-{\pi_P^{-1}} & Q \ar@{-->}[r]_-{g} &
    Q \ar@{-->}[r]_-{\pi_P} & \bP^2.
}$$
Note that this diagram factorizes $f=\pi_{P}\circ g\circ\pi_{P}^{-1}$ as $f=\widetilde{\pi}_{P}\circ\widetilde{g}\circ \widetilde{\pi}_{P}^{-1}$. The two lines passing through $P$ in $Q$ become disjoint $(-1)$-curves on $\Bl_PQ$ that are Galois conjugate to each other, and the morphism $\widetilde{\pi}_P$ is the blow-down of these two lines. Hence $\widetilde{\pi}_P$ and $\widetilde{\pi}_{P}^{-1}$ are both defined at all $\bF_q$-points.

It suffices to show that $\widetilde{g}$ induces a bijection on the $\bF_q$-points of $\Bl_PQ$. Indeed, $g$ induces a bijection on $Q(\bF_q)$ and fixes $P$. Hence $\widetilde{g}$ induces a birational self-map, and thus an automorphism, on the exceptional curve over $P$. As a result, $f$ is defined at all $\bF_q$-points of $\bP^2$. By symmetry, the same argument applies to $f^{-1}$, and hence $f\in\BCr_2(\bF_q)$.

\item Let $A\in \bP^2(\bF_q)\setminus C(\bF_q)$. Then $\pi^{-1}_{P}(A)\in Q\setminus C_0$, which implies $g(\pi_{P}^{-1}(A))=\pi_{P}^{-1}(A)$. Hence $f(A)=\pi_{P}\circ g\circ \pi_{P}^{-1}(A) = A$.

\item This follows from the relation $f = \pi_{P}\circ g\circ \pi_{P}^{-1}$ and the fact that $g$ permutes the points of $C_0(\bF_q)$ as a $(q+1)$-cycle.

\item Take an $\bF_q$-point $B$ on $C$ and consider the tangent line $\ell\colonequals T_BC\subset\bP^2$. Then $\ell\cap C=\{B\}$. The map $f$ acts as the identity on all the $\bF_q$-points of $\ell$ except for $B$, and sends $B$ to another point on $C$ not lying on $\ell$. Consequently, the map does not preserve collinearity.
\end{enumerate}
\end{proof}

Theorem~\ref{thm:realization} in the case of odd $q$ is then a consequence of Theorem~\ref{thm:groupTheory} and Proposition~\ref{prop:inducedmapP2}.

\begin{rmk}\label{rmk:lying-on-z=0}
Recall that $Q$ and $L$ are defined as
\begin{gather*}
	Q\colonequals\left\{
		x^2 - ty^2 + z^2 = w^2
	\right\}\subset\bP^{3},
	\quad
	L\colonequals\left\{
	    z = w = 0
	\right\}\subset\bP^{3},
\end{gather*}
and $P = [0:0:1:1]\in Q$. Projection from $P$ defines a birational map
$$
    \pi_P\colon Q\dashrightarrow \bP^2:
    [x:y:z:w]\mapsto
    [x:y:w-z]
$$
whose inverse is given by
$$
    \pi_P^{-1}\colon \bP^2 \dashrightarrow Q :
    [x:y:u] \mapsto [2ux : 2uy : x^2-ty^2-u^2 : x^2-ty^2+u^2].
$$
On the other hand, the map $g$ has the form
$$
    g\colon Q\dashrightarrow Q:
    [x:y:z:w]\mapsto
    [\alpha x + t\beta y:\beta x + \alpha y:
    \gamma z:\gamma w]
$$
where $\alpha$, $\beta$, $\gamma$ are homogeneous in $z$, $w$ and satisfy $\alpha^2 - t\beta^2 = \gamma^2$. These expressions allow one to compute $f = \pi_{P}\circ g\circ \pi_{P}^{-1}$ explicitly. Also recall that the smooth fiber $C_0 = \pi_L^{-1}([0:1])$ lies on $H\colonequals\{z=0\}$. To compute the action of $f$ on $C = \pi_P(C_0)$, one may identify $H$ with the codomain $\bP^2$ of $\pi_P$ via $[x:y:u]\mapsto [x:y:0:u]$.
\end{rmk}

%-----The construction in characteristic 2
\subsubsection{The construction in characteristic $2$}
\label{subsubsect:constrInChar2}

We first explain the construction over $\mathbb{F}_2$. Consider the quadric surface given by
$$
	Q\colonequals\left\{
		x^2+xy+y^2+z^2+x(z+w)+y(z+w)+zw=0
	\right\}\subset\bP^{3},
$$
As before, let $L\colonequals\{z=w=0\}\subset \bP^3$. We consider the projection $\bP^3\dashrightarrow\bP^1$ given by $[x:y:z:w]\mapsto [z:w]$. Restricting the map  to $Q$, we get a conic bundle $\pi_L: Q\to\bP^1$. We analyze the conics on the three $\bF_2$-fibers:
\begin{align*}
    C_0\colonequals \pi_L^{-1}([0:1]) &\cong \{[x:y:u]:  x^2+xy+y^2+xu+yu=0\} \\
    C_1\colonequals \pi_L^{-1}([1:0]) &\cong
    \{[x:y:u]: 
    x^2+xy+y^2+z^2+xz+yz=0\} \\
    C_2\colonequals \pi_L^{-1}([1:1])&\cong 
    \{[x:y:u]: x^2+xy+y^2=0]\}
\end{align*}
where we used the identification $H=\{z=0\}\cong \bP^2$ with homogeneous coordinates $x, y$ and $u$ mentioned in Remark~\ref{rmk:lying-on-z=0}.

One can check that $C_0$ is smooth, while $C_1$ and $C_2$ are both union of two $\bF_4$-lines meeting at a single $\bF_2$-point. In fact, 
\begin{align*} 
C_{0}(\bF_2)&= \{ [0:0:1], [1:0:1], [0:1:1]\}  \\
C_{1}(\bF_2)&=\{ [1:1:1] \} \\
C_{2}(\bF_2)&=\{ [0:0:1] \}
\end{align*}
Consider the map $g:\bP^3\to \bP^3$, given by $[x: y: z: w] \mapsto [y: x: z: w]$. By the symmetry of the defining equation, the quadric $Q$ is preserved under $g$. It is also evident that $g$  acts as a single transposition on $C_{0}(\bF_2)$, and trivially on both $C_{1}(\bF_2)$ and $C_{2}(\bF_2)$. Using the same argument given in Proposition~\ref{prop:inducedmapP2}, we see that the induced map $f=\pi_{P}\circ g\circ\pi_{P}^{-1}$ is an element of $\BCr_2(\bF_2)$. Furthermore, the induced permutation $f:\bP^2(\bF_2)\to\bP^2(\bF_2)$ is odd, as it transitively permutes the three points of $C_0(\bF_q)$. It also does not preserve collineation for the same reason explained in Proposition~\ref{prop:inducedmapP2}~\ref{prop:inducedmapP2:colli}. By Theorem~\ref{thm:groupTheory}, $\sigma_{2}(\BCr_2(\bF_2))=\Sym(\bP^2(\bF_2))$.

For $q=2^m\geq 4$, following Cantat, we use the quadric
$$
    Q\colonequals\{
        x^2 + r xy + s y^2 + z^2 + x(z+w) + y(z+w)+zw=0
    \}
$$
where $r, s\in\bF_q$ are chosen so that the polynomial $X^2+rX+s=0$ has no roots in the field $\bF_q$. The map $g:\bP^3\to \bP^3$, given by $[x: y: z: w] \mapsto [y: x: z: w]$ preserves the quadric. It can be checked that the fiber $C_0\colonequals \pi_L^{-1}([0:1])$ is a smooth conic. Using the same argument in Proposition~\ref{prop:inducedmapP2}, we see that the induced map $f=\pi_{P}\circ g\circ\pi_{P}^{-1}$ is an element of $\BCr_2(\bF_q)$. Moreover, the induced permutation $f:\bP^2(\bF_q)\to\bP^2(\bF_q)$ does not preserve collineation by the same argument given in Proposition~\ref{prop:inducedmapP2}~\ref{prop:inducedmapP2:colli} that involves looking at the tangent line: $f$ fixes all the $\bF_q$-points on the tangent line $T_P C$ except for $P$, while $P$ is sent by $f$ to another $\bF_q$-point away from $T_P C$. By Theorem~\ref{thm:groupTheory}, we deduce that $\sigma_{q}(\BCr_2(\bF_q))\supset \Alt(\bP^2(\bF_q))$.

%--------------------Birational invariance of parity
\section{Birational invariance of parity}
\label{sect:bir-inv-parity}

In this section, we prove that automorphisms of $\bP^n$ for $n\geq 1$ over $\bF_q$, where $q=2^m\geq 4$, induce only even permutations on the set of rational points. This result allows us to study the parity problem without specifying a coordinate system on $\bP^n$. Then we prove Theorem~\ref{thm:parityBirMapIntro}, namely, the invariance of parity under conjugations by birational maps. The proof of this theorem is built on the fact that one can resolve a birational map between surfaces over a perfect field via a sequence of blow-ups at closed points.

\begin{eg}
It is easy to construct a counterexample to Theorem~\ref{thm:parityBirMapIntro} for odd $q$ and $q = 2$. Consider an element $g\in\PGL_3(\bF_q)$ of the form
$$
    g = \begin{pmatrix}
        a & b & 0\\
        c & d & 0\\
        0 & 0 & 1
    \end{pmatrix}.
$$
Note that $g$ fixes $p = [0:0:1]$. Let $X$ be the blow-up of $\bP^2$ at $p$. Then $g$ lifts to an automorphism on $X$ which acts on the exceptional $\bP^1$ as
$
    \begin{pmatrix}
        a & b\\
        c & d
    \end{pmatrix},
$
and the parity is altered via the lifting if this matrix acts as an odd permutation on $\bP^1(\bF_q)$.
For example, one can choose
$
    \begin{pmatrix}
        1 & 0\\
        0 & \alpha
    \end{pmatrix}
$
if $q$ is odd, where $\alpha$ is a generator for the multiplicative group $\bF_q^\ast$, and choose
$
    \begin{pmatrix}
        1 & 1\\
        0 & 1
    \end{pmatrix}
$
if $q = 2$.
\end{eg}

%----------Parities induced by linear transformations
\subsection{Parities induced by linear transformations}
\label{subsect:linearTransform}

According to Waterhouse \cite{Wat89}, the group $\mathrm{GL}_{n+1}(\bF_q)$ is generated by two elements $A_n$ and $B_n$ for all $q$ and $n\geq 1$, which clearly descend to generators for $\PGL_{n+1}(\bF_q)$. Therefore, to prove that $\PGL_{n+1}(\bF_q)\subset\Alt(\bP^n(\bF_q))$, it is sufficient to verify that $A_n$ and $B_n$ induce even permutations.

The general formulas for $A_n$ and $B_n$ depend on whether $n = 1$ or $n\geq 2$. Let us denote by $I_{n+1}$ the identity matrix of size $n+1$, and $E_{i,j}$ the square matrix of size $n+1$ with $1$ at the $(i,j)$-th entry and zeros elsewhere. In the case $n\geq2$, we can choose a generator $\alpha$ for the multiplicative group $\bF_{q}^{\ast}$,
and let
\begin{align*}
    A_n = I_{n+1} + (\alpha-1)E_{2,2} + E_{n+1, 1},
    \qquad
    B_n = E_{1,2}+E_{2,3}+\cdots +E_{n+1, 1}.
\end{align*}
For example, when $n = 2$ we get
$$
    A_2 = \begin{pmatrix}
    1 & 0 & 0\\
    0 & \alpha & 0\\
    1 & 0 & 1
    \end{pmatrix},
    \qquad
    B_2 = \begin{pmatrix}
    0 & 1 & 0\\
    0 & 0 & 1\\
    1 & 0 & 0
    \end{pmatrix}.
$$
In the case $n = 1$ and $q > 2$, we choose a generator $\beta$ for the multiplicative group $\bF_{q^2}^{\ast}$, and define
$$
    \alpha\colonequals\beta^{q+1},\quad
    s\colonequals\mathrm{Tr}(\beta)=\beta+\beta^{q},\quad
    r\colonequals-\mathrm{Norm}(\beta)=-\beta^{q+1}.
$$
Then we let
$A_1 = \begin{pmatrix}
    0 & r \\
    1 & s 
\end{pmatrix}$
and
$B_1 = \begin{pmatrix}
    \alpha & 0 \\
    0 & 1 
\end{pmatrix}$.
We emphasize that the case $n = 1$ and $q = 2$ is not covered by these formulas. In this last case, $\GL_2(\bF_2)$ is generated by
$\begin{pmatrix}
    0 & 1\\
    1 & 1
\end{pmatrix}$
and
$\begin{pmatrix}
    1 & 1\\
    0 & 1
\end{pmatrix}$
which act respectively as a $3$-cycle and a $2$-cycle on $\bP^1(\bF_2)$.

%Parity of A1 & B1
\begin{lemma}
\label{lemma:parityA1B1}
Both $A_1$ and $B_1$ induce even permutations on $\bP^1(\bF_q)$ where $q = 2^m\geq4$.
\end{lemma}

\begin{proof}
The element $\alpha$ is a generator for $\bF_q^*\cong\bZ/(q-1)\bZ$,
so $B_1$ fixes $[1:0]$ and $[0:1]$ and acts as a $(q-1)$-cycle on
$
    \bP^1(\bF_q)\setminus\{[1:0]\cup[0:1]\}\cong\bF_q^*,
$
which is even for all $q = 2^m\geq2$.
On the other hand, $A_1$ can be factorized as
$$\begin{pmatrix}
    0 & r \\
    1 & s
\end{pmatrix}
=
\begin{pmatrix}
    r & 0 \\
    0 & 1
\end{pmatrix}
\begin{pmatrix}
    1 & 0 \\
    s & 1
\end{pmatrix}
\begin{pmatrix}
    0 & 1 \\
    1 & 0
\end{pmatrix}
\equalscolon A_{11}A_{12}A_{13}.
$$
Among the factors:
\begin{itemize}
    \item $A_{11}$ has the same parity as $B_1$ since $r = -\alpha = \alpha$.
    \item $A_{12}$ fixes $[0:1]$ and acts on $\bP^1(\bF_q)\setminus\{[0:1]\}\cong\bF_q$ as a translation by $s$, which is a composition of $q/2$ transpositions (because $\operatorname{char}(k)=2$) and thus even for $q = 2^m\geq4$.
    \item $A_{13}$ is an involution fixing $[1:1]$, so it is a composition of $q/2$ transpositions which is even for $q = 2^m\geq4$.
\end{itemize}
As a result, $A_1$ acts as a compositions of three even permutations, so $A_1$ is even.
\end{proof}

%Parity of An
\begin{lemma}
\label{lemma:parityAn}
Assume $n\geq 2$. Then $A_n$ induces an even permutation on $\bP^n(\bF_q)$ for $q = 2^m\geq2$.
\end{lemma}
\begin{proof}
One can verify directly that $A_n = T_nM_n$, where
$$
    T_n\colonequals\begin{pmatrix}
    1 & 0 & \cdots & 0\\
    0 & 1 & \cdots & 0\\
    \vdots & \vdots & \ddots & \vdots\\
    1 & 0 & \cdots & 1
    \end{pmatrix},
    \qquad
    M_n\colonequals\begin{pmatrix}
    1 & 0 & \cdots & 0\\
    0 & \alpha & \cdots & 0\\
    \vdots & \vdots & \ddots & \vdots\\
    0 & 0 & \cdots & 1
    \end{pmatrix}.
$$
Note that $M_n$ has an odd order $q-1$, so its action is even. Therefore, it is sufficient to prove that $T_n$ induces an even permutation. First, by writing $T_n = I_{n+1} + E_{n+1,1}$, we obtain
$$
    T_n^2 = I_{n+1} + 2E_{n+1,1} + E_{n+1,1}^2
    = I_{n+1},
$$
so $T_n$ is an involution, thus its action on $\bP^n(\bF_q)$ is a product of disjoint transpositions. Second, $T_n$ acts on the homogeneous coordinates as
$$
    [x_0 : x_1 : ... : x_n]
    \mapsto
    [x_0 : ... : x_{n-1} : x_0+x_{n}],
$$
so its fixed locus is the hyperplane $\{x_0=0\}$. Hence the number of transpositions in $T_n$ equals
$$
    \frac{1}{2}
    \left(
        |\bP^{n}(\bF_q)|- |\bP^{n-1}(\bF_q)|
    \right)
    = \frac{1}{2}
    \left(
        \frac{q^{n+1} - 1}{q-1} - \frac{q^{n} - 1}{q-1}
    \right)
    = \frac{1}{2}
    \left(
        \frac{q^{n+1} - q^n}{q-1}
    \right)
    = \frac{q^{n}}{2},
$$
which is even for $n\geq 2$ and $q=2^m\geq2$.
\end{proof}

%Parity of Bn
\begin{lemma}
\label{lemma:parityBn}
Assume $n\geq 2$. If $q = 2^m\geq4$, then the action of $B_n$ on $\bP^n(\bF_q)$ is even. If $q = 2$, then the action is odd when $n=2^\ell-1$ for some $\ell$ and even otherwise.
\end{lemma}

\begin{proof}
We choose a generator $b\in\Gal(\bF_{q^{n+1}}/\bF_q)\cong\bZ/(n+1)\bZ$ and an element $\theta\in\bF_{q^{n+1}}$ such that
$$
    \{\theta_i\colonequals b^i(\theta) : i = 0,\dots,n\}\subset\bF_{q^{n+1}}
$$
form a normal basis over $\bF_q$.
This identifies the underlying affine space $\bF_q^{n+1}$ of $\bP^n$ as
$$
    \bF_q\theta_0\oplus\bF_q\theta_1\oplus\cdots\oplus\bF_q\theta_n
    \;\cong\;\bF_{q^{n+1}}
$$
where a point $(x_0,\dots,x_n)\in\bF_q^{n+1}$ corresponds to $x_0\theta_0 + \cdots + x_n\theta_n\in\bF_{q^{n+1}}$.
Since $b(\theta_i) = \theta_{i+1}$ for $i = 0,\dots, n-1$ and $b(\theta_n) = \theta_{0}$, we have
$$
    b(x_0\theta_0 + x_1\theta_1 + \cdots + x_n\theta_n)
    = x_n\theta_0 + x_0\theta_1 + \cdots + x_{n-1}\theta_n,
$$
which identifies the multiplication of $B_n$ on $\bF_q^{n+1}$ from the left as the action of $b^{-1}$ on $\bF_{q^{n+1}}$. Therefore, it is sufficient to compute the parity of the action of $b$ on $\bF_{q^{n+1}}$.

Let us write $n+1 = u2^{\ell}$ where $u$ is odd. Then the parity of $b^u$ is the same as the parity of $b$ and the cycle decomposition of $b^u$ contains only $2^r$-cycles for $r\geq 0$. There is a filtration of $\bF_{q^{n+1}}$ invariant under the action of $b^u$:
$$
    \bF_{q^{u}}\subset\cdots
    \subset
    \bF_{q^{u2^{r-1}}}
    \subset
    \bF_{q^{u2^{r}}}
    \subset
    \cdots\subset\bF_{q^{u2^\ell}}
    = \bF_{q^{n+1}}.
$$
For each $1\leq r\leq \ell$, there are $q^{u2^r} - q^{u2^{r-1}}$ many elements in $\bF_{q^{u2^{r}}}\setminus\bF_{q^{u2^{r-1}}}$,
and the $b^u$-orbit of each element has size
$
    [\bF_{q^{u2^{r}}}:\bF_{q^{u}}] = 2^r.
$
Therefore, the number of $2^r$-cycles in $b^u$ equals
$$
    \frac{1}{2^r}|\bF_{q^{u2^{r}}}\setminus\bF_{q^{u2^{r-1}}}|
    = \frac{1}{2^r}(q^{u2^r} - q^{u2^{r-1}}).
$$
On the quotient space $\bP^n(\bF_q) = \bP(\bF_{q^{n+1}})$,
which we consider as the set of $\bF_q$-lines in $\bF_q^{n+1}$ through the origin, the number of $2^r$-cycles for the action of $b^u$ becomes
\begin{equation}
\label{eqn:number_2^r-cycles}
    \frac{1}{2^r}\left(
    \frac{q^{u2^r} - q^{u2^{r-1}}}{q-1}
    \right)
    = \frac{q^{u2^{r-1}}}{2^r}\left(
    \frac{q^{u2^{r-1}}-1}{q-1}
    \right).
\end{equation}
\begin{itemize}
\item Suppose $q = 2^m\geq4$. Then $m\geq 2$, thus $mu2^{r-1}-r>0$ for $u\geq 1$ and $1\leq r\leq \ell$. Hence
\begin{equation}
\label{eqn:a-factor}
    \frac{q^{u2^{r-1}}}{2^r}
    = \frac{2^{mu2^{r-1}}}{2^r}
    = 2^{mu2^{r-1} - r}
\end{equation}
is even. As the fraction $\frac{q^{u2^{r-1}}-1}{q-1}$ is clearly an integer, we conclude that the number of $2^r$-cycles in $b^u$ when acting on $\bP^n(\bF_q)$ is even for all $1\leq r\leq\ell$,
thus the action is even itself.

\item Suppose $q=2$ and $n=2^\ell-1$ for some $\ell$. Note that $\ell\geq2$ as $n\geq 2$. Then $m=1$ and $u=1$. In this case, \eqref{eqn:a-factor} equals $1$ for $r=1$ and is even for $2\leq r\leq\ell$. This implies that \eqref{eqn:number_2^r-cycles} equals $1$ for $r=1$ and is even for $2\leq r\leq\ell$. As a result, the action of $b^u$ on $\bP^n(\bF_q)$ consists of one $2$-cycle and an even number of $2^r$-cycle for each $2\leq r\leq\ell$, thus is an odd action.

\item Suppose $q=2$ and $n\neq 2^\ell-1$ for all $\ell$. Then $m=1$ and $u>1$. This implies that \eqref{eqn:a-factor} is even for $1\leq r\leq\ell$. We conclude that the action of $b^u$ is even as in the first case.
\end{itemize}
These cover all the cases, so the proof is done.
\end{proof}

\begin{prop}
\label{prop:PGLpermute}
For $n\geq 1$ and $q = 2^m\geq4$, the action of $\PGL_{n+1}(\bF_q)$ on $\bP^n(\bF_q)$ is even.
\end{prop}

\begin{proof}
The case $n=1$ (resp. $n\geq 2$) follows from Lemma~\ref{lemma:parityA1B1} (resp. Lemmas~\ref{lemma:parityAn} and \ref{lemma:parityBn}).
\end{proof}

The parity of a permutation is invariant upon raising to an odd power, so we usually assume the order of a permutation to be a power of $2$ when studying the parity. For a permutation induced by a linear transformation, the following result shows that we can say more about the cycle type if its order is a power of $2$.

\begin{cor}
\label{cor:powerTwoCycles}
Let $n\geq 1$ and $q = 2^m\geq4$.
Suppose that $\sigma\in\PGL_{n+1}(\bF_q)$ induces a permutation of order $2^r$ on $\bP^n(\bF_q)$. Define $c_i$, where $0\leq i\leq r$, to be the number of $2^i$-cycles in the cycle decomposition. Then $c_0$ is odd and the sum $c_1 + \dots + c_r$ is even. In the case $n = 1$, there are only two possibilities:
\begin{enumerate}[label=\textup{(\arabic*)}]
    \item $c_0 = q+1$ and $c_i = 0$ for all $1\leq i\leq r$, i.e., $\sigma$ is the identity.
    \item $c_0 = 1$ and $c_i = 0$ for all but one $1\leq i\leq r$. The unique nonzero $c_j$ where $1\leq j\leq r$ equals $q/2^j > 1$.
\end{enumerate}
\end{cor}

\begin{proof}
Because a $2^i$-cycle is odd for all $i\geq 1$, the sum $c_1+\dots+c_r$ must be even due to Proposition~\ref{prop:PGLpermute}.
Then the relations
$$
    |\bP^n(\bF_q)| = q^n + \dots + q + 1 = c_0 + 2c_1 + \dots + 2^rc_r
$$
imply that $c_0$ is odd. Assume $n = 1$ and that $\sigma$ is not the identity. Then $\sigma$ fixes at most $2$ points, which implies that $c_0 = 1$. Let $1\leq j\leq r$ be the smallest integer such that $c_j\neq 0$. Then $\sigma^{2^j}$ becomes the identity as it fixes $1 + 2^jc_j\geq 3$ points. It follows that every nontrivial cycle in $\sigma$ has the same size $2^j$. If $2^j = q$, then $\sigma$ is a $q$-cycle thus is odd, which is impossible by Proposition~\ref{prop:PGLpermute}. Hence $2^j < q$, and the equalities
$
    |\bP^1(\bF_q)| = q + 1 = 1 + 2^jc_j
$
imply that $c_j = q/2^j > 1$.
\end{proof}

%----------Projective bundles over finite sets
\subsection{Projective bundles over finite sets}
\label{subsect:finiteProjBundle}

We define a \emph{$\bP^n$-bundle over a finite set $B$} to be the disjoint union of projective $n$-spaces:
$$
    \cP = \coprod_{i\in B} P_i,
    \quad P_i\cong\bP^n,
    \quad\text{equipped with the map}\quad
    h\colon\cP\to B:P_i\mapsto i.
$$
Consider the set $\cP(k)$ of $k$-points on $\cP$. We are interested in elements $\sigma\in\Sym(\cP(k))$ of the form:
\begin{enumerate}[label=(\arabic*)]
    \item\label{proj_fiberToFiber}
    For every $i\in B$, there exists $j\in B$ such that $\sigma(P_i(k)) = P_j(k)$. Then $h\sigma h^{-1}$ is well-defined as an element of $\Sym(B)$.
    \item\label{proj_inducedByIsom}
    Each bijection $\sigma\colon P_i(k)\to P_j(k)$ is induced by a linear isomorphism over $k$.
\end{enumerate}
Note that such elements form a subgroup of $\Sym(\cP(k))$.

\begin{lemma}
\label{lemma:autoProjBundle}
Let $k=\bF_q$, $q = 2^m\geq 4$,
and $\sigma\in\Sym(\cP(k))$ be an element satisfying \ref{proj_fiberToFiber} and \ref{proj_inducedByIsom}.
Then $\sigma$ and $\sigma_B\colonequals h\sigma h^{-1}\in\Sym(B)$ have the same parity.
\end{lemma}

\begin{proof}
The parity of a permutation is invariant upon raising it to an odd power,
so we can assume that both $\sigma$ and $\sigma_B$ consist of disjoint cycles of sizes powers of $2$. Suppose that
$$
    O\colonequals\{p_1, ..., p_{r}\}\subset B,\quad r=2^\ell\geq1,
$$
is one of the orbits of $\sigma_B$. Then the set of $k$-points in $h^{-1}(O)\subset\cP$ is invariant under $\sigma$. Therefore, it suffices to prove the statement under the hypothesis $O = B$. Note that the case $r = 1$ follows immediately from Proposition~\ref{prop:PGLpermute}. Hence we assume that $r\geq 2$, in which case $\sigma_{B}$ is odd, and so our goal is to prove that $\sigma$ is also odd.

Fix an element $p\in O$. The assumption $O = B$ implies $\sigma_B^r = \mathrm{id}$, so $\sigma^r$ acts on the $k$-points of $h^{-1}(p)\cong\bP^n$. Denote this action as $\sigma^r_p$. Observe that, in the cycle decompositions, a $u$-cycle in $\sigma^r_p$ contributes a $(ur)$-cycle in $\sigma$, and every cycle in $\sigma$ is obtained this way. Assume that $\sigma^r_p$ consists of $c_i$ many $2^i$-cycles where $i\geq0$. Then $\sigma$ consists of $c_i$ many $(2^ir)$-cycles, which are all odd since the assumption $r\geq 2$ implies $2^i r\geq 2$. By Corollary~\ref{cor:powerTwoCycles} applied to $\sigma^{r}_{p}$, the sum $\sum_{i\geq0}c_i$, which also equals the number of cycles in $\sigma$, is an odd integer. We conclude that $\sigma$ is odd.
\end{proof}

%----------Proof of the birational invariance of parity
\subsection{Proof of the birational invariance of parity}
\label{subsect:proof-birinvparity}

Let $X$ and $Y$ be smooth surfaces over $k = \bF_{2^m}$ where $m\geq 2$. Given $\alpha\in\BBir_k(X)$, $\beta\in\BBir_k(Y)$, and a birational map $h\colon X\to Y$ over $k$ that satisfy $\alpha = h^{-1}\beta h$, we prove Theorem~\ref{thm:parityBirMapIntro}, namely, that the permutations induced by $\alpha$ and $\beta$ on $X(k)$ and $Y(k)$, respectively, have the same parity. Note that, if $h$ induces a bijection between $X(k)$ and $Y(k)$, then the relation $\alpha = h^{-1}\beta h$ implies immediately that the induced permutations have the same cycle type and thus the same parity. The main content of Theorem~\ref{thm:parityBirMapIntro} consists in that the same conclusion holds even if $h$ is not a bijection on the sets of rational points.

In the following, we establish Theorem~\ref{thm:parityBirMapIntro} from scratch, starting from the case when the birational map $h\colon X\to Y$ is a blow-up at a set of closed points, then the case when $h$ is a birational morphism, and finally the full generality.

\begin{lemma}
\label{lem:sigma_B}
Let $Y$ be a smooth surface over $k=\bF_{2^m}$, $m\geq 2$, and $h\colon X\to Y$ be a birational morphism over $k$ that blows up a set $\overline{B}\subset Y(\overline{k})$ of closed points. Define $B\colonequals\overline{B}\cap Y(k)$ and $E\colonequals h^{-1}(B)\subset X$. Pick $\beta\in\BBir_k(Y)$ and assume that $\alpha\colonequals h^{-1}\beta h\in\BBir_k(X)$. Then we have $\alpha(E)=E$ and $\beta(B)=B$.
\end{lemma}

\begin{proof}
The map $\alpha$ does not contract any curve in $E$. Indeed, every irreducible component of $E$ is a rational curve over $k$, thus contains more than one $k$-points. If $\alpha$ contracts any of them, we would have $\alpha\notin\BBir_k(X)$, contradiction. It follows that $\alpha(E) = h^{-1}\beta h(E) = h^{-1}\beta(B)$ is a curve, so $\beta(B)\subset B$. Since $\beta$ induces a bijection on $Y(k)$, we have $\beta(B)=B$, and hence $\alpha(E)=E$.
\end{proof}

%Case 1
\begin{lemma}
\label{lem:parityBirMap_case1}
Retain the setting from Lemma~\ref{lem:sigma_B}. Then the actions of $\alpha$ on $X(k)$ and $\beta$ on $Y(k)$ have the same parity.
\end{lemma}

\begin{proof}
Let $U\colonequals X\setminus E$ and $V\colonequals Y\setminus B$. Note that $h|_U\colon U\to V$, though may not be an isomorphism, induces a bijection on the sets of $k$-points. By Lemma~\ref{lem:sigma_B}, we have $\alpha(U)=U$ and $\beta(V)=V$, and the relation $\alpha = h^{-1}\beta h$ implies $\alpha|_U = (h|_U)^{-1}(\beta|_V) (h|_U)$. Hence the restrictions of $\alpha$ to $U$ and $\beta$ to $V$ have the same parity when acting on the $k$-points.

Now consider the actions of $\alpha$ on $E$ and $\beta$ on $B$. Note that $E$ is a $\bP^1$-bundle over $B$. Restricting $h$ to $E(k)$ induces the map among finite sets
$$\xymatrix{
    E(k)\cong\bP^1(k)\times B(k)\ar[rr]^-{h|_{E(k)}} && B(k),
}$$
as well as the relation
$
    \beta|_{B(k)}
    = h|_{E(k)}\circ
    \alpha|_{E(k)}\circ
    (h|_{E(k)})^{-1}.
$
Then the permutations $\beta|_{B(k)}$ and $\alpha|_{E(k)}$ have the same parity by Lemma~\ref{lemma:autoProjBundle}. This completes the proof.
\end{proof}

The following two lemmas will be needed in the proofs of the remaining cases.

\begin{lemma}
\label{lem:blowup-at-except}
Let $X$ and $Y$ be smooth surfaces over a perfect field $k$ and $h\colon X\rightarrow Y$ a birational morphism over $k$. Then we can factorize $h$ as a sequence of blow-ups at closed points
$$\xymatrix{
    h\colon X = Y_r\ar[r]^-{\epsilon_r}
    & Y_{r-1}\ar[r]^-{\epsilon_{r-1}}
    & \cdots\ar[r]^-{\epsilon_2}
    & Y_1\ar[r]^-{\epsilon_1}
    & Y.
}$$
Moreover, this sequence can be arranged in the way that the points in $Y_i$ blown up by $\epsilon_{i+1}$ lie in the exceptional locus of $\epsilon_i$.
\end{lemma}

\begin{proof}
According to \cite{Man86}*{Lemma~18.1.3}, we can factorize $h$ as a sequence of blow-ups at closed points. To prove the second statement, assume that there exists a point $x\in Y_i$ blown up by $\epsilon_{i+1}$ but not in the exceptional locus of $\epsilon_i$. Consider the commutative diagram
$$\xymatrix{
    Y_{i+1}'\ar[r]^-{\epsilon_{i+1}'}\ar[d]^-\sim
    & Y_{i}'\ar[r]^-{\epsilon_{i}'}\ar[d]^-{\mathrm{Bl}_x}
    & Y_{i-1}\ar@{=}[d]\\
    Y_{i+1}\ar[r]^-{\epsilon_{i+1}}
    & Y_{i}\ar[r]^-{\epsilon_{i}}
    & Y_{i-1}
}$$
where $\epsilon_{i}'$ is $\epsilon_{i}$ followed by the blow-up at $x$, and $\epsilon_{i+1}'$ blows up the same points as $\epsilon_{i+1}$ except for $x$. Then $Y_{i+1}'$ and $Y_{i+1}$ are canonically isomorphic and we can replace $\epsilon_{i}\epsilon_{i+1}$ by $\epsilon_{i}'\epsilon_{i+1}'$. Repeating this process from $i = r-1$ to $i = 1$ gives us the desired sequence. 
\end{proof}

\begin{lemma}
\label{lem:regular-on-except}
Let $Y$ be a surface, $\beta$ be a birational self-map on $Y$, and $q\in Y$ be a closed point at which $\beta$ is well-defined. Let $\epsilon\colon Y'\to Y$ be the blow-up at the set $\{q,\beta(q)\}$, and $E_q$ be the exceptional divisor over $q$. Then the composition $\epsilon^{-1}\beta\epsilon$, which is a birational self-map on $Y'$, is well-defined everywhere on $E_q$.
\end{lemma}

\begin{proof}
Let $q'\colonequals\beta(q)$ and $E_{q'}\subset Y'$ be the exceptional divisor over $q'$. Denote $\beta'\colonequals\epsilon^{-1}\beta\epsilon$. Then we have the commutative diagram
$$\xymatrix{
    E_q\ar@{^(->}[r]
    & Y'\ar[r]^-\epsilon\ar@{-->}[d]^-{\beta'}
    & Y\ar@{-->}[d]^-\beta\\
    E_{q'}\ar[r]
    & Y'\ar[r]^-\epsilon
    & Y.
}$$
The composition $\beta\epsilon\colon Y'\dashrightarrow Y$ pulls $q'$ back as the divisor $E_q$ while $q'$ is blown up by $\epsilon$ as $E_{q'}$. By the universal property of blowing up, $\beta\epsilon$ factors through the bottom $\epsilon$ uniquely as
$$\xymatrix{
    E_q\ar@{^(->}[r]\ar[d]_-{\beta_i''|_{E_q}}^-\sim
    & Y'\ar[r]^-\epsilon\ar@{-->}[d]^-{\exists\;\beta''}
    & Y\ar@{-->}[d]^\beta\\
    E_{q'}\ar[r]
    & Y'\ar[r]^-\epsilon
    & Y.
}$$
Note that $\beta''$ is well-defined everywhere on $E_q$ and $\beta'' = \epsilon^{-1}\beta\epsilon = \beta'$. Hence $\beta'$ is well-defined everywhere on $E_q$.
\end{proof}

%Case 2
Now we prove the invariance of parity under conjugations by birational morphisms.

\begin{lemma}
\label{lem:parityBirMap_case2}
Let $X$ and $Y$ be smooth surfaces over $k=\bF_{2^m}$, $m\geq 2$, and $h\colon X\rightarrow Y$ a birational morphism over $k$. Pick $\beta\in\BBir_k(Y)$ and assume that $\alpha\colonequals h^{-1}\beta h\in\BBir_k(X)$. Then the actions of $\alpha$ on $X(k)$ and $\beta$ on $Y(k)$ have the same parity.
\end{lemma}

\begin{proof}
By Lemma~\ref{lem:blowup-at-except}, we can factorize $h$ as
$$\xymatrix{
    h\colon X = Y_r\ar[r]^-{\epsilon_r}
    & Y_{r-1}\ar[r]^-{\epsilon_{r-1}}
    & \cdots\ar[r]^-{\epsilon_2}
    & Y_1\ar[r]^-{\epsilon_1}
    & Y
}$$
such that the points in $Y_i$ blown up by $\epsilon_{i+1}$ lie in the exceptional locus of $\epsilon_i$. Denote $\beta_0\colonequals\beta$ and define inductively that
\begin{equation}
\label{eqn:inductiveConjugate}
    \beta_i\colonequals
    \epsilon_i^{-1}\beta_{i-1}\epsilon_i
    \in\Bir_k(Y_i),
    \quad i=1,\dots,r.
\end{equation}
Note that $\beta_r = \alpha$. Let us prove that every $\beta_i\in\BBir_k(Y_i)$ by induction. The case $i=0$ follows by definition. Suppose that $\beta_{i-1}\in\BBir_k(Y_{i-1})$ and, to the contrary, that $\beta_i\notin\BBir_k(Y_i)$. Let $p\in Y_i(k)$ be a base-point of $\beta_i$. Consider the two points
$$
    q\colonequals\epsilon_i(p)\in Y_{i-1}(k),
    \qquad
    q'\colonequals\beta_{i-1}(q)=\beta_{i-1}\epsilon_i(p)\in Y_{i-1}(k).
$$
There are three possible situations:
\begin{enumerate}[label=(\arabic*)]
\item $q'$ is not blown up by $\epsilon_{i}$. This implies that $\beta_i$ is well-defined at $p$ due to (\ref{eqn:inductiveConjugate}), which contradicts our assumption.
    
\item $q'$ is blown up by $\epsilon_{i}$ while $q$ is not. Let $E_{q'}\subset Y_i$ denote the exceptional divisor over $q'$. In this case, $p$ does not lie in the exceptional locus of $\epsilon_i$, so it is mapped bijectively to a point $\widetilde{p}\in X(k)$ via $(\epsilon_{i+1}\cdots\epsilon_r)^{-1}$. Relations~(\ref{eqn:inductiveConjugate}) imply that $\alpha^{-1}$ contracts the proper transform of $E_{q'}$ to $\widetilde{p}$, so $\widetilde{p}$ is a base-point of $\alpha$, which contradicts the fact that $\alpha\in\BBir_k(X)$.
    
\item $q'$ and $q$ are both blown up by $\epsilon_{i}$. By Lemma~\ref{lem:regular-on-except}, the map $\beta_i$ is well-defined everywhere on the exceptional divisor $E_q\subset Y_i$ over $q$. Since $p\in\epsilon_i^{-1}(q)=E_q$, we conclude that $\beta_i$ is well-defined at $p$, contradiction.
\end{enumerate}
Since we get contradictions in all possible cases, we conclude that $\beta_i\in\BBir_k(Y_i)$, hence the claim is fulfilled by induction.
By Lemma~\ref{lem:parityBirMap_case1}, the permutations induced by $\beta_{i}$ for all $i$, including $\alpha$ and $\beta$, have the same parity.
\end{proof}

Before entering the proof of the general case, let us introduce a method about resolving a birational self-map as a birational permutation. Let $X'$ be a smooth surface over a finite field $k$ and $\epsilon\colon X\to X'$ be a birational morphism over $k$ that blows up a set $C\subset X'$ of closed points with exceptional locus $E\subset X$. Pick $\alpha'\in\BBir_k(X')$ and define $\alpha\colonequals\epsilon^{-1}\alpha'\epsilon$. Note that $\alpha$ belongs to $\Bir_k(X)$ but may not belong to $\BBir_k(X)$ in general.

\begin{lemma}
\label{lem:bijectivization}
Retain the notation above. Let $O_1,\dots,O_n\subset X'(k)$ be the orbits of $\alpha'$ that meet the center $C$ nontrivially. Note that the preimages of $O_j\setminus C$ in $X$ make up the subset
$$
    B\colonequals
    \bigcup_{j=1}^n\epsilon^{-1}(O_j\setminus C)
    \;\subset\; X(k)\setminus E.
$$
Consider the blow-up $\eta\colon Z\colonequals\Bl_BX\longrightarrow X$. Then the composition $\eta^{-1}\alpha\eta$ belongs to $\BBir_k(Z)$.
\end{lemma}

\begin{proof}
Let $O\subset X'(k)$ be any of the orbits of $\alpha'$. Note that, if $O\cap C=\emptyset$, then $\alpha$ is well-defined on the subset $(\epsilon^{-1}(O))(k)\subset X(k)$. Assume $O\cap C\neq\emptyset$. Then there are two possibilities:
\begin{enumerate}[label=(\alph*)]
\item If $O\subset C$, then one can show that $\alpha$ is well-defined on $(\epsilon^{-1}(O))(k)\subset X(k)$ by applying Lemma~\ref{lem:regular-on-except} possibly a multiple of times.

\item If $O\not\subset C$, then there exists $q\in O\setminus C$ such that $\alpha'(q)\in C$ and
$$
    O\setminus C
    = \{q, \alpha'^{-1}(q),\dots,\alpha'^{-\ell}(q)\}
    \quad\text{for some}\quad
    \ell\geq 0.
$$
Note that $O\setminus C$ is a finite set as we are working over a finite field. In this case, $\alpha$ is undefined at $\epsilon_1^{-1}(q)$. Blowing up $\epsilon_1^{-1}(q)$ will resolve this indeterminacy by Lemma~\ref{lem:regular-on-except}, though this will create a new base-point at $\epsilon^{-1}(\alpha'^{-1}(q))$. By blowing up this point and then $\epsilon^{-1}(\alpha'^{-2}(q)),\dots,\epsilon^{-1}(\alpha'^{-\ell}(q))$ subsequently, the base-points in $\epsilon^{-1}O$ will all be resolved.
\end{enumerate}
By applying the above to $O_1,\dots,O_n$, we conclude that $\eta^{-1}\alpha\eta\in\BBir_k(Z)$.
\end{proof}

%Case 3
\begin{proof}[Proof of Theorem~\ref{thm:parityBirMapIntro}]
We can eliminate the indeterminacy locus of $h$ by a sequence of blow-ups at closed points \cite{Kol07}*{Corollary~1.76}
$$\xymatrix{
    X_r\ar[drrrr]_-{\widetilde{h}}\ar[r]^-{\epsilon_r}
    & X_{r-1}\ar[r]^{\epsilon_{r-1}}
    & \cdots\ar[r]^-{\epsilon_2}
    & X_1\ar[r]^-{\epsilon_1}
    & X_0 = X\ar@<-12pt>@{-->}[d]^-{h}\\
    &&&& Y\qquad
}$$
For each $\epsilon_i$ where $1\leq i\leq r$, let $E_i\subset X_i$ be its exceptional locus and
$
    C_{i-1}\colonequals\epsilon_i(E_i)\subset X_{i-1}
$
be its center. We also define $C_r\colonequals\emptyset$. By Lemma~\ref{lem:blowup-at-except}, we can assume $C_i\subset E_i$ for $i = 1,\dots,r-1$. Let $\alpha_0\colonequals\alpha$ and define inductively that
\begin{equation}
\label{eqn:inductiveConjugate_case3}
    \alpha_i\colonequals
    \epsilon_i^{-1}\alpha_{i-1}\epsilon_i
    \in\Bir_k(X_i),
    \quad i=1,\dots,r.
\end{equation}

Let us prove by induction on $i$ that there exists a birational morphism
\begin{equation}
\label{eqn:regularization}
    \eta_i\colon Z_i\longrightarrow X_i
    \quad\text{such that}\quad
    \begin{cases}
        \tau_i\colonequals\eta_i^{-1}\alpha_i\eta_i\in\BBir_k(Z_i)\\
        \text{its center }B_i\subset X_i\text{ is disjoint from }C_i.
    \end{cases}
\end{equation}
For the initial case $i = 1$, consider the action of $\alpha_0$ on $X(k)$ and let $O_1,\dots,O_n\subset X(k)$ be the orbits that meet $C_0$ nontrivially. Define
$$
    B_1\colonequals
    \bigcup_{j=1}^n\epsilon_1^{-1}(O_j\setminus C_0)
    \;\subset\; X_1(k)\setminus E_1,
$$
and consider the blow-up
$$
    \eta_1\colon Z_1\colonequals\Bl_{B_1}X_1\longrightarrow X_1.
$$
Then $\tau_1\colonequals\eta_1^{-1}\alpha_1\eta_1\in\BBir_k(Z_1)$ by Lemma~\ref{lem:bijectivization}. Moreover, $B_1$ is disjoint from $E_1$ by construction. As $C_1\subset E_1$, we conclude that $B_1\cap C_1 = \emptyset$. This completes the initial step.

Assume that there exists an $\eta_i$ as in (\ref{eqn:regularization}) for some $1\leq i\leq r-1$. Consider the fiber diagram
$$\xymatrix{
    & X_{i+1}'\colonequals X_{i+1}\times_{X_{i}}Z_{i}\ar[d]_-{\pi_1}\ar[r]^-{\pi_2} & Z_i\ar[d]^-{\eta_i} &\\
    \cdots\ar[r] & X_{i+1}\ar[r]^-{\epsilon_{i+1}} & X_i\ar[r]^-{\epsilon_i} & \cdots
}$$
where $\pi_1$ and $\pi_2$ are the projections to the components of $X_{i+1}'$. Note that $X_{i+1}'$ is the blow-up of $X_i$ at the disjoint union $B_i\cup C_i$, and we can identify $\pi_2$ as the blow-up
$$
    \pi_2\colon X_{i+1}'\cong\Bl_{\eta_i^{-1}C_i}Z_i\longrightarrow Z_i.
$$
By hypothesis, we have $\tau_i = \eta_i^{-1}\alpha_i\eta_i\in\BBir_k(Z_i)$,
which can be lifted to $X_{i+1}'$ as
$$
    \alpha_{i+1}'\colonequals\pi_2^{-1}\tau_i\pi_2\in\Bir_k(X_{i+1}').
$$
By tracking the fiber diagram above, we obtain
\begin{equation}
\label{eqn:sigmaSigma'}
    \alpha_{i+1}'
    = \pi_2^{-1}\tau_i\pi_2
    = \pi_2^{-1}\eta_i^{-1}
    \alpha_i
    \eta_i\pi_2\\
    = \pi_1^{-1}\epsilon_{i+1}^{-1}
    \alpha_i
    \epsilon_{i+1}\pi_1
    = \pi_1^{-1}\alpha_{i+1}\pi_1.
\end{equation}
Let $O_1,\dots,O_n\subset Z_i(k)$ be the orbits of the action of $\tau_i$ that satisfy $O_j\cap\eta_i^{-1}C_i\neq\emptyset$. Define
$$
    B_{i+1}'\colonequals
    \bigcup_{j=1}^n\pi_2^{-1}(O_j\setminus\eta_i^{-1}C_i)
    \;\subset\; X_{i+1}'(k)
$$
and consider the blow-up
$$
    \eta_{i+1}'\colon Z_{i+1}\colonequals\Bl_{B_{i+1}'}X_{i+1}'\longrightarrow X_{i+1}'.
$$
Then
$
    \tau_{i+1}\colonequals
    \eta_{i+1}'^{-1}
    \alpha_{i+1}'
    \eta_{i+1}'
    \in\BBir_k(Z_{i+1})
$
by Lemma~\ref{lem:bijectivization}. Define
$$
    \eta_{i+1}\colonequals\pi_1\eta_{i+1}'
    \colon Z_{i+1}\longrightarrow X_{i+1}.
$$
Using (\ref{eqn:sigmaSigma'}), we obtain
$$
    \tau_{i+1}
    = \eta_{i+1}'^{-1}\alpha_{i+1}'\eta_{i+1}'
    = \eta_{i+1}'^{-1}\pi_1^{-1}\alpha_{i+1}\pi_1\eta_{i+1}'
    = \eta_{i+1}^{-1}\alpha_{i+1}\eta_{i+1}.
$$
Hence $\eta_{i+1}$ satisfies the first requirement in (\ref{eqn:regularization}). For the second requirement, recall that $\eta_{i+1}$ is constructed by subsequently blowing up $\epsilon_{i+1}^{-1}B_i\subset X_{i+1}$ and $B_{i+1}'\subset X_{i+1}'$. The set $\epsilon_{i+1}^{-1}B_i$ is disjoint from $C_{i+1}$ because $C_{i+1}\subset E_{i+1}$ and $B_i\cap C_i=\emptyset$. On the other hand, the image
$$
    \eta_i\pi_2(B_{i+1}') = \epsilon_{i+1}\pi_1(B_{i+1}')\subset X_i
$$
is disjoint from $C_i$, so $\pi_1(B_{i+1}')$ is disjoint from $E_{i+1}$ and thus from $C_{i+1}$. We conclude that the center $B_{i+1}$ of $\eta_{i+1}$ is disjoint from $C_{i+1}$. This completes the inductive step.

Formula~(\ref{eqn:regularization}) with $i=r$ gives a birational morphism
$$
    \eta_r\colon Z_r\longrightarrow X_r
    \quad\text{ such that }\quad
    \tau_r\colonequals\eta_r^{-1}\alpha_r\eta_r\in\BBir_k(Z_r).
$$
As a result, we obtain the commutative diagram
\[\xymatrix{
    & Z_r\ar[dl]_-f\ar[dr]^-g &\\
    X\ar@{-->}[rr]^-h && Y
}\]
where $f = \epsilon_1\cdots\epsilon_r\eta_r$ and $g = \widetilde{h}\eta_r$ are birational morphisms. Moreover,
$$
    \gamma
    \colonequals f^{-1}\alpha f
    = (\epsilon_1\cdots\epsilon_r\eta_r)^{-1}
    \alpha_0
    (\epsilon_1\cdots\epsilon_r\eta_r)
    = \eta_r^{-1}\epsilon_r^{-1}\cdots\epsilon_1^{-1}
    \alpha_0
    \epsilon_1\cdots\epsilon_r\eta_r
    = \eta_r^{-1}\alpha_r\eta_r
    = \tau_r,
$$
which belongs to $\BBir_k(Z_r)$.
Using the relations $h = gf^{-1}$ and $\beta = h\alpha h^{-1}$, we deduce that
$$
    \gamma
    = f^{-1}\alpha f
    = g^{-1}h\alpha h^{-1}g
    = g^{-1}\beta g.
$$
By Lemma~\ref{lem:parityBirMap_case2}, the actions of $\alpha$ and $\beta$ on the sets of $k$-points induce the same parity as the action of $\gamma$, which completes the proof.
\end{proof}

%--------------------Birational permutations on rational surfaces
\section{Birational permutations on rational surfaces}
\label{sect:birratsurf}

We prove Theorem~\ref{thm:evenPermutationsIntro} in this section. Using Theorem~\ref{thm:parityBirMapIntro}, this amounts to showing that over $\bF_q$, $q=2^m\geq 4$, permutations induced by the following maps are all even:
\begin{itemize}
    \item Birational permutations on a conic bundle over $\bP^1$ preserving the fiber class.
    \item Automorphisms of a rational del Pezzo surface.
    \item Elements of $\BCr_2(\bF_q)$ of finite order.
\end{itemize}
One may wonder if there exists a surface over $\bF_q$, $q = 2^m\geq 4$, that admits a birational odd permutation. Below we exhibit such an example over $\bF_4$.

\begin{eg}
Let us write $\bF_4 = \bF_2(\xi)$, where $\xi^2 + \xi + 1 = 0$, and let $\overline{\xi}$ denote the Galois conjugate of $\xi$. Consider the elliptic curve defined by the Weierstrass equation
$$
    E: y^2 + xy = x^3 + 1.
$$
Then $j(E) = 1$, and the group $\Aut(E)\cong\bZ/2\bZ$ is generated by $\sigma_E\colon(x,y)\mapsto(x,y+x)$ \cite{Sil09}*{Propositions~A.1.1 \& A.1.2}. One can verify straightforwardly that
$$
    E(\bF_2) = \{
        (1,0), (0,1), (1,1), p_\infty
    \},
    \qquad
    E(\bF_4) = E(\bF_2)\cup\{
        (\xi, 0),
        (\overline{\xi}, 0),
        (\xi, \xi),
        (\overline{\xi}, \overline{\xi})
    \}
$$
where $p_\infty$ denotes the point at infinity. Moreover, the involution $\sigma_E$ fixes $(0,1)$, $p_\infty$, and exchanges points in each of the pairs $\{(1,0), (1,1)\}$, $\{(\xi,0), (\xi,\xi)\}$, $\{(\overline{\xi},0), (\overline{\xi},\overline{\xi})\}$. In particular, $\sigma_E$ acts on $E(\bF_4)$ as a product of three transpositions and thus is odd. Now consider the $\bP^1$-bundle
$$
    X\colonequals\bP^1\times E
    \longrightarrow E
$$
and define $\sigma_X\in\Aut(X)$ by $\sigma_X(p,q)=(p,\sigma_E(q))$.
Then $\sigma_X$ acts on $X(\bF_4)$ as an odd permutation by Lemma~\ref{lemma:autoProjBundle}. In fact, it is not hard to see that this permutation consists of $5$ disjoint permutations of the same type as $\sigma_E$.
\end{eg}

%----------birational permutations on conic bundles
\subsection{Birational permutations on conic bundles}
\label{subsect:conicBundle}

Over a finite field $k$, a conic $C\subset\bP^2_k$ can only be one of the followings:
\begin{enumerate}[label=(\Roman*)]
\item\label{nondeg} $C$ is smooth, which implies that $C\cong\bP^1_k$.
\item\label{deg-double} $C$ is a double line.
\item\label{deg-conj} $C = \ell\cup\ell'$ where $\ell$ and $\ell'$ are conjugate over the quadratic extension.
\item\label{deg-k} $C = \ell\cup\ell'$ where $\ell$ and $\ell'$ are distinct lines both defined over $k$.
\end{enumerate}
As an analogue of projective bundles over finite sets (see \S\ref{subsect:finiteProjBundle}), given a finite set $B$, we define a \emph{conic bundle over $B$} to be a union of conics indexed by $B$:
$$
    \cC = \bigcup_{i\in B}C_i
    \quad\text{equipped with the map}\quad
    h\colon\cC\to B:C_i\mapsto i.
$$
Consider the set $\cC(k)$ of $k$-points on $\cC$. We are interested in elements $\sigma\in\Sym(\cC(k))$ of the form:
\begin{enumerate}[label=(\arabic*)]
    \item\label{fiberToFiber}
    For every $i\in B$, there exists $j\in B$ such that $\sigma(C_i(k)) = C_j(k)$. Then $h\sigma h^{-1}$ is well-defined as an element of $\Sym(B)$.
    \item\label{inducedByIsom}
    Each bijection $\sigma\colon C_i(k)\to C_j(k)$ is induced by an isomorphism between conics over $k$.
\end{enumerate}
Note that such elements form a subgroup of $\Sym(\cC(k))$.

\begin{lemma}
\label{lemma:keyFact}
Let $k=\bF_q$, $q = 2^m\geq 4$, and $\sigma\in\Sym(\cC(k))$ be an element satisfying \ref{fiberToFiber} and \ref{inducedByIsom}. Then $\sigma$ and $\sigma_B\colonequals h\sigma h^{-1}\in\Sym(B)$ have the same parity.
\end{lemma}

\begin{proof}
Since the parity of a permutation is invariant upon raising it to an odd power, we can assume that both $\sigma$ and $\sigma_B$ consist of disjoint cycles of sizes powers of $2$. In this setting, each nontrivial cycle is an odd permutation. Suppose that
$$
    O\colonequals\{p_1, ..., p_{r}\}\subset B,\quad r=2^s\geq1,
$$
is any orbit of $\sigma_B$. Then $\sigma$ acts on the set of $k$-points on $h^{-1}(O)\subset\cC$, and it suffices to show that this action is odd. This reduces the proof to the case $O = B$.

By property~\ref{inducedByIsom}, the fibers over $O$ are mutually isomorphic and thus of the same type. If they are of type~\ref{nondeg}, then the statement follows from Lemma~\ref{lemma:autoProjBundle}. The case of type~\ref{deg-double} is covered by the previous case by passing to the reduced substructure. If they are of type~\ref{deg-conj}, then the node in each fiber appears as the only $k$-point in that fiber. This implies that $\sigma$ and $\sigma_B$ have the same cycle type, thus are both odd.

Assume that the fibers are of type~\ref{deg-k}, that is, $C_i = h^{-1}(p_i) = \ell_i\cup\ell_i'$ where $\ell_i$ and $\ell_i'$ are copies of $\bP^1_k$. Let $\sigma_L$ denote the action of $\sigma$ on the set of lines
$$
    L\colonequals\{
        \ell_1, \ell_1', \ell_2, \ell_2', \dots, \ell_r, \ell_r'
    \}.
$$
In this case, the nodes $\delta_i\colonequals\ell_i\cap\ell_i'$ for $i=1,\dots,r$ form a single orbit under the action of $\sigma$. This forces $\sigma_L$ to be one of the following forms:
\begin{enumerate}[label=(\roman*)]
\item\label{2-orbits} $L$ has two orbits of size $r$. In this case, we can relabel the components of $C_i$ as $\ell_i^{+}$ and $\ell_i^{-}$ such that there is a cycle decomposition
$
    \sigma_L = (\ell_1^{+}, \dots, \ell_r^{+})(\ell_1^{-}, \dots, \ell_r^{-}).
$
\item\label{1-orbit} $L$ forms a single orbit of size $2r$. In this case, we can relabel the components of $C_i$ as $\ell_i^{+}$ and $\ell_i^{-}$ such that
$
    \sigma_L = (\ell_1^{+}, \dots, \ell_r^{+}, \ell_1^{-}, \dots, \ell_r^{-}).
$
\end{enumerate}
In both cases, we have the $\bP^1_k$-bundles
$$\xymatrix{
    \cC^{\pm} = \ell_1^{\pm}\cup\dots\cup\ell_r^{\pm}\ar[r]^-{h^{\pm}}
    &
    O^{\pm}:\ell_i^{\pm}\ar@{|->}[r] & p_i^{\pm}.
}$$
where $O^{\pm} = \{p_1^{\pm}, ..., p_r^{\pm}\}$ are two copies of $O$. Taking their (disjoint) union gives a conic bundle
$$\xymatrix{
   \widetilde{\cC} = \cC^{+}\amalg\cC^{-}\ar[rr]^-{\widetilde{h} = h^{+}\amalg h^{-}}
   &&
   O^{+}\amalg O^{-}.
}$$
Note that the node $\delta_i$ splits as $\delta_i^{+}\in\ell_i^{+}$ and $\delta_i^{-}\in\ell_i^{-}$ for each $1\leq i\leq r$.

Suppose that case~\ref{2-orbits} holds. Replacing the cycle $(\delta_1, \dots, \delta_r)$ in $\sigma$ by the product
$$
    (\delta_1^{+}, \dots, \delta_r^{+})(\delta_1^{-}, \dots, \delta_r^{-})
$$
defines an element $\widetilde{\sigma}\in\Sym(\widetilde{\cC}(k))$ that satisfies \ref{fiberToFiber} and \ref{inducedByIsom}. Now we have
$$
    \widetilde{h}\widetilde{\sigma}\widetilde{h}^{-1}
    = (p_1^{+}, ..., p_r^{+})(p_1^{-}, ..., p_r^{-})
$$
which is even. Because the fibers of $\widetilde{h}$ are smooth, we conclude that $\widetilde{\sigma}$ is even by the result for type~\ref{nondeg}. Since $\sigma$ has one less odd cycle than $\widetilde{\sigma}$, the parity of $\sigma$ is odd. If case~\ref{1-orbit} holds, we can define $\widetilde{\sigma}\in\Sym(\widetilde{\cC}(k))$ by replacing $(\delta_1, \dots, \delta_r)$ in $\sigma$ with the cycle
$$
    (\delta_1^{+}, \dots, \delta_r^{+}, \delta_1^{-}, \dots, \delta_r^{-}).
$$
Then $\widetilde{\sigma}$ satisfies \ref{fiberToFiber} and \ref{inducedByIsom}, and we have
$$
    \widetilde{h}\widetilde{\sigma}\widetilde{h}^{-1}
    = (p_1^{+}, ..., p_r^{+}, p_1^{-}, ..., p_r^{-})
$$
which is odd. We conclude in a similar way that $\widetilde{\sigma}$ is odd, which implies that $\sigma$ is odd.
\end{proof}

For our applications of the above lemma, we are interested in the case when $B$ is the set of $k$-points on a curve. The following corollary is then immediate.

\begin{cor}
\label{cor:conicBundle_v2}
Let $\cC\rightarrow D$ be a conic bundle over a curve $D$ over $k=\bF_{q}$, $q = 2^m\geq4$. Suppose that $f\in\BBir_k(\cC)$ preserves the conic bundle structure, and let $\rho(f)\in\Aut(D)$ be the induced automorphism on $D$. Then
\begin{itemize}
\item the actions of $f$ on $\cC(k)$ and $\rho(f)$ on $D(k)$ have the same parity, and
\item $f$ induced an even permutation on $\cC(k)$ if $D=\bP^1$.
\end{itemize}
\end{cor}

\begin{proof}
The fibers of $\cC$ over $D(k)$ form an example of a conic bundle over a finite set. The action of $f$ on $\cC(k)$ satisfies properties \ref{fiberToFiber} and \ref{inducedByIsom}. Then the first conclusion follows Lemma~\ref{lemma:keyFact}, and the second statement follows from Proposition~\ref{prop:PGLpermute}.
\end{proof}

%----------Automorphisms of rational del Pezzo surfaces
\subsection{Automorphisms of rational del Pezzo surfaces}
\label{subsect:delPezzo}

Over an arbitrary field $k$, a \emph{del Pezzo surface} $X$ is defined to be a smooth projective surface such that the anticanonical divisor $-K_X$ is ample. The \emph{degree} of $X$ is defined as the integer $d=K_X^2$ which takes values from $1$ to $9$. For example, a del Pezzo surface $X$ of degree $9$ is a \emph{Severi--Brauer surface}, namely, a surface that satisfies $X_{\overline{k}}\colonequals X\otimes_k\overline{k}\cong\bP^2_{\overline{k}}$. Below is a simple observation about automorphisms of del Pezzo surfaces over finite fields:

\begin{prop}
\label{prop:dP_finite_auto}
Let $X$ be a del Pezzo surface over a finite field $k$. Then $\Aut(X)$ is a finite group.
\end{prop}

\begin{proof}
The anticanonical class $-K_X$ is ample and thus $-rK_X$ becomes very ample for some $r\geq 1$. The linear system $|-rK_X|$ defines an embedding $X\hookrightarrow\bP^n$. Since every $f\in\Aut(X)$ preserves $K_X$, it extends to an automorphism on $\bP^n$. This defines an embedding $\Aut(X)\hookrightarrow\PGL_{n+1}(k)$. Then the statement follows as $\PGL_{n+1}(k)$ is a finite group when $k$ is finite.
\end{proof}

A surface $X$ over a field $k$ is called \emph{rational} if there exists a birational map $X\dashrightarrow\bP^2$ defined over $k$. In this section, we investigate the parities of the permutations on $X(\bF_q)$ induced by automorphisms of a rational del Pezzo surface $X$ over $\bF_q$ for $q=2^m\geq 4$. Our goal is to prove the following theorem:

\begin{thm}
\label{thm:delPezzo}
Automorphisms of a rational del Pezzo surface $X$ over $\bF_{q}$ for $q=2^m\geq4$ induce only even permutations on $X(\bF_q)$.
\end{thm}

We will proceed the proof case-by-case with the degree $d$ going from high to low. As the parity of a permutation is invariant upon taking an odd power, we will assume the order of a permutation to be a power of $2$ when studying its parity. The following lemma will be useful under this assumption:

\begin{lemma}
\label{lem:powerOf2}
Let $X$ be a surface defined over $k=\bF_{q}$, $q = 2^m\geq 4$, which is rational over the algebraic closure, and let $\sigma\in\Sym(X(k))$.
\begin{enumerate}[label = \textup{(\arabic*)}]
    \item\label{powerOf2_fixpoints}
    If $\ord(\sigma) = 2^r$ for some $r\geq0$, then $\sigma$ has odd number of fixed points.
    \item\label{powerOf2_parity}
    If $\ord(\sigma) = 2$ and the number of fixed points equals $1$ modulo $4$, then $\sigma$ is even.
\end{enumerate}
\end{lemma}

\begin{proof}
It is well-known that $|X(k)| = q^2 + aq + 1$ for some non-negative integer $a$ (\cite{Wei56}, see also \cite{Poo17}*{Proposition~9.3.24}). Since the size of each orbit of $\sigma$ divides $\ord(\sigma)=2^r$, we have
$$
    q^2 + aq + 1
    = 2\ell + |\{\text{fixpoints of }\sigma\}|
    \quad\text{for some}\quad
    \ell\geq 0
$$
which implies \ref{powerOf2_fixpoints}. Assume $\ord(\sigma)=2$, that is, $\sigma$ is an involution. In particular, $\sigma$ is a product of disjoint $2$-cycles. If $\sigma$ has $4b+1$ fixed points, then the amount of $2$-cycles equals
$$
    \frac{1}{2}(|X(k)| - (4b + 1))
    = \frac{1}{2}(q^2 + aq - 4b)
$$
which is an even number for $q=2^m\geq 4$. This proves \ref{powerOf2_parity}.
\end{proof}

\begin{rmk}
\label{rmk:oddpermutation-dP6}
Over $\bF_2$, there exists an automorphism of a rational del Pezzo surfaces $X$ which induces an odd permutation on $X(\bF_2)$. To construct an example, one can start with a quadratic transformation $f\in\BCr_2(\bF_2)$, that is, $f$ is defined by the linear system of conics passing through three non-collinear points in $\bP^2$ that form a $\Gal(\bF_8/\bF_2)$-orbit. By Lemma~\ref{lem:quad-invol}, upon composing $f$ with a linear transformation, we can assume that $f$ is involutive, so that $\Bs(f)=\Bs(f^{-1})$. Blowing up $\bP^2$ at $\Bs(f)$ produces a del Pezzo surface $X$ of degree $6$ and resolves $f$ as an automorphism $f'$ on $X$. The action of $f$ on $\bP^2(\bF_2)$ is odd by Lemma~\ref{lem:std-invol}, so the action of $f'$ on $X(\bF_2)$ is odd as well by Theorem~\ref{thm:parityBirMapIntro}.
\end{rmk}

%-----Rational del Pezzo surfaces of degree at least 4
\subsubsection{Rational del Pezzo surfaces of degree at least 4}
\label{subsubsect:dPS_d>=4}

Here we prove that the claim of Theorem~\ref{thm:delPezzo} holds for rational del Pezzo surfaces of degree $d\geq 4$. The case $d=9$ is covered by Proposition~\ref{prop:PGLpermute} since a rational Severi--Brauer surface is isomorphic to $\bP^2$ by Ch\^{a}telet. (See, for example, \cite{GS17}*{Theorem~5.1.3}.) We prove the remaining cases below:

\begin{prop}
\label{prop:delPezzo_4d8}
Automorphisms of a rational del Pezzo surface $X$ over $k=\bF_{q}$, $q=2^m\geq4$, of degree $4\leq d\leq 8$ induce only even permutations on $X(\bF_q)$.
\end{prop}

\begin{proof}
Let $g\in\Aut(X)$. Because raising to an odd power does not change the parity of a permutation, we can assume the action on $X(k)$ induced by $g$ has order a power of $2$. This allows us to choose a point $p\in X(k)$ fixed by $g$ as guaranteed by Lemma~\ref{lem:powerOf2}~\ref{powerOf2_fixpoints}.

\begin{description}[leftmargin=0cm]
\item[Case~$d=8$.]
If $X$ is not minimal (over $k$), then there exists a $(-1)$-curve $E\subset X$ over $k$, and contracting $E$ gives a morphism $h\colon X\to\bP^2$. Every $g\in\Aut(X)$ leaves $E$ invariant,
thus is conjugate to an automorphism of $\bP^2$ fixing $h(E)\in\bP^2$. Therefore, $g$ induces an even permutation on $X(k)$ by Proposition~\ref{prop:PGLpermute} and Theorem~\ref{thm:parityBirMapIntro}.

If $X$ is minimal, then it is a quadric surface obtained by blowing up $\bP^2_k$ at a point of degree $2$ (resp. two rational points), and then contracting the proper transform of the unique line through that point (resp. the two rational points). In particular, over the quadratic extension $L\colonequals\bF_{q^2}$, we have $X_L\cong\bP^1_L\times\bP^1_L$. Let $X_7$ be the blow-up of $X$ at the fixed point $p$ and let $E$ be the exceptional curve. Then the two rulings of $X\cong\bP^1_L\times\bP^1_L$ meeting at $p$ lift to disjoint $(-1)$-curves $E_1, E_2\subset X_7$ over $L$ that are conjugate to each other (resp. both rational) over $k$, and $g$ is conjugate to $g'\in\Aut(X_7)$ which leaves the set $\{E_1, E_2\}$ invariant. Let $h\colon X_7\to\bP^2_k$ be the contraction of $E_1$ and $E_2$. Then $hg'h^{-1}$ is a $\PGL_3(k)$-action on $\bP^2$ leaving the set $\{h(E_1), h(E_2)\}$ invariant.
It then follows from Proposition~\ref{prop:PGLpermute} and Theorem~\ref{thm:parityBirMapIntro} that $g$ induces an even permutation.

\item[Case~$d=7$.]
There is a unique $(-1)$-curve $E$ on $X$ that is invariant under both $\Gal(\overline{k}/k)$ and $g$. Hence contracting $E$ gives $X\to X_8$ where $X_8$ is a del Pezzo surface of degree 8, and $g$ descends to an automorphism $g_8$ on $X_8$. The result then follows from Theorem~\ref{thm:parityBirMapIntro} and Case $d=8$.

\item[Case~$d=6$.]
Over the algebraic closure, $X_{\overline{k}}$ is obtained by blowing up three points $a_1,a_2,a_3$ in $\bP^2_{\overline{k}}$, and it contains six $(-1)$-curves $E_1$, ..., $E_6$ such that, for $i\neq j$, we have $E_i\cdot E_j = 1$ if $j\equiv i+1\pmod{6}$ and $E_i\cdot E_j = 0$ otherwise. Note that both $\Gal(\overline{k}/k)$ and $g$ act on this set of $(-1)$-curves and preserve the intersection relations.

If $p$ does not lie on any of these $(-1)$-curves, then the blow-up $X_5=\Bl_p(X)$ is a del Pezzo surface of degree~$5$, and $g$ lifts to an automorphism $g_5$ of $X_5$. Over $\overline{k}$, the three lines in $\bP^2_{\overline{k}}$ that pass through $p$ and one of $a_1,a_2,a_3$ lift to pairwise disjoint $(-1)$-curves on $X_5$ that meet three disjoint members of $\{E_1,\dots,E_6\}$. Since this configuration is invariant under the action of both $\Gal(\overline{k}/k)$ and $g_5$, we can contract the three new $(-1)$-curves to get $X_5\rightarrow X_8$, where $X_8$ is a del Pezzo surface of degree $8$, such that $g_5$ descends to an automorphism $g_8$ of $X_8$. By Case~$d=8$, $g_8$ induces an even permutation on $X_8(k)$, and we finish by applying Theorem~\ref{thm:parityBirMapIntro}.

Suppose $p$ lies on one of the $(-1)$-curves, say, $E_1$. If $p$ does not lie on any other $(-1)$-curve, then $E_1$ is invariant under both $\Gal(\overline{k}/k)$ and $g$. We can then blow down $E_1$ to get $X\to X_7$ where $X_7$ is a del Pezzo surface of degree $7$, and $g$ descends to an automorphism $g_7$ of $X_7$. By Case $d=7$, $g_7$ induces an even permutation on $X_7(k)$, and we finish by applying Theorem~\ref{thm:parityBirMapIntro}.
Otherwise, $p$ lies on the intersection of two lines, say, $E_1$ and $E_2$. Then the orbit structure of $\{E_1,\ldots,E_6\}$ under both $\Gal(\overline{k}/k)$ and $g$ is either $\{E_1\}\cup\dots\cup\{E_6\}$ or $\{E_1,E_2\}\cup\{E_3,E_6\}\cup\{E_4,E_5\}$. In either case, $\{E_3,E_6\}$ is invariant under both $\Gal(\overline{k}/k)$ and $g$, so blowing down $E_3,E_6$ yields $X\to X_8$, and $g$ descends to an automorphism on $X_8$. We finish by applying Case~$d=8$ and  Theorem~\ref{thm:parityBirMapIntro}.

\item[Case~$d=5$.]
Over the algebraic closure, $X_{\overline{k}}$ is obtained by blowing up four points $b_1,b_2,b_3,b_3$ in $\bP^2_{\overline{k}}$, and it contains ten $(-1)$-curves, where six of them come from the lines passing through two of $b_1,b_2,b_3,b_3$, and the remaining four are the exceptional curves. Let us denote the $(-1)$-curve passing through $b_i$ and $b_j$ as $D_{kl}$, where $k<l$ and $\{k,l\}=\{1,2,3,4\}\setminus\{i,j\}$, and denote the exceptional curve over $b_i$ as $D_{i5}$. In this setting, we have $D_{ij}\cdot D_{kl} = 1$ if $i,j,k,l$ are pairwise distinct and $D_{ij}\cdot D_{kl} = 0$ otherwise.

If $p$ does not lie on any of the $(-1)$-curves, then the blow-up $X_4=\Bl_p(X)$ is a del Pezzo surface of degree~4, and $g$ lifts to an automorphism $g_4$ on $X_4$. Let $E_p\subset X_4$ denote the exceptional curve lying above $p$. Over $\overline{k}$, the lines (resp. the conic) passing through $p$ and one of (resp. all of) $b_1,b_2,b_3,b_4$ lift to five pairwise disjoint $(-1)$-curves that intersect $E_p$. These $(-1)$-curves form a set invariant under $\Gal(\overline{k}/k)$, so we can blow them down to get $X_4\to\bP^2$, and $g_4$ also descends to an automorphism on $\bP^2$. An application of Proposition~\ref{prop:PGLpermute} and Theorem~\ref{thm:parityBirMapIntro} does the job.

Suppose $p$ lies on a $(-1)$-curve, say, $D_{12}$. If $p$ does not lie on any other $D_{ij}$, then $D_{12}$ is invariant under both $\Gal(\overline{k}/k)$ and $g$, so we can contract it to get $X\to X_6$, where $X_6$ is a del Pezzo surface of degree~6, and $g$ descends to an automorphism of $X_6$. We are then done by Case~$d=6$ and Theorem~\ref{thm:parityBirMapIntro}. If $p$ lies on another $(-1)$-curve, we can assume this is $D_{34}$. One can verify that these are the only two $(-1)$-curves that contain $p$. It follows that $D_{12}\cup D_{34}$ is defined over $k$ and invariant under $g$. The other $(-1)$-curves that intersect $D_{12}\cup D_{34}$ are $D_{35}, D_{45}, D_{15}, D_{25}$. Hence the union $D_{35}\cup D_{45}\cup D_{15}\cup D_{25}$ is defined over $k$ and invariant under $g$. These four curves are pairwise disjoint. Contracting them gives $X\to\bP^2$, and $g$ descends to an automorphism of $\bP^2$. We are done after applying Proposition~\ref{prop:PGLpermute} and Theorem~\ref{thm:parityBirMapIntro}.

\item[Case~$d=4$.]
First assume that $p$ does not lie on a $(-1)$-curve. Then the blow-up of $X$ at $p$ is a cubic surface $X_3\subset\bP^3$, and the exceptional curve $E\subset X$ is a line in $\bP^3$ over $k$. Each plane $H\subset\bP^3$ containing $E$ intersects $X_3$ in a residual conic, so the pencil of such planes determines a conic bundle $X_3\to\bP^1$ over $k$. Corollary~\ref{cor:conicBundle_v2} yields the claim in this case.

Suppose that $p$ lies on a $(-1)$-curve. If it lies on only one such curve, then we can blow this curve down, and $g$ will descend to an automorphism of a del Pezzo surface of degree $5$. Then the claim follows from Case~$d=5$ and Theorem~\ref{thm:parityBirMapIntro}. Otherwise, $p$ lies on exactly two $(-1)$-curves. This defines a (singular) conic $Q$ on $X$. We can then define a conic bundle as follows: The linear system $|-K_X|$ embeds $X$ into $\bP^4$ as an intersection of two quadrics. Consider the pencil of hyperplanes containing $Q$. Each hyperplane intersects $X$ at a conic residual to $Q$. This defines a morphism $X\to\bP^1$ where the fibers are conics. Since $g$ preserves $Q$ and extends to an automorphism of $\bP^4$, it preserves the conic bundle structure. Hence, it follows from Corollary~\ref{cor:conicBundle_v2} that $g$ induces an even permutation on $X(k)$.
\qedhere
\end{description}
\end{proof}

%-----Rational del Pezzo surfaces of low degrees
\subsubsection{Rational del Pezzo surfaces of low degrees}
\label{subsubsect:dPS_d=123}

To prove Theorem~\ref{thm:delPezzo} for rational del Pezzo surfaces of degree $d=1,2,3$, we first prove a fact about permutations induced by a double cover structure that appear in these cases.

\begin{lemma}
\label{lem:Galoi2Cover2}
Let $Y=\bP(a_0,\ldots,a_n)$ be a weighted projective space, with $a_i$ the weights, over $k = \bF_q$, where $q = 2^m\geq2$. Let $\pi\colon X\to Y$ be a degree two Galois cover where $X$ is given by
$$
    w^2+fw+g=0,
$$
for some nonzero homogeneous polynomials $f$ and $g$ in the weighted polynomial ring $k[x_0,\ldots,x_n]$ of degrees $d$ and $2d$, respectively. Let $\beta\in\Aut(X)$ be the deck transformation and $B\subset X$ be the ramification locus defined by $f=0$. Assume that there is an exact sequence of groups
$$\xymatrix{
    1\ar[r] & \langle\beta\rangle\ar[r] & \Aut(X)\ar[r]^-{\pi_\ast} & \Aut(Y)
}$$
where $\pi_*h\colonequals\pi h\pi^{-1}$ for every $h\in\Aut(X)$,
and that $\beta$ acts as an even permutation on $X(k)$.
Then every $h\in\Aut(X)$ induces an even permutation on $X(k)\setminus B(k)$.
\end{lemma}

\begin{proof}
Let $h\in\Aut(X)$ and denote $h_0\colonequals\pi_*h\in\Aut(Y)$. Since $h_0$ fixes the branch locus, $h_0^*(f)=cf$ for some nonzero constant $c\in k$. Let $k(X)$ be the function field of $X$, which is a quadratic extension over $k(Y)$, so by the Artin--Shreier theory, it is given by
$$
    u^2+u=z
    \quad\text{for some}\quad
    z\in k(Y).
$$
In our setting, the equation $w^2+fw+g=0$ can be turned into
\begin{equation}
\label{eqn:ArtinShreier-1}
    w'^2+w' = \frac{g}{f^2}
    \quad\text{where}\quad
    w' = \frac{g}{fw}.
\end{equation}
This is our Artin--Shreier extension. Now consider the double cover coming from the composition $h_0\pi\colon X\to Y$. Under this viewpoint, we can repeat the same calculation to conclude that $k(X)$ is given by the extension
\begin{equation}
\label{eqn:ArtinShreier-2}
    w''^2+w'' = \frac{g'}{c^2f^2}
    \quad\text{where}\quad
    g'=h_0^*(g).
\end{equation}
It is well-known that \eqref{eqn:ArtinShreier-1} and \eqref{eqn:ArtinShreier-2} define the same extension if and only if there exists $a\in k(Y)$ such that
\begin{equation}
\label{eqn:ArtinShreier-equalext}
    \frac{g'}{c^2f^2} = \frac{g}{f^2}+a^2+a,
    \qquad\text{or equivalently,}\qquad
    g'=c^2g+c^2f^2(a^2+a).
\end{equation}
By comparing the degrees among the terms, we conclude that $a\in k$.

Define an automorphism $h'\in\Aut(X)$ by
$$
    x_i\mapsto h_0^*(x_i),\quad w\mapsto cw+caf.
$$
Then $h'^*(f)=cf$ and $h'^*(g)=g'$, and one can use \eqref{eqn:ArtinShreier-equalext} to verify that this is well-defined. Let us show that $h'$ induces an even permutation on $X(k)\setminus B(k)$ case-by-case:
\begin{itemize}
\item ($a=0$) Let $p\in \pi(X(k))\subset Y(k)$, singular or non-singular, and not lying on the branch locus, and $O_p$ be the orbit of $p$ under $h_0$. Let $r=|O_p|$ and note that $\pi^{-1}(O_p)$ consists of $2r$ many $k$-points. The assumption $a=0$ implies that $\pi^{-1}(O_p)$ breaks into two orbits of the same size under $h'$. Hence $h'$ induces an even permutation on $\pi^{-1}(O_p)$. As a consequence, $h'$ induces an even permutation on $X(k)\setminus B(k)$.

\item ($a=1$) The transformation $\beta$ is defined by $\beta^*(x_i)=x_i$ and $\beta^*(w)=w+f$, so $h'\beta$ has the same formula as $h'$ but with $a = 0$, thus induces an even permutation on $X(k)\setminus B(k)$ by the previous case. The fact that $\beta$ fixes every point on $B$ implies that it is an even permutation on $X(k)\setminus B(k)$. Therefore, $h'$ is an even permutation on $X(k)\setminus B(k)$.

\item ($a\neq0,1$) Keep the notation of $p,O_p,r$ as in the case $a=0$. Because $h_0^r$ fixes $p$ as a point in $Y=\bP(a_1,\dots,a_n)$, it rescales the coordinates of $p$ by a constant $e$ respecting the weights. Since $h_0^*(f)=cf$, plugging in $p$ gives $f(h_0(p))=cf(p)$. This implies $e^d=c^r$. As a result, we get $g(h_0^r(p))=c^{2r}g(p)$. On the other hand, applying $h_0^*$ inductively on \eqref{eqn:ArtinShreier-equalext} gives
$$
    h_0^{*r}(g)
    = c^{2r}g + r(a^2+a)c^{2r}f^2.
$$
Plugging in $p$, we get
$$
    c^{2r}g(p)
    = g(h_0^r(p))
    = c^{2r}g(p) + r(a^2+a)c^{2r}f(p)^2,
$$
so that $r(a^2+a)=0$, which implies $r$ is even. Hence $h'^{*r}(w)=c^rw$, so both points above $p$ are fixed by $h'^r$. So then $\pi^{-1}(O_p)$ breaks into two orbits of size $r$ under $h'$, which shows $h'$ induces even permutation on $X(k)\setminus B(k)$.
\end{itemize}

Now we finish the proof by showing $h$ is an even permutation on $X(k)\setminus B(k)$. The composition $hh'^{-1}$ acts as the identity on $Y$, so it is either the identity or $\beta$. Because $h'$ and $\beta h'$ both induce even permutations on $X(k)\setminus B(k)$, the result follows.
\end{proof}

\begin{prop}
\label{prop:delPezzo_d23}
Automorphisms of a rational del Pezzo surface $X$ over $\bF_{q}$, where $q=2^m\geq4$, of degree $d=2,3$ induce only even permutations on $X(\bF_q)$.
\end{prop}

\begin{proof}
\begin{description}[leftmargin=0cm]

\item[Case~$d=2$.]
The anticanonical model of $X$ is a hypersurface of degree 4 in the weighted projective space $\bP(w,x,y,z)=\bP(2,1,1,1)$, defined by
$$
    w^2+fw=g,
$$
where $f,g\in k[x,y,z]$ have degrees 2, 4 respectively \cite{kollar1999rational}*{Theorem~III.3.5}. The linear system $|-K_X|$ defines a double cover $\pi\colon X\to\bP^2$ sending $[w:x:y:z]$ to $[x:y:z]$. The double cover involution on $X$ is called the \emph{Geiser involution}, which we denote by $\gamma$. Since $K_X$ is preserved under any automorphism, we have an exact sequence
$$
    0\longrightarrow\langle\gamma\rangle
    \longrightarrow\Aut(X)
    \longrightarrow\Aut(\bP^2).
$$

Let us first prove that $\gamma$ induces an even permutation. By Lemma~\ref{lem:powerOf2}~\ref{powerOf2_parity}, it suffices to show that the fixed point set $\mathrm{Fix}(\gamma)(\bF_q)$ of $\gamma$ in $X(k)$ has cardinality $|\mathrm{Fix}(\gamma)(\bF_q)|\equiv 1\bmod4$. We have
\begin{equation}
\label{eqn:geiser}
    \gamma([w:x:y:z])=[-w-f:x:y:z].
\end{equation}
In characteristic 2, the fixed locus is given by $f=0$, a conic in $\bP^2$. This contains $q+1$ many $\bF_q$-points if it is smooth. If singular, it contains either $1$, $2q+1$, or $q+1$ many $\bF_q$-points if it consists respectively of two conjugate $\bF_{q^2}$-lines, two $\bF_q$-lines, or a double line. Because $q=2^m\geq 4$, we have $|\mathrm{Fix}(\gamma)(\bF_q)|\equiv1\bmod4$, as desired.

Now applying Lemma~\ref{lem:Galoi2Cover2}, we conclude that every $h\in\Aut(X)$ induces an even permutation on $X(k)\setminus B(k)$ where $B=\{f=0\}$. Hence, to finish the proof for $d=2$, it suffices to show $h$ induces an even permutation on $B(k)$. Since $B$ is a conic in $\bP^2$, this follows from Lemma~\ref{lemma:keyFact}.

\item[Case~$d=3$.]
Let $g\in\Aut(X)$ and assume that its action on $X(k)$ has order a power of $2$. Then Lemma~\ref{lem:powerOf2}~\ref{powerOf2_fixpoints} implies that $g$ has a fixed point $p\in X(k)$. If $p$ does not lie on any $(-1)$-curve of $X_{\overline{k}}$, then $g$ lifts to the blow-up $\Bl_p(X)$ which is a del Pezzo surface of degree~$2$. Then the result follows from Case~$d=2$ proved above and Theorem~\ref{thm:parityBirMapIntro}.

Suppose $p$ lies on exactly one $(-1)$-curve $L$. Then $L$ is defined over $k$ and invariant under $g$. Contracting $L$ gives a del Pezzo surface $X_4$ of degree $4$, and $g$ descends to an automorphism of $X_4$. Then the result follows from Proposition~\ref{prop:delPezzo_4d8} and Theorem~\ref{thm:parityBirMapIntro}.

Suppose $p$ lies on exactly two $(-1)$-curves $L_1,L_2$. The linear system $|-K_X|$ embeds $X$ as a cubic surface in $\bP^3$. The plane containing $L_1,L_2$ intersects $X$ at a third $(-1)$-curve $L_3$. Since the union $L_1\cup L_2$ is invariant under both $\Gal(\overline{k}/k)$ and $g$, the curve $L_3$ is also invariant under $\Gal(\overline{k}/k)$ and $g$. Hence we can contract $L_3$ and conclude as in the previous case.

Suppose $p$ lies on three $(-1)$-curves $L_1,L_2,L_3$. Then $p$ is an Eckardt point, and $g$ lifts to an automorphism $g_2$ on the blow-up $X_2\colonequals\Bl_p(X)$, which is a weak del Pezzo surface of degree~$2$. The strict transforms of $L_1,L_2,L_3$ give a $\Gal(\overline{k}/k)$-invariant set of three disjoint $(-2)$-curves on $X_2$. We can contract them to get $X_2\to Y$, and $g_2$ descends to an automorphism on $Y$. The morphism $X_2\to\bP^2$ induced by the projection from $p$ factors through $Y\to\bP^2$, which is a double cover ramified along a singular quartic curve. (The singular points of $Y$ are above the singular points of the quartic.) The same argument as in Case $d=2$ above shows that every automorphism of $Y$ induces an even permutation. We finish by applying Theorem~\ref{thm:parityBirMapIntro}.
\qedhere
\end{description}
\end{proof}

\begin{prop}
\label{prop:delPezzo_d1}
Automorphisms of a rational del Pezzo surface $X$ over $\bF_{q}$, where $q=2^m\geq4$, of degree $d=1$ induce only even permutations on $X(\bF_q)$.
\end{prop}

\begin{proof}
The anticanonical model of $X$ is a hypersurface of degree $6$ in the weighted projective space $\bP(w,z,x,y)=\bP(3,2,1,1)$, defined by
$$
    w^2+a_1wz+a_3w=z^3+a_2z^2+a_4z+a_6
$$
where $a_i\in k[x,y]$ is homogeneous of degree $i$ \cite{kollar1999rational}*{Theorem~III.3.5}. The linear system $|-K_X|$ defines a rational map
$$
    \rho\colon X\dashrightarrow\bP^1
    : [w:z:x:y]\mapsto[x:y]
$$
whose indeterminacy locus consists of the single point $O\colonequals[1:1:0:0]\in X(\bF_q)$, and its general fibers are elliptic curves possessing $O$ as the identity elements. Since $K_X$ is fixed under any automorphism of $X$, we get an exact sequence
$$
    1\longrightarrow
    G\longrightarrow
    \Aut(X)\longrightarrow
    \Aut(\bP^1).
$$
Every element in $G$ has the form
$
    [w:z:x:y]\mapsto[W(w,z,x,y):Z(w,z,x,y):x:y]
$
which preserves the equation of $X$. Comparing the degrees in $x,y$ yields that $W=w$ or $W=w-a_1z-a_3$, and $Z^3=z^3$. Furthermore, if $a_4\neq0$, then $Z=z$, which implies that $G\simeq\bZ/2\bZ$ and is generated by the \emph{Bertini involution}
\begin{equation}
\label{eqn:bertini}
    \beta\colon[w:z:x:y]\mapsto[w-a_1z-a_3:z:x:y].
\end{equation}
(This involution induces the inverse map under the group law when restricting to a smooth fiber of the elliptic fibration $\rho\colon X\dashrightarrow\bP^1$.) Suppose that $a_4=0$. If $a_2\neq0$, then $Z^2=z^2$, thus $Z=z$, which implies again that $G=\langle\beta\rangle$. If $a_2=0$ and there exists a primitive third root of unity $\delta$, then $G$ is generated by $\beta$ and the element $[w:z:x:y]\mapsto[w:\delta z:x:y]$, hence $G\simeq\bZ/6\bZ$. If there is no such $\delta$, then $G=\langle\beta\rangle$.

We first show that the involution $\beta$ induces an even permutation on $X(\bF_q)$. By Lemma~\ref{lem:powerOf2}~\ref{powerOf2_parity}, it suffices to show that the fixed point set $\mathrm{Fix}(\beta)(\bF_q)$ of $\beta$ in $X(\bF_q)$ has cardinality $1\bmod 4$. In characteristic~$2$, the fixed locus is given by $a_1(x,y)z+a_3(x,y)=0$. Note that $O=[1:1:0:0]$ is a fixed rational point, and is the only such point when $x=y=0$. We now proceed by two cases depending on whether $a_1 = a_1(x,y)$ is the zero polynomial or not:
\begin{itemize}
\item If $a_1\neq0$, each $[x:y]\in\bP^1(\bF_q)$ with $a_1(x,y)\neq0$ contributes a fixed $\bF_q$-points by setting
\begin{align*}
    z&=a_3(x,y)/a_1(x,y),\\
    w^2&=z^3+a_2(x,y)z^2+a_4(x,y)z+a_6(x,y),
\end{align*}
which gives $q$ more points. Now let $[x_0:y_0]\in\bP^1(\bF_q)$ be such that $a_1(x_0,y_0)=0$. If $a_3(x_0,y_0)\neq0$, then $\rho^{-1}([x_0:y_0])$ has no fixed $\bF_q$-point. If $a_3(x_0,y_0)=0$, then $\rho^{-1}([x_0:y_0])$ is a singular affine curve with $q$-many $\bF_q$-points (unique solution in $w$ for every choice of $z$) which are all fixed under $\beta$. Hence, together with $O$, we have a total of either $q+1$ or $2q+1$ fixed $\bF_q$-points on $X$. In particular, $|\mathrm{Fix}(\beta)(\bF_q)|\equiv1\bmod4$.

\item If $a_1=0$, then $a_3\neq0$ since $X$ is smooth. Let $[x_0:y_0]\in\bP^1(\bF_q)$ be such that $a_3(x_0,y_0)=0$. Then the same argument as above shows that $\rho^{-1}([x_0:y_0])$ has $q$ many fixed $\bF_q$-points. Hence, $|\mathrm{Fix}(\beta)(\bF_q)|=q+1\equiv1\bmod4$.
\end{itemize}

The involution $\beta$ is also the deck transformation of the double cover $X\to\bP(2,1,1)$ which maps $[w:z:x:y]$ to $[z:x:y]$. This double cover is defined by $|-2K_X|$, which is preserved under any automorphism of $X$, so there is an exact sequence
$$
    1\longrightarrow\langle\beta\rangle
    \longrightarrow\Aut(X)
    \longrightarrow\Aut(\bP(2,1,1)).
$$
By Lemma~\ref{lem:Galoi2Cover2}, we get that any $h\in\Aut(X)$ induces an even permutation on $X(\bF_q)\setminus B(\bF_q)$ where $B\colonequals\{a_1z+a_3=0\}$. It remains to show that $h$ induces an even permutation on $B(\bF_q)$. Note that $O\in B(\bF_q)$ is the unique base-point of $|-K_X|$, so it is fixed under $h$. Moreover, since we only care about the rational points, it suffices to consider the reduced subscheme $B_0\colonequals B_{red}\setminus\{O\}$. We proceed by cases:
\begin{itemize}
\item If $a_1\neq0$ and $a_1$ does not divide $a_3$, then $B_0$ is isomorphic to $\bA^1$. Hence $h|_{B_0}$ induces an even permutation as a consequence of Proposition~\ref{prop:PGLpermute}.

\item If $a_1\neq0$ and $a_1$ divides $a_3$, then $B_0$ is isomorphic to a union of two copies of $\bA^1$ meeting at a point, where one copy is a section of the elliptic fibration while the other is a fiber. The result again follows from Proposition~\ref{prop:PGLpermute}.

\item If $a_1=0$, then $B_0$ is isomorphic to a disjoint union of $r$ copies of $\bA^1$ where $0\leq r\leq 3$. The case $r=0$ is trivial, and the case $r=1$ follows from Proposition~\ref{prop:PGLpermute}. If $r=2$, we can identify the disjoint union $\bA^1\cup\bA^1$ as the smooth part of a degenerate conic, so this case is covered by Lemma~\ref{lemma:keyFact}. Suppose that $r=3$. If $h$ leaves one $\bA^1$ invariant while switches the other two, then the claim follows from Proposition~\ref{prop:PGLpermute} and Lemma~\ref{lemma:keyFact}. If $h$ acts on the three copies of $\bA^1$ as a $3$-cycle, we can first compactify each $\bA^1$ as a $\bP^1$, which gives us a $\bP^1$-bundle over a finite set of $3$ elements, and then extend the action of $h$ to this bundle by multiplying it with a disjoint $3$-cycle. This new permutation is even by Lemma~\ref{lemma:autoProjBundle}, which implies that the original permutation is even.
\end{itemize}
As a result, the actions of $h$ are even on $X(\bF_q)\setminus B(\bF_q)$ and $B(\bF_q)$, thus is even on $X(\bF_q)$.
\end{proof}

\begin{proof}[Proof of Theorem~\ref{thm:delPezzo}]
Let $d=K_X^2$ be the degree of $X$. The claim follows from Proposition~\ref{prop:PGLpermute} for $d=9$, from Proposition~\ref{prop:delPezzo_4d8} for $4\leq d\leq 8$, from Proposition~\ref{prop:delPezzo_d23} for $d=2,3$, and from Proposition~\ref{prop:delPezzo_d1} for $d=1$.
\end{proof}

%----------Birational self-maps of finite order
\subsection{Birational self-maps of finite order}
\label{subsect:finiteOrder}

\begin{lemma}\label{lem:finiteorder}
Let $k$ be a perfect field.
Suppose $G\subset\Cr_2(k)$ is a finite subgroup.
Then there exists a surface $X$ together with a birational map $\phi\colon X\dashrightarrow\bP^2$ such that there is an injective homomorphism
\begin{equation}
\label{eqn:G-to-Aut}
    \phi^\ast\colon G \hookrightarrow \Aut(X)
    : g\to\phi^{-1}g\phi.
\end{equation}
Moreover, $X$ can be minimal with respect to $G$ in the sense that
\begin{enumerate}[label = \textup{(\arabic*)}]
    \item\label{minimalConic}
    $X$ admits a structure of a conic bundle with $\Pic(X)^G\cong\bZ^2$, or
    \item\label{minimalDelPezzo}
    $X$ is isomorphic to a del Pezzo surface with $\Pic(X)^G\cong\bZ$.
\end{enumerate}
\end{lemma}
\begin{proof}
The first statement can be proved by the same argument as in \cite{DI09}*{Lemma~3.5}. Now consider $G$ as a subgroup of $\Aut(X)$. Assume that $X$ is not minimal with respect to $G$, i.e., there exists a surface $Y$ and a birational morphism $h\colon X\to Y$ together with an inclusion
$$
    h^\ast\colon G \hookrightarrow \Aut(Y)
    : g\to h^{-1}gh
$$
such that the rank of $\Pic(Y)^G$ is strictly less than the rank of $\Pic(X)^G$.
This process terminates at either \ref{minimalConic} or \ref{minimalDelPezzo} by \cite{Isk79}*{Theorem~1G}.
\end{proof}

As a corollary, given $f\in\BCr_2(k)$ of finite order, we can always conjugate it to an automorphism on a minimal surface.
This reduces the parity problem for such elements to the problem on the parities induced by the automorphisms on a conic bundle or a del Pezzo surface.

\begin{proof}[Proof of Theorem~\ref{thm:evenPermutationsIntro}]
The statement for birational permutations on a conic bundle over $\bP^1$ follows from Corollary~\ref{cor:conicBundle_v2}. The statement for automorphisms of del Pezzo surfaces follows from Theorem~\ref{thm:delPezzo}. For birational permutations conjugate to maps of the previous two types, we apply Theorem~\ref{thm:parityBirMapIntro}. Note that this covers the elements in $\BCr_2(\bF_q)$ of finite order due to Lemma~\ref{lem:finiteorder}.
\end{proof}

%--------------------Non-existence of odd permutations
\section{Non-existence of odd permutations}
\label{sect:no-odd-permutation}

In this section, we produce a list of generators for $\BCr_2(k)$ where $k$ is a perfect field. Then we conclude the proof of Theorem~\ref{mainthm:intro} by showing that the generators in this list induce only even permutations over $k=\bF_{2^m}$ for $m\geq 2$. Throughout this section, we say a smooth zero-dimensional subscheme of a del Pezzo surface $X$ (resp. a conic bundle $X$) is \emph{in general position} if the blow-up of $X$ at the subscheme is still a del Pezzo surface (resp. a conic bundle over the same base).

%----------A list of generators over perfect fields
\subsection{A list of generators over perfect fields}
\label{subsect:generators}

\begin{lemma}
\label{lem:Schneider}
Let $k=\bF_{q}$ for $q=p^m$, where $p\geq2$ is a prime and $m\geq1$. 
\begin{enumerate}[label=\textup{(\arabic*)}]
\item\label{Schneider:degree2}
Let $p,p',q,q'$ be four points of degree $2$ in $\bP^2$ in general position. Then there exists $A\in\Aut(\bP^2)$ that sends $p,p'$ onto $q,q'$.
\item\label{Schneider:degree4}
Let $p,q$ be two points of degree $4$ in $\bP^2$ in general position. Then there exists $A\in\Aut(\bP^2)$ that sends $p$ onto $q$.
\end{enumerate}
\end{lemma}
\begin{proof}
To prove \ref{Schneider:degree2}, let $p_1,p_2$ (resp. $p_1',p_2'$, resp. $q_1,q_2$ resp. $q_1',q_2'$) be the geometric components of $p$ (resp. $p'$ resp. $q$ resp. $q'$). Then each  $p_i,p_i',q_i,q_i'$ is defined over $\bF_{q^2}$, $i=1,2$, and there  exists a unique $\bF_{q^2}$-automorphism $A$ of $\bP^2$ that sends $p_i$ onto $q_i$ and $p_i'$ onto $q_i'$ for $i=1,2$. For any $g\in\Gal(\bF_{q^2}/\bF_q)$ we have
$$
    (A^gA^{-1})(q_i)
    = A^g((p_i^{g^{-1}})^g)
    = (Ap_i^{g^{-1}})^g
    = (q_i^{g^{-1}})^g
    = q_i.
$$
In particular, $A^gA^{-1}$ is the identity map for all $g\in\Gal(\bF_{q^2}/\bF_q)$. Hence $A$ is defined over $\bF_q$. 

To prove \ref{Schneider:degree4}, let $p_1,p_2,p_3,p_4$ (resp. $q_1,q_2,q_3,q_4$) its geometric components of $p$ (resp. $q$). Then each $p_i$ and $q_i$ is defined over $\bF_{q^4}$, $i=1,2$, and over $\bF_{q^2}$, $p$ (resp. $q$) splits into two orbits, say $\{p_1,p_2\}$ and $\{p_3,p_4\}$ (resp. $\{q_1,q_2\}$ and $\{q_3,q_4\}$). By (1), there exists a $\bF_{q^2}$-automorphism $A$ of $\bP^2$ that sends $p_i$ onto $q_i$, $i=1,\dots,4$. As analogously to above, we obtain that $A^gA^{-1}q_i=(Ap_i^{g^{-1}})^g=q_i$ for any $g\in\Gal(\bF_{q^2}/\bF_q)$ and for $i=1,\dots,4$; hence $A$ is defined over $\bF_q$.
\end{proof}

%\begin{definition}
Let $S$ be a smooth projective surface over a perfect field $k$, $B$ a point or a curve defined over $k$, and $\pi\colon S\to B$ a surjective morphism over $k$. We say that $S/B$ is a \emph{Mori fibre surface} if $\pi$ has connected fibres, the relative Picard rank $\rho(S/B)$ of $S$ over $B$ is $\rho(S/B)=1$ and $-K_S$ is $\pi$-ample, that is $-K_S\cdot C>0$ for all curves $C$ contracted by $\pi$. A \emph{Sarkisov link} is a birational map $\phi\colon S\dashrightarrow S'$ between two Mori fibre spaces $\pi\colon S\rightarrow B$ and $\pi'\colon S'\rightarrow B'$ that is one of the following four types:
\begin{enumerate}[label=\textbf{Type~\Roman*.}, wide]
\item $B$ is a point, $B'$ is a curve and $\varphi$ is the blow-up of a point. 

\item $B\simeq B'$, and $\phi=\eta_2\eta_1$, where $\eta_1$ is the blow-up of a point $p=\{p_1,\dots,p_d\}$ of degree $d$ with those $p_i$ in general position, and $\eta_2$ is the contraction of an orbit of $(-1)$-curves of size $e$. We write $\phi=f_{de}$ if we want to emphasize the degree of the base-point of $\phi$.

\item the inverse of a link of type I, i.e.  $B$ is a curve, $B'$ is a point and $\phi$ is the contraction of a Galois-orbit of disjoint $(-1)$-curves defined over the algberaic closure of $k$.

\item $S=S'$ and $B,B'$ are both curves. If $S$ is rational, then $B=B'\simeq\bP^1$ and the $\phi$ is the exchange of the two fibrations.
\end{enumerate}

\begin{prop}
\label{pro:decomposition_links}
Let $X\to B$ and $X'\to B'$ be Mori fibre surfaces and $\psi\colon X\dashrightarrow X'$ a birational map. Then there is a decomposition $\psi=\phi_r\cdots\phi_1$ into Sarkisov links and isomorphism of Mori fibre surfaces such that
\begin{enumerate}[label=\textup{(\arabic*)}]
\item\label{decomposition_links:1} for $i=1,\dots,r-1$, $\phi_{i+1}\phi_i$ is not an automorphism,
\item\label{decomposition_links:2} for $i=1,\dots,r$, every base-point of $\phi_i$ is a base-point of $\phi_r\cdots\phi_i$.
\end{enumerate}
\end{prop}

\begin{proof}
The claim follows from the proof of \cite{Isk96}*{Theorem~2.5}, see also \cite{BM14}*{Proposition~2.7}.
\end{proof}

\begin{rmk}
\label{rmk:decomposition_links}
In particular, if $\psi$ induces a map $X(k)\to X'(k)$, then the link $\phi_1$ does not have any rational base-points. Moreover, the rational base-points of
$
    \psi(\phi_1)^{-1}=\phi_r\cdots\phi_2
$
are exactly the base-points of $(\phi_1)^{-1}$. Since $\phi_2\phi_1$ is not an automorphism, $\phi_2$ does not have a rational base-point.
\end{rmk}

The proof of the following proposition is similar to the proof of \cite{BM14}*{Theorem~1.2}, which shows that $\BCr_2(\bR)$ is generated by $\Aut(\bP^2)$ and elements of $\BCr_2(\bR)$ of degree $5$; the latter are in family \ref{generating_set:1} and they are the only non-linear maps in the generating set from Lemma~\ref{lem:generating_set_BCr} that exist over $k=\bR$. 

A surface $X_d,X_d'$ denote del Pezzo surfaces of degree $d$ and $Q,Q'$ del Pezzo surfaces of degree $8$ with $\rho(Q)=\rho(Q')=1$.

\begin{lemma}
\label{lem:generating_set_BCr}
Let $k$ be a perfect field. Then $\BCr_2(k)$ is generated by $\Aut(\bP^2)$ and the set of elements $f$ in the list below that exist over $k$.
\begin{enumerate}[label=\textup{(\arabic*)}]
\item\label{generating_set:1} $f$ sends the pencil of conics passing through two points of degree $2$ in general position onto a pencil of conics passing through two points of degree $2$ in general position.\\
If $k$ is finite, we can choose the two pencils to pass through the same points.
\item\label{generating_set:2} $f$ sends the pencil of conics passing through one point of degree $4$ in general position onto a pencil of conics passing through a point of degree $4$ in general position.\\
If $k$ is finite, we can choose the two pencils to pass through the same points.
\item\label{generating_set:3} $f$ is one of the following compositions, where $X_d$ is a del Pezzo surface of degree $d=(K_{X_d})^2$ and $f_{ab}$ is a Sarkisov link of type II blowing up a point of degree $a$ and its inverse blowing up a point of degree $b$:
\begin{equation}\label{gs:1}\begin{tikzcd}[column sep=0.2cm,row sep=0.2cm]
&X_6\ar[dl]\ar[dr] &&& X_2\ar[dl]\ar[dr] &&&X_1\ar[dl]\ar[dr] &&& X_3\ar[dl]\ar[dr]& \\ 
\bP^2\ar[rr,"f_{33}",dashed]&&\bP^2 & \bP^2\ar[rr,"f_{77}",dashed]&&\bP^2 & \bP^2\ar[rr,"f_{88}",dashed]&&\bP^2 &
\bP^2\ar[rr,"f_{66}",dashed]&&\bP^2
\end{tikzcd}
\end{equation}
or
\begin{equation}
\label{gs:2}
\begin{tikzcd}[column sep=0.2cm,row sep=0.2cm]
& X_7\ar[dl,swap]\ar[dr,"p"]&&X_{8-d}\ar[dl]\ar[dr]&&X_7\ar[ld,"p'",swap]\ar[dr] \\ 
 \bP^2\ar[rr,"f_{21}",dashed]&&Q\ar[rr,dashed,"f_{dd}"]&&Q\ar[rr,"f_{12} ",dashed]&&\bP^2 
\end{tikzcd}
\begin{tikzcd}[column sep=0.2cm,row sep=0.2cm]
\quad d\in\{7,6\}\\ p'=f_{dd}(p)
\end{tikzcd}
\end{equation}
or
\begin{equation}\label{gs:3}\begin{tikzcd}[column sep=0.2cm,row sep=0.2cm]
& X_7\ar[dl,swap]\ar[dr,"p"]&&X_3\ar[dl]\ar[dr]&&X_{5-d}\ar[ld]\ar[dr]&&X_3\ar[dl]\ar[dr]&&X_7\ar[dl,"p'",swap]\ar[dr] &\\ 
\bP^2\ar[rr,"f_{21}",dashed]&&Q\ar[rr,dashed,"f_{52}"]&&X_5\ar[rr,"f_{dd}",dashed]&&X_5\ar[rr,"f_{52}^{-1}",dashed]&&Q\ar[rr,"f_{12} ",dashed]&&\bP^2 
\end{tikzcd}
\begin{tikzcd}[column sep=0.2cm,row sep=0.2cm]
\quad d\in\{3,4\}\\ p'=f_{52}^{-1}f_{dd}f_{52}(p)
\end{tikzcd}
\end{equation}
or
\begin{equation}\label{gs:4}\begin{tikzcd}[column sep=0.2cm,row sep=0.2cm]
& X_7\ar[dl,swap]\ar[dr,"p"]&&X_3\ar[dl]\ar[dr]&&X_4\ar[ld,"p'",swap]\ar[dr]&\\ 
\bP^2\ar[rr,"f_{21}",dashed]&&Q\ar[rr,dashed,"f_{52}"]&&X_5\ar[rr,"f_{15}",dashed]&&\bP^2 
\end{tikzcd}
\begin{tikzcd}[column sep=0.2cm,row sep=0.2cm]
\quad p'=f_{52}(p)
\end{tikzcd}
\end{equation}
or
\begin{equation}\label{gs:5}\begin{tikzcd}[column sep=0.2cm,row sep=0.2cm]
& X_7\ar[dl,swap]\ar[dr,"p"]&&X_3\ar[dl]\ar[dr]&&X_3'\ar[dl]\ar[dr]&&X_7'\ar[dl,"p'",swap]\ar[dr]&\\ 
\bP^2\ar[rr,"f_{21}",dashed]&&Q\ar[rr,dashed,"f_{52}"]&&X_5\ar[rr,"f_{25}",dashed]&&Q' \ar[rr,"f_{12}",dashed]&&\bP^2
\end{tikzcd}
\begin{tikzcd}[column sep=0.2cm,row sep=0.2cm]
\quad p'=f_{25}f_{52}(p)
\end{tikzcd}
\end{equation}
or
\begin{equation}\label{gss:6}\begin{tikzcd}[column sep=0.2cm,row sep=0.2cm]
& X_7\ar[dl,swap]\ar[dr,"p"]&&X_5\ar[dl]\ar[dr,"t"]&&&X_5'\ar[lld,"p'",swap]\ar[dr]&&X_7'\ar[dl,"t'",swap]\ar[dr] &\\ 
\bP^2\ar[rr,"f_{21}",dashed]&&Q\ar[rr,dashed,"f_{31}"]&&X_6\ar[rrr,"f_{13}",dashed]&&&Q'\ar[rr,"f_{12} ",dashed]&&\bP^2 
\end{tikzcd}
\begin{tikzcd}[column sep=0.2cm,row sep=0.2cm]
p'=f_{31}(p)\\ t'=f_{13}(t)
\end{tikzcd}
\end{equation}
or
\begin{equation}\label{gss:7}\begin{tikzcd}[column sep=0.2cm,row sep=0.2cm]
& X_7\ar[dl,swap]\ar[dr,"p"]&&X_5\ar[dl]\ar[dr,"r"]&&X_{6-d}\ar[ld]\ar[dr]&&X_5'\ar[dl,"f_{dd}(r)",swap]\ar[dr]&&X_7'\ar[dl,"p' ",swap]\ar[dr] &\\ 
\bP^2\ar[rr,"f_{21}",dashed]&&Q\ar[rr,dashed,"f_{31}"]&&X_6\ar[rr,"f_{dd}",dashed]&&X_6'\ar[rr,"f_{13}",dashed]&&Q'\ar[rr,"f_{12} ",dashed]&&\bP^2 
\end{tikzcd}
\begin{tikzcd}[column sep=0.2cm,row sep=0.2cm]
d\in\{2,3,4,5\}\\ p'=f_{13}f_{dd}f_{31}(p)
\end{tikzcd}
\end{equation}
or
\begin{equation}\label{gss:8}\begin{tikzcd}[column sep=0.2cm,row sep=0.2cm]
& X_4'\ar[dl,swap]\ar[dr,"p"]&&X_{5-d}\ar[dl]\ar[dr]&&X_4'\ar[ld,"p'",swap]\ar[dr]&\\ 
\bP^2\ar[rr,"f_{51}",dashed]&&X_5\ar[rr,dashed,"f_{dd}"]&&X_5\ar[rr,"f_{15}",dashed]&&\bP^2 
\end{tikzcd}
\begin{tikzcd}[column sep=0.2cm,row sep=0.2cm]
\quad d\in\{4,3\}\\ p'=f_{dd}(p)
\end{tikzcd}
\end{equation}
Moreover, all links of the form $f_{dd}$ can be chosen to be involutions, except possibly $f_{66}$ in (\ref{gs:1}), $f_{33}$ and $f_{22}$ in (\ref{gss:7}).
\end{enumerate} 
\end{lemma}

Since the proof of Lemma~\ref{lem:generating_set_BCr} is quite long, we will check afterwards in Lemma~\ref{lem:redundant_generators} that the generators (\ref{gs:5}) and (\ref{gss:6}), (\ref{gss:7}, $d=2$) and (\ref{gss:8}, $d=4$) are redundant.

\begin{proof}
First note that any element in \ref{generating_set:3} is contained in $\BCr_2(k)$ as they only contract curves not defined over the ground field $k$. The list of involutions is from \cite{Isk96}*{Theorem 2.6}. For \ref{generating_set:1} and \ref{generating_set:2}, the claim over a finite field $k$ follows from Lemma~\ref{lem:Schneider}.

Let $\psi\in\BCr_2(k)$. There is a decomposition into Sarkisov links $\psi=\phi_r\cdots\phi_1$ as in Proposition~\ref{pro:decomposition_links}. We do induction on $r$, the case $r=0$ corresponding to $\psi\in\Aut(\bP^2)$. Let $r\geq1$. Then $\phi_1$ is a link of type I or II, and its base-point is a base-point of $\psi$, so is of degree $\geq2$. By \cite{Isk96}*{Theorem~2.6(i,ii)}, $\phi_1$ a link of type I with a base-point of degree $4$ or a link of type II of the form $f_{88},f_{77},f_{66},f_{33},f_{21}$ or $f_{51}$. We are going to look at these cases separately.

\begin{enumerate}[wide, labelwidth=!]
\item[\bf(a)] If $\phi_1:\bP^2\dashrightarrow X$ is a link of type~I, then it is the the blow-up of a point of degree $d_1=4$; $X/\bP^1$ is a conic bundle whose fibres are the strict transforms of conics through the four points, and $K_X^2=5$. Now $\phi_2$ is either a link of type~II of conic bundles, a link of type III \cite{Isk96}*{Theorem~2.6(i-iv)}, or an isomorphism. As $\phi_2\phi_1\notin\Aut(\bP^2)$ by hypothesis (see Proposition~\ref{pro:decomposition_links}~\ref{decomposition_links:1}), $\phi_2$ is a link of type~II of conic bundles or an isomorphism. Moreover, $\psi\phi_1^{-1}=\phi_r\cdots\phi_2$ is well-defined on $X(k)$, so $\phi_2$ is well-defined on $X(k)$ as well by Remark~\ref{rmk:decomposition_links}. Let $r-1\geq s\geq2$ be the maximal index such that $\phi_i$ is an isomorphism over $\bP^1$ or a link of type~II over $\bP^1$ without a rational base-point for any $2\leq i\leq s$. The map $\phi_s\cdots\phi_1$ is a birational map over $\bP^1$ from $X$ to a Mori fibre surface $X'/\bP^1$. We now look at two cases

If $\phi_{s+1}$ is a link of type III, then  $\nu':=\phi_{s+1}\phi_s\cdots\phi_2\phi_1$ is as in \ref{generating_set:2}. Note that $\psi\nu^{-1}=\phi_r\cdots\phi_{s+2}$ is as in Proposition~\ref{pro:decomposition_links}.

If $\phi_{s+1}$ is not a link of type III, then the map $\nu:=\phi_1^{-1}\phi_s\cdots\phi_2\phi_1\in\BCr_2(k)$ is as in \ref{generating_set:2} and the map $\psi\nu^{-1}=\phi_r\cdots\phi_{s+1}\phi_1$ is as in Proposition~\ref{pro:decomposition_links} since the base-point of $\phi_1$ is a base-point of $\phi_r\cdots\phi_{s+1}$ by construction.

\item[\bf(b)] Suppose that $\phi_1$ is a link of type~II, i.e. one of the forms $f_{33}$, $f_{66}$, $f_{77}$, $f_{88}$, $f_{21}$, or $f_{51}$. In the first four cases it is of the form (\ref{gs:1}) and we proceed by induction with $\psi\phi_1^{-1}=\phi_r\cdots\phi_2$. If $\phi_1$ is of the form $f_{21}$ (case {\bf(b1)}) or $f_{51}$ (case {\bf(b2)}), then $\phi_1^{-1}$ has a rational base-point $p$, which is the unique base-point of $\psi\phi_1^{-1}=\phi_r\cdots\phi_2$. Since $\phi_2\phi_1$ is not an automorphism by hypothesis, $p$ is not a base-point of $\phi_2$. Then $\phi_2(p)$ is the unique rational base-point of $\psi\phi_1^{-1}\phi_2^{-1}=\phi_r\cdots\phi_3$. It may or may not be a base-point of $\phi_3$. 

\begin{enumerate}[wide, labelwidth=!]
\item[\bf(b1)] Suppose that $\phi_1=f_{21}\colon\bP^2\dashrightarrow Q$. Then $\phi_2$ is a link of type I (case {\bf(b1.1)}) or II \cite{Isk96}*{Theorem~2.6}. If $\phi_2$ is a link of II, then it is of the form $f_{77},f_{66},f_{44}$ (case {\bf (b1.2)}) or $f_{52}$ (case {\bf(b1.3)}) or $f_{31}$ (case {\bf(b1.4)}) by \cite{Isk96}*{Theorem~2.6(ii)}.
The option $\phi_2=f_{12}$ does not occur since it forces $\phi_2\phi_1\in\Aut_k(\bP^2)$, which is not allowed by hypothesis. 

\item[\bf(b1.1)] Suppose that $\phi_2\colon Q\dashrightarrow X$ is a link of type I. Then it is the inverse of blowing-up a point $t$ of degree $2$ \cite{Isk96}*{Theorem~2.6(i)}. Then $K_X^2=6$ and $X\rightarrow\bP^1$ is a Mori fibre space whose fibres are the images by $\phi_2\phi_1$ of conics in $\bP^2$ passing through $p$ and $\phi_1^{-1}(t)$. Now, $\phi_3$ is an isomorphism or a link $\phi_3$ of type II or III. We will assume that $\phi_3$ is not an isomorphism, as otherwise we can assume that $\phi_4$ is not an isomorphism and continue the argument below with $\phi_4$ instead of $\phi_3$. 
Since $\phi_3\phi_2$ is not an automorphism by hypothesis, $\phi_3\colon X\dashrightarrow X'$ is a link of type II over $\bP^1$.

\item[(b1.1.i)] If $\phi_3$ has a rational base-point $q$, then $q=\phi_2(p)$, where $p$ is the base-point of $\phi_1^{-1}$, as it is the unique rational base-point of $\phi_r\cdots\phi_3$ by hypothesis, see (b). There exists a link $\phi_2'\colon X'\rightarrow Q'$ of type III to a quadric surface $Q'$. Let $q'\in X'$ be the base-point of $\phi_3^{-1}$. It is a rational point, so there exists a link $f_{12}\colon Q'\dashrightarrow\bP^2$ of type II with base-point $\phi_2'(q')$. The map $\nu := f_{12}\phi_2'\phi_3\phi_2\phi_1\in\BCr_2(k)$ sends the pencil of conics through $p,\phi_1^{-1}(t)$ onto the pencil of conics through the base-point of $f_{12}^{-1}$ and the image by $f_{12}$ of the base-point of $\phi_2^{-1}$, hence belongs to the family \ref{generating_set:1}. The map $\psi\nu^{-1}=\phi_r\cdots\phi_4\phi_2'f_{12}^{-1}$ is a decomposition as in Proposition~\ref{pro:decomposition_links} and we can proceed by induction.  

\item[(b1.1.ii)] Suppose that $\phi_3$ has no rational base-point.
Let $3\leq s\leq r-1$ be the maximal index such that $\phi_i$ is an isomorphism over $\bP^1$ or a link of type II with no rational base-points for all $3\leq i\leq s$ and consider the map $\phi_s\cdots\phi_3\colon X\dashrightarrow X'$. The map $\phi_{s+1}$ is a link of type III  or a link of type II with a rational base-point. If $\phi_{s+1}$ is a link of type II, we proceed as in (b1.1.i) with $\phi_{s+1}\phi_s\cdots\phi_3$ instead of $\phi_3$. If $\phi_{s+1}$ is a link of type III, then $\phi_{s+1}$ is a contraction $X'\rightarrow Q'$ to a quadric surface $Q'$. Recall from (b) that $\phi_2(p)$ is the unique rational base-point of $\phi_r\cdots\phi_3$, where $p$ is the base-point of $\phi_1^{-1}$. There exists a link $f_{12}\colon Q'\dashrightarrow\bP^2$ of type II with base-point $(\phi_{s+1}\phi_s\cdots\phi_3\phi_2)(p)$. The map $\nu:=f_{12}\phi_{s+1}\cdots\phi_1$ sends the pencil of conics through $p,\phi_1^{-1}(t)$ onto the pencil of conics through the base-point of $f_{12}^{-1}$ and the image by $f_{12}$ of the base-point of $\phi_{s+1}^{-1}$. We proceed as in (b1.1.i).

\item[\bf(b1.2)] 
If $\phi_2\in\{f_{77},f_{66}\}$, then $\phi_2$ is, up to an automorphism of $Q$, a birational involution of $Q$ \cite{Isk96}*{Theorem~2.6(ii)}. Recall from (b) that $\phi_1^{-1}$ has a rational base-point $p\in Q$, which is the unique rational base-point of $\phi_r\cdots\phi_2$. There exists a link $f_{12}\colon Q\dashrightarrow\bP^2$ of type II with base-point $\phi_2(p)$.
Then $f_{12}\phi_2\phi_1\in\BCr_2(k)$ and is as in (\ref{gs:2}). Furthermore, $\psi(f_{12}\phi_2\phi_1)^{-1}=\phi_r\cdots\phi_3f_{12}^{-1}$
is a decomposition as in Proposition~\ref{pro:decomposition_links} as the base-point of $f_{12}^{-1}$ is a base-point of $\phi_r\cdots\phi_3f_{12}^{-1}$ by construction.

If $\phi_2=f_{44}\colon Q\dashrightarrow Q'$, let $f_{12}\colon Q'\dashrightarrow\bP^2$ be the link of type II with $\phi_2(p)$ as base-point and $q,q'$ the base-point of $\phi_2,\phi_2^{-1}$, respectively. Then $f_{12}\phi_2\phi_1$ sends the pencil of conics through $\phi_1^{-1}(q)$ onto the pencil of conics through $f_{12}(q')$, so it is a member of \ref{generating_set:2}.

\item[\bf(b1.3)] Suppose that $\phi_2=f_{52}\colon Q\dashrightarrow X_5$, where $X_5$ is a del Pezzo surface of degree $5$. Then $\phi_3$ is one of $f_{33},f_{44},f_{15},f_{25}$ \cite{Isk96}*{Theorem~2.6}.

If $\phi_3\in\{f_{33},f_{44}\}$, then it is a birational self-map of $X_5$ \cite{Isk96}*{Theorem~2.6(ii)}. Let $f_{12}\colon Q\dashrightarrow\bP^2$ be a link of type II with base-point $(\phi_2^{-1}\phi_3\phi_2)(p)$, where $p$ is the (rational) base-point of $\phi_1^{-1}$ according to (b). Then $\nu:=f_{12}\phi_2^{-1}\phi_3\phi_2\phi_1$ is in the family (\ref{gs:3}) and $\psi\nu^{-1}=\phi_r\cdots\phi_4\phi_2 f_{12}^{-1}$ is a decomposition as in Propostion~\ref{pro:decomposition_links}.

If $\phi_3=f_{15}$, then its base-point is $q=\phi_2(p)$ by (b) and so $\phi_3\phi_2\phi_1$ is as in (\ref{gs:4}). 

If $\phi_3=f_{25}$, then it is a map to a quadric surface $Q'$. Let $f_{12}\colon Q'\dashrightarrow\bP^2$ be a link of type II whose base-point is $\phi_3\phi_2(p)$, where $p$ is the (rational) base-point of $\phi_1^{-1}$ according to (b). Then $f_{12}\phi_3\phi_2\phi_1\in\BCr_2(k)$ is as in (\ref{gs:5}), and $\psi(f_{12}\phi_3\phi_2\phi_1)^{-1}=\phi_r\cdots\phi_4 f_{12}^{-1}$ is a decomposition as in Proposition~\ref{pro:decomposition_links}. 

\item[\bf(b1.4)] If $\phi_2=f_{31}\colon Q\dashrightarrow X_6$, then  $\psi\phi_1^{-1}\phi_2^{-1}=\phi_r\cdots\phi_3$
has two rational base-points, namely $\phi_2(p)$ and the base-point $t$ of $\phi_2^{-1}$. Furthermore, $\phi_3$ is a link of type~II of the form $f_{55},f_{44},f_{33},f_{22}$ or $f_{13}$ or a link of type~III to a quadric surface \cite{Isk96}*{Theorem~2.6}. The latter forces $\phi_3\phi_2$ to be an automorphism, which contradicts our hypothesis, see Proposition~\ref{pro:decomposition_links}\ref{decomposition_links:1}.

%\item[(b1.4.i)] 
Suppose that $\phi_3=f_{13}\colon X_6\dashrightarrow Q'$ is a link to a quadric surface $Q'$. As $\psi\phi_1^{-1}\phi_2^{-1}=\phi_r\cdots\phi_3$ has exactly two rational base-points, namely $\phi_2(p)$ and $t$, and the base-point of $q$ of $\phi_3$ is a base-point of $\phi_r\cdots\phi_3$ by hypothesis (see Proposition~\ref{pro:decomposition_links}\ref{decomposition_links:2}), it follows that $q=\phi_2(p)$ or $q=t$. The latter forces $\phi_3\phi_2$ to be an automorphism, which contradicts our hypothesis (see Proposition~\ref{pro:decomposition_links}\ref{decomposition_links:1}), so $q=\phi_2(p)$. Let $f_{12}\colon Q'\dashrightarrow\bP^2$ be a link of type II with base-point $\phi_3\phi_2(t)$. Then $\nu:=f_{12}\phi_3\phi_2\phi_1$ is of the form (\ref{gss:6}) and $\psi\nu^{-1}=\phi_r\cdots\phi_4f_{12}^{-1}$ is as in Proposition~\ref{pro:decomposition_links}. 

%\item[(b1.4.ii)] 
Suppose that $\phi_3\colon X_6\dashrightarrow X_6'$ is one of $f_{55},f_{44},f_{33},f_{22}$. There is a link $f_{13}\colon X_6'\dashrightarrow Q'$ of type~II with base-point $\phi_3(t)$,
and $f_{12}\colon Q'\dashrightarrow\bP^2$ a link of type~II with base-point $f_{13}\phi_3\phi_2(p)$. Then $\nu:=f_{12}f_{13}\phi_3\cdots\phi_1$ is of the form (\ref{gss:7}) and $\psi\nu^{-1}=\phi_r\cdots\phi_4f_{13}^{-1}f_{12}^{-1}$
is a decomposition as in Proposition~\ref{pro:decomposition_links}. By \cite{Isk96}*{Theorem~2.6}, $f_{55}$ and $f_{44}$ can be taken to be birational involutions.

\item[\bf(b2)] Finally, suppose that $\phi_1=f_{51}\colon Q\dashrightarrow X_5$. Then, as $\phi_2$ has no rational base-point by (b), it is a link of type II and hence of the form $f_{44},f_{33},f_{25}$ \cite{Isk96}*{Theorem~2.6}. We proceed as in case (b1.3) with $\phi_2$ instead of $\phi_3$ and construct a map as in (\ref{gss:8}) if $\phi_2=f_{dd}$, $d=3,4$, or the inverse of a map of type (\ref{gs:4}) if $\phi_2=f_{25}$.
\qedhere
\end{enumerate}
\end{enumerate}
\end{proof}

\begin{lemma}
\label{lem:redundant_generators}
In the list in Lemma~\ref{lem:generating_set_BCr}, the generators (\ref{gs:5}) and (\ref{gss:6}), (\ref{gss:7}, $d=2$) and (\ref{gss:8}, $d=4$) are redundant.
\end{lemma}
\begin{proof}
{\bf (\ref{gs:5}):} Consider a map $\psi:=f_{12}f_{25}f_{52}f_{21}$ as in (\ref{gs:5}) and denote by $q_5$ (resp. $q_2$) the base-point of $f_{52}$ (resp. $f_{25}$) and $q_2'$ (resp. $q_5'$) the base-point of $f_{52}^{-1}$ (resp. $f_{25}^{-1}$). We complete the blow-up diagram of $\psi$ given in Lemma~\ref{lem:generating_set_BCr}~(\ref{gs:5}) as follows:
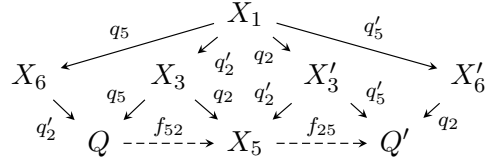
\[
\begin{tikzcd}[column sep=0.2cm,row sep=0.2cm]
&&&X_1\ar[dlll,swap,"q_5"]\ar[dl,"q_2'"]\ar[dr,swap,"q_2"]\ar[drrr,"q_5'"]&&&\\
X_6\ar[dr,swap,"q_2'"]&&X_3\ar[dl,swap,"q_5"]\ar[dr,"q_2"]&&X_3'\ar[dl,swap,"q_2'"]\ar[dr,"q_5'"]&&X_6'\ar[dl,"q_2"]\\ 
&Q\ar[rr,dashed,"f_{52}"]&&X_5\ar[rr,"f_{25}",dashed]&&Q'&
\end{tikzcd}
\]
Thus $\psi$ sends the pencil of conics through the base-point of $f_{21}$ and $f_{21}^{-1}(q_2')$ onto the pencil of conics through the base-point of $f_{12}^{-1}$ and $f_{12}(q_2)$, and is hence in the family \ref{generating_set:1}. 

{\bf(\ref{gss:6}):} Consider a map $\psi:=f_{12}f_{13}f_{31}f_{21}$ as in (\ref{gss:6}) and denote by $q_2,q_3,q_3',q_2'$ the base-point of $f_{21},f_{31},f_{13}^{-1},f_{12}^{-1}$ respectively. We complete the blow-up diagram of $\psi$ given in Lemma~\ref{lem:generating_set_BCr}~(\ref{gss:6}) as follows:
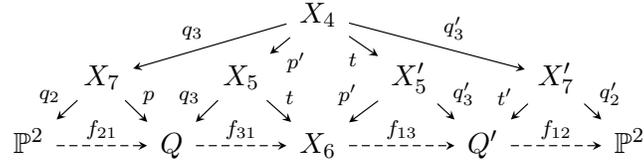
\[
\begin{tikzcd}[column sep=0.2cm,row sep=0.2cm]
&&&&X_4\ar[llld,swap,"q_3"]\ar[dl,"p'"]\ar[drr,"t",swap]\ar[rrrrd,"q_3'"]&&&&\\
& X_7\ar[dl,swap,"q_2"]\ar[dr,"p"]&&X_5\ar[dl,"q_3",swap]\ar[dr,"t"]&&&X_5'\ar[lld,"p'",swap]\ar[dr,"q_3'"]&&X_7'\ar[dl,"t'",swap]\ar[dr,"q_2'"] &\\ 
\bP^2\ar[rr,"f_{21}",dashed]&&Q\ar[rr,dashed,"f_{31}"]&&X_6\ar[rrr,"f_{13}",dashed]&&&Q'\ar[rr,"f_{12} ",dashed]&&\bP^2 
\end{tikzcd}
\]
where $p'=f_{31}(p)$ and $t'=f_{13}(t)$. Let $r_1,r_2$ (resp. $s_1,s_2,s_3$) be the geometric components of $q_2$ (resp. $f_{21}^{-1}(q_3)$). On $X_4$, there are exactly sixteen $(-1)$-curves over the algebraic closure $\overline{k}$ of $k$:
\begin{itemize}
    \item The exceptional divisor of $r_1,r_2$; they make up an orbit of length $2$.
    \item The exceptional divisor of $s_1,s_2,s_3$; they make up an orbit of length $3$.
    \item The strict transform of the conic through $r_1,r_2,s_1,s_2,s_3$, which is rational.
    \item The strict transform of the line through $r_1,r_2$, which is rational.
    \item The strict transform of the line through $s_i,s_j$, $i\neq j$; they make up an orbit of length $3$.
    \item The strict transform of the line through $r_i,s_j$; they make up an orbit of length $6$ whose members are not disjoint.
\end{itemize}
It follows that the blow-up of $q_2,q_2'$ is redundant and $\psi=f_{33}$.

{\bf(\ref{gss:7}, $d=2$):} Consider a map $\psi:=f_{12}f_{13}f_{22}f_{31}f_{21}$ as in (\ref{gss:7}) and denote by $q_3,q_2,q_2',q_3'$ the base-points of $f_{31},f_{22},f_{22}^{-1},f_{13}^{-1}$ respectively. We complete the blow-up of $\psi$ given in Lemma~\ref{lem:generating_set_BCr}~(\ref{gss:7}) as follows:
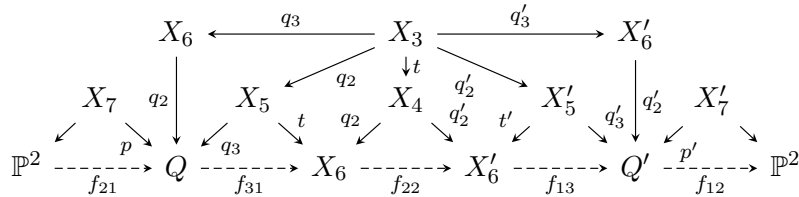
\[
\begin{tikzcd}[column sep=0.2cm,row sep=0.2cm]
&&X_6\ar[dd,"q_2",swap]&&&X_3\ar[lll,"q_3",swap]\ar[lld,"q_2"]\ar[d,"t"]\ar[rrd,"q_2'",swap]\ar[rrr,"q_3'"]&&&X_6'\ar[dd,"q_2'"]&&\\
& X_7\ar[dl]\ar[dr,"p",swap]&&X_5\ar[dl,"q_3"]\ar[dr,"t"]&&X_4\ar[ld,"q_2",swap]\ar[dr,"q_2'"]&&X_5'\ar[dl,"t'",swap]\ar[dr,"q_3'"]&&X_7'\ar[dl,"p' "]\ar[dr] &\\ 
\bP^2\ar[rr,"f_{21}",dashed,swap]&&Q\ar[rr,dashed,"f_{31}",swap]&&X_6\ar[rr,"f_{22}",dashed,swap]&&X_6'\ar[rr,"f_{13}",dashed,swap]&&Q'\ar[rr,"f_{12} ",dashed,swap]&&\bP^2 
\end{tikzcd}
\]
where $p'=(f_{13}f_{22}f_{31})(p)$ and $t'=f_{22}(t)$. Thus $\psi$  belongs to the family \ref{generating_set:1}. 

{\bf(\ref{gss:8}, $d=4$):} Consider a map $\psi:=f_{15}f_{44}f_{51}$ as in (\ref{gss:8}). Let $q_4,q_4',q_5,q_5'$ be the base-point of $f_{44},f_{44}^{-1},f_{51},f_{15}$, respectively. We complete the blow-up of $\psi$ given in Lemma~\ref{lem:generating_set_BCr}~(\ref{gss:8}) as follows, where $Y$ is the blow-up of $X_1$ at the point $p$, and is not a del Pezzo surface:
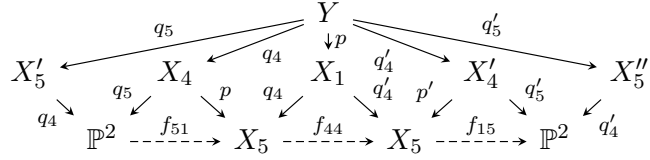
\[
\begin{tikzcd}[column sep=0.2cm,row sep=0.2cm]
&&&&Y\ar[lllld,"q_5",swap]\ar[lld,"q_4"]\ar[d,"p"]\ar[rrd,"q_4'",swap]\ar[rrrrd,"q_5'"]&&&&\\
X_5'\ar[rd,"q_4",swap]&& X_4\ar[dl,swap,"q_5"]\ar[dr,"p"]&&X_{1}\ar[dl,"q_4",swap]\ar[dr,"q_4'"]&&X_4'\ar[ld,"p'",swap]\ar[dr,"q_5'"]&&X_5''\ar[ld,"q_4'"]\\ 
&\bP^2\ar[rr,"f_{51}",dashed]&&X_5\ar[rr,dashed,"f_{44}"]&&X_5\ar[rr,"f_{15}",dashed]&&\bP^2& 
\end{tikzcd}
\]
where $p'=f_{dd}(p)$. With Lemma~\ref{lem:Schneider}, we obtain that $\psi$ is in the family \ref{generating_set:2}.
\end{proof}

\begin{proof}[Proof of Theorem~\ref{thm:generatorIntro}]
We compare the list of generators in \cite{Isk91} contained in $\BCr_2(k)$ with the list of generators in Lemma~\ref{lem:generating_set_BCr}, and see that the two lists coincide, if we replace ``preserving the pencil of conics through a point of degree $4$ (resp. two points of degree $2$)" by ``sending the pencil of conics trough a point of degree $4$ (resp. two points of degree $2$) onto a pencil of conics of the same kind" in \cite{Isk91}:
\medskip

    \begin{tabular}{c|c|c|c|c|c}
         Lemma~\ref{lem:generating_set_BCr}&\ref{generating_set:1}&\ref{generating_set:2}&(\ref{gs:1})&(\ref{gs:2})&(\ref{gs:3})  \\ \hline
         \cite{Isk91}&A11 &(15),(20) &(7),(8),(19'),(15''') &(10),(11) &(12),(13) 
    \end{tabular}\\ 
    \smallskip

\indent \begin{tabular}{c|c|c|c|c|c}
        Lemma~\ref{lem:generating_set_BCr}&(\ref{gs:4})&(\ref{gs:5})&(\ref{gss:6})&(\ref{gss:7})&(\ref{gss:8})  \\ \hline
         \cite{Isk91}&A17 &(14)&(19') &(16),(17),(11''),(18)&(21),(22)
    \end{tabular}
    
\medskip    
\noindent while type (9), (9'), (11'), (15'), (15''), (19) from \cite{Isk91} are not contained in $\BCr_2(k)$. Note that (\ref{gss:6}) is covered by (19') by Lemma~\ref{lem:redundant_generators}. 
\end{proof}

%----------Revisiting the parity problem
\subsection{Revisiting the parity problem}
\label{subsect:parityProblemRevisit}

Now let us prove that all generators given in Lemma~\ref{lem:generating_set_BCr} induce even permutations when the ground field is $k=\bF_{2^m}$ for $m\geq2$.

%-----Parities of f_{33}, f_{77}, and f_{88} in (5.1)
\subsubsection{Parities of $f_{33}$, $f_{77}$, and $f_{88}$ in (\ref{gs:1})}
\label{subsubsect:parityGenerator}

Up to automorphisms of $\bP^2$, the maps $f_{77}$ and $f_{88}$ are Geiser and Bertini involutions respectively given by equations \eqref{eqn:geiser} and \eqref{eqn:bertini}. By Theorem~\ref{thm:delPezzo}, they induce even permutations on $\bP^2(\bF_q)$ for $q = 2^m\geq 4$. On the other hand, the map $f_{33}$ is a quadratic transformation, that is, a Cremona map defined by the linear system of conics passing through three non-collinear points in $\bP^2$.

\begin{lemma}
\label{lem:quad-invol}
Let $k$ be any field, $f\in\BCr_2(k)$ be a quadratic transformation and $\tau\in\Cr_2(k)$ be the standard quadratic involution $[x:y:z]\mapsto [yz:xz:xy]$.
\begin{enumerate}[label=\textup{(\arabic*)}]
    \item\label{be-invol}
    There exists $g\in\PGL_3(k)$ such that the composition $gf$ is involutive.
    \item\label{be-std}
    If $f$ is involutive, then there exists $h\in\PGL_3(\overline{k})$ such that $\tau = h^{-1}fh$.
\end{enumerate}
\end{lemma}

\begin{proof}
There exists an extension $k'/k$ of degree $3$ and a generator $\sigma\in\Gal(k'/k)\cong\bZ/3\bZ$ such that
$$
    \Bs(f) = \{a,a^\sigma,a^{\sigma^2}\}
    \quad\text{for some}\quad
    a\in\bP^2(k').
$$
Since $f$ is given by blowing up $\{a,a^\sigma,a^{\sigma^2}\}$ and then contracting the three lines passing through these points, the indeterminacy locus of $f^{-1}$ is a Galois orbit for the same extension $k'/k$, namely,
$$
    \Bs(f^{-1}) = \{b,b^\sigma,b^{\sigma^2}\}
    \quad\text{for some}\quad
    b\in\bP^2(k').
$$
For every point $x = [x_0,x_1,x_2]\in\bP^2$, we define
$$
    g_x\colonequals\begin{pmatrix}
        x_0 & x_0^{\sigma} & x_0^{\sigma^2}\\
        x_1 & x_1^{\sigma} & x_1^{\sigma^2}\\
        x_2 & x_2^{\sigma} & x_2^{\sigma^2}
    \end{pmatrix}
$$
to be a linear map that sends the coordinate points $[1:0:0]$, $[0:1:0]$, $[0:0:1]$ to the Galois orbit points $x$, $x^{\sigma}$, $x^{\sigma^2}$, respectively. Note that $g_x$ is invertible when $x$, $x^{\sigma}$, $x^{\sigma^2}$ are not collinear. Let $g\colonequals g_ag_b^{-1}$, which can be easily verified to be defined over $k$. Then $gf$ is involutive as the indeterminacy loci of this map and its inverse both coincide with $\{a,a^\sigma,a^{\sigma^2}\}$. This proves \ref{be-invol}.

Assume that $f$ is involutive, or equivalently, that
$
    \Bs(f) = \{a,a^\sigma,a^{\sigma^2}\} = \Bs(f^{-1}).
$
Let $h = g_a$. Then the indeterminacy loci of $h^{-1}fh$ and its inverse both consist of the three coordinate points. This implies that $\tau = h^{-1}fh$ and thus proves \ref{be-std}.
\end{proof}

Recall that, for every $n\geq 1$, the \emph{standard involution} $\tau\colon\bP^n\dashrightarrow\bP^n$ is defined by
$$
    \tau([x_0:\dots:x_n]) = [\tau_0:\dots:\tau_n]
    \quad\text{where}\quad
    \tau_i = \prod_{j\neq i}x_j.
$$
In terms of the affine coordinates $(\xi_1,\dots,\xi_n)$ where $\xi_i = x_i/x_0$, this map is written as
$$
    \tau(\xi_1,\dots,\xi_n) = (\xi_1^{-1},\dots,\xi_n^{-1}).
$$
From this expression, one can deduce that the fixed locus of $\tau$ consists of points of the form $[\pm 1:\dots:\pm 1]$. Note that these are the same point in characteristic~$2$. In the following, we prove a general fact about bijective Cremona transformations of $\bP^n$ that are conjugate to $\tau$ by automorphisms, then use it to compute the parity induced by $f_{33}$.

\begin{lemma}
\label{lem:std-invol}
Let $n\geq 1$, $k = \bF_{2^m}$, and $f\in\BCr_n(k)$ be an involutive quadratic transformation. If there exists $h\in\PGL_{n+1}(\overline{k})$ such that $h^{-1}fh$ equals the standard involution $\tau$, then the permutation induced by $f$ on $\bP^n(k)$ is odd when $m=1$ and even when $m\geq 2$.
\end{lemma}

\begin{proof}
The relation $\tau = h^{-1}fh$ implies that a point $x\in\bP^n(\overline{k})$ is fixed by $\tau$ if and only if $h(x)$ is fixed by $f$. Because the fixed locus of $\tau$ consists of a single point $[1:\dots:1]$, the fixed locus of $f$ consists of a single point $y\in\bP^n(\overline{k})$ as well. If $y\notin\bP^n(k)$, then $f$ acts on $\bP^n(k)$ as an involution without a fixed point. This implies that the number of rational points
$$
    |\bP^n(k)| = |\bP^n(\bF_{2^m})| = (2^m)^n + \dots + 2^m + 1
$$
is even, contradiction. Hence $y\in\bP^n(k)$, and the action of $f$ on $\bP^n(k)$ is a composition of
$$
    \frac{1}{2}(|\bP^n(k)| - 1) = \frac{1}{2}((2^m)^n + \dots + 2^m)
$$
many transpositions. The last integer is odd if $m=1$ and even if $m\geq 2$, so the result follows.
\end{proof}

\begin{prop}
\label{prop:quadratic}
Let $k=\bF_{2^m}$ with $m\geq 2$. Assume that $f\in\BCr_2(k)$ is of type $f_{33}$. Then $f$ acts on $\bP^2(k)$ as an even permutation.
\end{prop}

\begin{proof}
By Lemma~\ref{lem:quad-invol}, there exists $g\in\PGL_3(k)$ and $h\in\PGL_3(\overline{k})$ such that $h^{-1}gfh$ is the standard quadratic involution. It follows from Lemma~\ref{lem:std-invol} that $gf$ acts on $\bP^2(k)$ as an even permutation. Since $g$ acts on $\bP^2(k)$ evenly by Proposition~\ref{prop:PGLpermute}, the result follows.
\end{proof}

%-----Parities of the generators (4.2) to (4.8)
\subsubsection{Parities of the generators (\ref{gs:2}) to (\ref{gss:8})}
\label{subsubsect:symmetric_generators}

Any birational map $f\in\BCr_2(k)$ which over $\overline{k}$ is a Geiser involution (resp. Bertini involution) up to an element of $\PGL_3({k})$ lifts to an automorphism of a del Pezzo surface of degree $2$ (resp. degree $1$).
In fact, the geometric description of $f$ is analogous to the one of the Geiser involution (resp. Bertini involution) over $\overline{k}$ and to the Geiser involution (resp. Bertini involution) over $k$ with only one base-point. 
It yields directly that $f$ lifts to an automorphism of a del Pezzo surface of degree $2$ (resp. degree $1$). Hence, $f$ induces an even permutation by Theorem~\ref{thm:delPezzo}.

{\it Generator (\ref{gs:2}), (\ref{gs:3}), or (\ref{gss:7}, $d=4,5$):} Let $f$ be the corresponding birational map. Note that we can take $f_{dd}$ in the respective generator to be an involution, so that geometrically $f_{dd}$ is either a Geiser or Bertini involution, which induces an even permutation. Upon applying an automorphism of $\bP^2$ or $Q$, we can assume that $f$ is conjugate to $f_{dd}$. Hence, $f$ also induces an even permutation by Theorem~\ref{thm:parityBirMapIntro}.

{\it Generator (\ref{gs:4}):} 
Let $q_2$ be a point of degree $2$ and $q_5$ a point of degree $5$, both in general positions. Over $\overline{k}$ there are exactly two cubic curves passing through $q_5,q_2$ with a double point at one of the points of $q_2$, and we call $C_2$ its orbit over $k$. Similarly, there are exactly five cubic curves with a double point at one of the points of $q_5$, and we call $C_5$ its orbit over $k$. 
We complete the blow-up diagram of $f=f_{15}f_{52}f_{21}$. By abuse of notation we write $p$ for $f_{21}(L)$, $f_{52}(p)$ and their image in $X_3$.
In $X_3$ there are exactly two curves which over $\overline{k}$ are orbits of disjoint $(-1)$-curves of length $2$ and $5$, namely the strict transforms of $C_2$ and $C_5$, denoted by $\widetilde{C}_2$ and $\widetilde{C}_5$. 
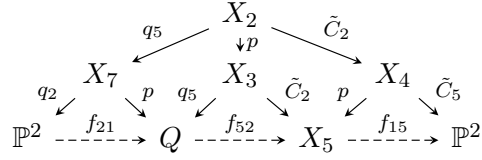
\[
\begin{tikzcd}[column sep=0.2cm,row sep=0.2cm]
& && X_2\ar[dll,swap,"q_5"] \ar[d,"p"] \ar[drr,"\widetilde{C}_2"]&&&\\
& X_7\ar[dl,swap,"q_2"]\ar[dr,"p"]&&X_3\ar[dl, swap,"q_5"]\ar[dr,"\widetilde{C}_2"]&&X_4\ar[ld,"p",swap]\ar[dr,"\widetilde{C}_5"]&\\ 
\bP^2\ar[rr,"f_{21}",dashed]&&Q\ar[rr,dashed,"f_{52}"]&&X_5\ar[rr,"f_{15}",dashed]&&\bP^2 
\end{tikzcd}
\]
The blow-up diagram of $f$ shows that $f$ has the same geometric description as a Geiser involution over $k$ with base-points $q_2$ and $q_5$. Thus, up to composition by an element of $\PGL_3(k)$, $f$ lifts to an automorphism of the del Pezzo surface $X_2$.  
Now Theorem~\ref{thm:delPezzo} and Proposition~\ref{prop:PGLpermute} imply that $f$ induces en even permutation over $k=\bF_q,q=2^m\geq4$.

{\it Generator (\ref{gs:5})} By Lemma~\ref{lem:redundant_generators}, this map is, up to an automorphism of $\bP^2$, a member of the family \ref{generating_set:1} and hence induces an even permutation for $k=\bF_{2^m}$, $m\geq2$ by Corollary~\ref{cor:conicBundle_v2}. 

{\it Generator (\ref{gss:6})} By Lemma~\ref{lem:redundant_generators}, this generator is equal to $f_{33}$, so is treated in Proposition~\ref{prop:quadratic}.

{\it Generator (\ref{gss:7}, $d=3$)} We can complete the blow-up diagram as in Lemma~\ref{lem:redundant_generators} to get
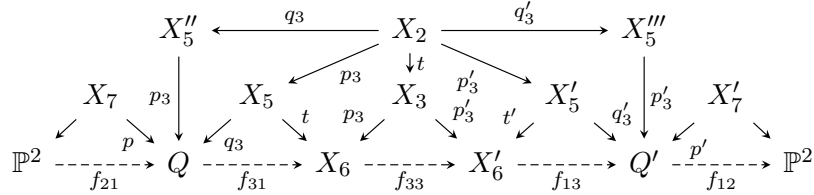
\[
\begin{tikzcd}[column sep=0.2cm,row sep=0.2cm]
&&X_5''\ar[dd,"p_3",swap]&&&X_2\ar[lll,"q_3",swap]\ar[lld,"p_3"]\ar[d,"t"]\ar[rrd,"p_3'",swap]\ar[rrr,"q_3'"]&&&X_5'''\ar[dd,"p_3'"]&&\\
& X_7\ar[dl]\ar[dr,"p",swap]&&X_5\ar[dl,"q_3"]\ar[dr,"t"]&&X_3\ar[ld,"p_3",swap]\ar[dr,"p_3'"]&&X_5'\ar[dl,"t'",swap]\ar[dr,"q_3'"]&&X_7'\ar[dl,"p' "]\ar[dr] &\\ 
\bP^2\ar[rr,"f_{21}",dashed,swap]&&Q\ar[rr,dashed,"f_{31}",swap]&&X_6\ar[rr,"f_{33}",dashed,swap]&&X_6'\ar[rr,"f_{13}",dashed,swap]&&Q'\ar[rr,"f_{12} ",dashed,swap]&&\bP^2 
\end{tikzcd}
\]
where $q_3,p_3,q_3',p_3'$ are the base-points of $f_{31},f_{33},f_{13},f_{33}^{-1}$ respectively. Hence, the composition $f_{13}f_{33}f_{31}$ is geometrically a Geiser involution. Hence the permutation induced on $Q\dashrightarrow Q$ is even. Since $f=f_{12}f_{13}f_{33}f_{31}f_{21}$ is conjugate to $f_{13}f_{33}f_{31}$ (upon applying automorphism of $\bP^2$), $f$ also induces an even permutation by Theorem~\ref{thm:parityBirMapIntro}.

{\it Generator (\ref{gss:8})} The case $d=4$ follows from Lemma~\ref{lem:redundant_generators}. If $d=3$, we have the blow-up diagram,
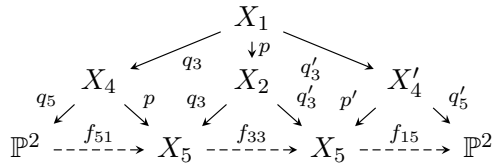
\[
\begin{tikzcd}[column sep=0.2cm,row sep=0.2cm]
&&&X_1\ar[lld,"q_3"]\ar[d,"p"]\ar[rrd,"q_3'",swap]\\
&X_4\ar[dl,swap,"q_5"]\ar[dr,"p"]&&X_{2}\ar[dl,"q_3",swap]\ar[dr,"q_3'"]&&X_4'\ar[ld,"p'",swap]\ar[dr,"q_5'"]\\ 
\bP^2\ar[rr,"f_{51}",dashed]&&X_5\ar[rr,dashed,"f_{33}"]&&X_5\ar[rr,"f_{15}",dashed]&&\bP^2
\end{tikzcd}
\]
where $q_3,q_3',q_5,q_5'$ are the base-points of $f_{33},f_{33}^{-1},f_{51},f_{15}$ respectively. Hence, $f=f_{15}f_{33}f_{51}$ is a Bertini involution, so $f$ induces an even permutation.

%%%%%%%%%%%%%%
\subsubsection{Parity of the generator $f_{66}$ in (\ref{gs:1})}
\label{subsubsect:quinticTransform}

We finally prove that the remaining generator, namely $f_{66}\colon\bP^2\dashrightarrow\bP^2$, induces a permutation of even parity on $\bP^2(\bF_{2^m})$ for $m\geq2$.

\begin{lemma}[\cite{LS21}*{Lemma~4.20}]
\label{lem:norationalpoint}
Let $p_1,\dots,p_6$ be a point of degree $6$ in $\bP^2$ over $\bF_q$ such that $p_1,\dots,p_6$ are in general position. Then at least $q^2+q$ rational points of $\bP^2$ are in general position with $p_1,\dots,p_6$.
\end{lemma}

\begin{proof}
Let $\sigma$ be the generator of $\mathrm{Gal}(\bF_{q^6}/\bF_q)$ and suppose that $\sigma^i(p_1)=p_i$ for $i=1,\dots,6$.
Let $L_{ij}$ be the line through $p_i,p_j$, and let $r$ be a rational point of $\bP^2$ that is not on the intersection of $L_{14},L_{25},L_{36}$.
The lines through the $p_1,...,p_6$ make up three orbits, namely the orbit of $L_{12}$, $L_{13}$ and $L_{14}$. We check that $r$ is not contained in one of these three lines, from which it follows that $r$ is not on any of the $L_{ij}$. If $r\in L_{12}$, then $L_{23}=\sigma(L_{12})$ contains $r,p_2$, so $L_{23}=L_{12}$, which is impossible. If $r\in L_{13}$, then $r,p_3$ are both contained in $\sigma^2(L_{13})=L_{35}$, which is again impossible. If $r\in L_{14}$, then $r$ is also contained in $\sigma(L_{14})=L_{25}$ and $\sigma^2(L_{14})=L_{36}$, which contradicts our choice of $r$. 
Finally, if $p_{1},..,p_{5},r$ lie on a conic $C$, then $\sigma(C)$ and $C$ contain 5 common points and hence are equal, which is impossible.
\end{proof}

\begin{lemma}
\label{lem:allbutonedegree2point}
Suppose $p_1,\dots,p_6$ make up a point of degree $6$ in $\bP^2$ over $\bF_{q}$ such that no three are collinear and let $L$ be the line through $p_1$ and $p_2$. Under the action of $\mathrm{Gal}(\bF_{q^6}/\bF_q)$, there is at most one point $r\in L$ whose orbit in $\bP^2$ is of length $2$. In this case, $r$ and its Galois conjugate form the only point of degree $2$ contained in the orbit of $L$.
\end{lemma}

\begin{proof}
Consider $i$ as an integer modulo $6$ and let
\begin{itemize}
    \item $\sigma$ to be the generator of $\mathrm{Gal}(\bF_{q^6}/\bF_q)$ such that $\sigma^i(p_1)=p_{i+1}$, and
    \item $L_{p_ip_{i+1}}$ to be the line through $p_i$ and $p_{i+1}$ so that $L=L_{p_1p_2}$.
\end{itemize}
Suppose that there exists $r\in L_{p_1p_2}$ such that $\{r,\sigma(r)\}$ form a point of degree~$2$ in $\bP^2$. Then
$$
    r\in L_{p_1p_2}\cap L_{p_3p_4}\cap L_{p_5p_6}
    \quad\text{and}\quad
    \sigma(r)\in L_{p_2p_3}\cap L_{p_4p_5}\cap L_{p_6,p_1}.
$$
In particular, $\{r,\sigma(r)\}$ is contained in the orbit of $L_{p_1p_2}$. If $L_{p_1p_2}$ contains another point $s$ whose orbit is of length~$2$. Then $L_{p_1p_2}\cap L_{p_3p_4}$ contains both $r$ and $s$, thus $L_{p_1p_2} = L_{p_3p_4}$, which contradicts the hypothesis that no three of the $p_i$'s are collinear.
\end{proof}

\begin{lemma}
\label{lem:nodalcubicrational}
Let $C\subset\bP^2$ be a singular cubic over an arbitrary field $k$. Then $C$ is rational, that is, its normalization $\widetilde{C}$ is isomorphic to $\bP^1$ over $k$.
\end{lemma}

\begin{proof}
By Ch\^{a}telet's theorem, $\widetilde{C}\cong\bP^1$ over $k$ if and only if $\widetilde{C}$ contains a $k$-point. This is always the case when $C$ is a cuspidal cubic. Suppose that $C$ is a nodal cubic and let $p\in C$ be the node. The linear system of lines passing through $p$ is isomorphic to $\bP^1$ over $k$. Note that $|\bP^1(k)|\geq 3$ for any field $k$. Since the tangent cone at $p$ contributes at most two elements to $\bP^1(k)$, there exists a line $\ell\in\bP^1(k)$ such that $\ell\cap C = \{p,p'\}$ for some $k$-point $p'\neq p$. The point $p'$ induces a $k$-point on the normalization $\widetilde{C}$, so the proof is done.
\end{proof}

\begin{lemma}
\label{lem:notonaconic}
Let $q=2^m\geq2$ and suppose $p_1,\dots,p_6$ make up a point of degree $6$ in $\bP^2$ over $\bF_{q}$ contained in a singular cubic $C$. Then there are at least $\frac{1}{2}(q^2-2q-2)$ points of degree $2$ on $C$ that are not on a conic with $p_1,p_2,p_4,p_5$.
\end{lemma}

\begin{proof}
There is an involution $\sigma$ on $C$ which maps a general $x\in C(\overline{\bF_q})$ to the residual intersection of the conic passing through $p_1,p_2,p_4,p_5,x$ with $C$. Let $z\in C$ be the singular point and $\tau\colon\widetilde{C}\to C$ be the normalization. Then $\sigma$ lifts to an involution $\widetilde{\sigma}$ on $\widetilde{C}$ which preserves the set $\tau^{-1}(z)\subset\widetilde{C}$. Notice that $\widetilde{C}\cong\bP^1$ by Lemma~\ref{lem:nodalcubicrational}. Then an elementary computation shows that $\widetilde{\sigma}$, up to conjugation over $\bF_{q^6}$, acts on $\widetilde{C}$ as $x\mapsto x+a$ for some $a\in\bF_{q^6}$.

If $C$ is a nodal cubic, the total number of points of degree~$2$ on $C$ is given by
$$
    \frac{1}{2}|C(\bF_{q^2})\setminus C(\bF_{q})|
    = \begin{cases}
    \frac{1}{2}(q^2-q) & \text{if }\tau^{-1}(z)\text{ consists of two } \bF_q\text{-points},\\
    \frac{1}{2}(q^2-q-2) & \text{if }\tau^{-1}(z)\text{ is a point of degree }2\text{ over }\bF_q.
    \end{cases}
$$
If $C$ is a cuspidal cubic, the total number of points of degree~$2$ on $C$ is given by
$$
    \frac{1}{2}|C(\bF_{q^2})\setminus C(\bF_{q})|=\frac{1}{2}(q^2-q).
$$
Pick any $x\in C(\bF_{q^2})\setminus C(\bF_q)$. Then $x$ and its conjugate $x^q$ lie on a conic with $p_1,p_2,p_4,p_5$ if and only if $x^q = x+a$. The last equation has at most $q$ distinct solutions in $x$, so the number of degree-$2$ points on $C$ lying on a conic with $p_1,p_2,p_4,p_5$ is at most $\frac{1}{2}q$. As a consequence, at least
$$
    \frac{1}{2}(q^2-q-2) - \frac{1}{2}q
    = \frac{1}{2}(q^2-2q-2)
$$
many points of degree~$2$ on $C$ do not lie on a conic with $p_1,p_2,p_4,p_5$.
\end{proof}

\begin{lemma}
\label{lem:cubicsurface66}
Let $q=2^m\geq4$. Let $p$ be a point of degree $6$ in $\bP^2$ over $\bF_q$ such that its blow-up is a del Pezzo surface. Then there exists at least one point $r$ of degree $2$ in $\bP^2$ such that the blow-up at $p,r$ is still a del Pezzo surface (i.e. $p,r$ are in general position). 
\end{lemma}

\begin{proof}
Choose a generator $\sigma$ for $\mathrm{Gal}(\bF_{q^6}/\bF_q)$ and let $p_1,\dots,p_6$ be the orbit making up $p$ such that $\sigma(p_i) = p_{i+1}$ for each $i$ modulo $6$. In the following, we prove that there exists a point $r = \{r_1,r_2\}$ of degree~$2$ in $\bP^2$ such that
\begin{itemize}
    \item no three of the eight points $p_1,\dots,p_6,r_1,r_2$ are on a line,
    \item no six of them are on a conic, and
    \item no eight of them are on a nodal cubic with one being the double point.
\end{itemize}

Let $r = \{r_1,r_2\}$ be a point of degree~$2$ in $\bP^2$ such that $r_1$ (resp. $r_2$) is not collinear with any two consecutive $p_i$'s. Let $L_{ij}$ be the line through $p_i,p_j$. The lines through the $p_1,...,p_6$ make up three orbits, namely the orbit of $L_{12}$, $L_{13}$ and $L_{14}$. By Lemma~\ref{lem:allbutonedegree2point} there is at most one point of degree~$2$ in the orbit of $L_{12}$, and we choose $r$ to be outside of the orbit of $L_{12}$. Note that the line through $r$ is rational, so it cannot contain any $p_i$. 
Suppose that $r_1\in L_{13}$. Then $r_1\in \sigma^2(L_{13})=L_{35}$ and thus $\sigma^2(L_{13})\cap L_{13}$ contains $p_3,r_1$. This implies $\sigma(L_{13})=L_{13}$, which is against our hypothesis. Suppose that $r_1\in L_{14}$. Then $r_2\in\sigma^3(L_{14})=L_{14}$ and hence $L_{14}=\sigma^3(L_{14})$ is the line through $r$, which is impossible as we have already explained.

Suppose that $p_1,\dots,p_4,r_1,r_2$ are on a conic $C$. Then $\sigma(C)\cap C$ contains $p_2,p_3,p_4,r_1,r_2$, hence $C=\sigma(C)$, that is, $C$ is invariant under $\mathrm{Gal}(\bF_{q^6}/\bF_q)$. This implies that $C$ contains $p$, which is against our hypothesis. Suppose that $p_1,\dots,p_5,r_1$ are on a conic $C$. Then $\sigma^2(C)$ passes through $p_3,p_4,p_5,p_6,p_1,r_1$. We have $C\cap\sigma^2(C)$ contains $p_1,p_3,p_4,p_5,r_1$ and hence $\sigma^2(C)=C$, which is impossible as $C$ does not contain $p_6$.
To finish the conic case, recall from Lemma~\ref{lem:norationalpoint} that there is a rational point $s$ in $\bP^2$ such that $s,p_1,\dots,p_6$ are in general position. There exists a singular cubic containing $p_1,\dots,p_6$ with $s$ its singular point. By Lemma~\ref{lem:notonaconic}, there are at least $\frac{1}{2}(q^2-2q-2)\geq3$ points of degree $2$ not on a conic with $p_1,p_2,p_4,p_5$. We can choose $r_1,r_2$ to be one of them.

Finally, if there is a nodal cubic $C$ through the eight points $p_1,\dots,p_6,r_1,r_2$ with one of them its double point, then $\sigma(C)\neq C$ and $C\cdot\sigma(C)\geq10$, which is impossible.
\end{proof}

\begin{rmk}
Let $p,r$ be points in $\bP^2$ of degree $6$ and $2$ such that their blow-up is a del Pezzo surface. 
On can describe the Bertini involution on this surface in a very nice way: 
Let $S$ be the blow-up of $p$ and view it as cubic surface in $\bP^3$. We now can view $r$ as a point of $\bP^3$, and denote by $L\subset\bP^3$ the line passing through $r$. 
We claim that $L$ is not a $(-1)$-curve on $S$. Indeed, the $27$ lines on $S$ are the six exceptional divisors of the components $p_1,\dots,p_6$ of $p$, the $15$ strict transforms of the lines through two of the $p_i$, and the $6$ strict transforms of the conics passing through five of the $p_i$. None of these curves is defined over $\bF_q$, while $L$ is defined over $\bF_q$. 
So, the line $L$ intersects $S$ transversely in $r$ and a rational point $s$. 
The planes in $\bP^3$ containing $L$ induces an elliptic fibration on $S$, or more precisely, on the blow-up of $S$ at $r,s$, where the exceptional curve of $r$ defines a zero section. In particular, the Bertini involution can be defined as it is the multiplication by $-1$ using the group law on the generic fiber.
\end{rmk}

\begin{prop}[\cite{LS21}*{Lemma~4.12~(2)}]
\label{prop:standardquinticseven}
Assume that $m\geq2$ and $q=2^m\geq4$. 
Then any link $f_{66}\colon\bP^2\dashrightarrow\bP^2$ induces an even permutation on $\bP^2(\bF_q)$.
\end{prop}
\begin{proof}
Let $p$ be the base-point of degree $6$ of $f_{66}$. 
By Lemma~\ref{lem:cubicsurface66}, there exists a point $r$ of degree $2$ such that the blow-up at $r,p$ is a del Pezzo surface $T$. Denote respectively by $E_1,E_2$ and $E_1',\dots,E_6'$ the geometric components of their exceptional divisors. Let $L$ be the pullback of the class of a line in $\bP^2$. Then the only orbits of $(-1)$-curves in $T$ of length at most $8$ with pairwise disjoint members are as follows:
\begin{alignat*}{3}
&\red{E}&&\colonequals\left\{E_1,E_2\right\},\\
&\blue{E'}&&\colonequals\left\{E_1',\dots,E_6'\right\},\\
&\ell&&\colonequals\left\{L-E_1-E_2\right\},\\
&\blue{C}&&\colonequals\left\{
    2L-\sum\nolimits_{j\in\{1,\dots,6\}\setminus\{i\}}E_{j}' \;\middle|\; i=1,\dots,6
\right\},\\
&\blue{F}&&\colonequals\left\{ 
    4L-2E_1-2E_2-2E_{i}'-\sum\nolimits_{j\in\{1,\dots,6\}\setminus\{i\}}E_{j}' \;\middle|\; i=1,\dots,6
\right\},\\
&D&&\colonequals\left\{
    5L-E_1-E_2-2\left(
        \sum\nolimits_{j=1}^6E_j'
    \right)
\right\},\\
&\red{S}&&\colonequals\left\{
    6L-3E_i-2E_{3-i}-2\left(
        \sum\nolimits_{j=1}^6E_j'
    \right)\;\middle|\; i=1,2
\right\},\\
&\blue{S'}&&\colonequals\left\{
    6L-2E_1-2E_2-3E_{i}'-2\left(
        \sum\nolimits_{j\in\{1,\dots,6\}\setminus\{i\}}E_{j}'
    \right)\;\middle|\; i=1,\dots,6
\right\}.
\end{alignat*}
Drawing all possible blow-downs from $T$ over $\bF_{q}$, we obtain the following commutative diagram, where the arrows are denoted by the set of $(-1)$-curves they contract.

\begin{center}
\begin{tikzpicture}[baseline=(a).base]
\node[scale=1](a) at (0,0){
\begin{tikzcd}[link3]
&\bP^2\ar[rr,dashed,-]\ar[ldd,dashed,-]&&\bP^2\ar[ddr,dashed,"\tau_2"]&\\ 
&&&&\\
Q&Y_3\ar[l,"D"]\ar[uu,"\red{S}"]&Y_3'\ar[uul,"\blue{S'} ",swap]\ar[uur,"\blue{F}"]&Y_3''\ar[uu,"\red{S}"]\ar[r,"\ell"]&Q\\
&&&&\\
&Y_2\ar[ddl,"\blue{C}",swap]\ar[uul,"\blue{S'} "]&T\ar[l,"D",swap]\ar[uul,"\blue{S'}",swap]\ar[uu,"\red{S}"]\ar[uur,"\blue{F}"]\ar[r,"\ell"]\ar[ddr,"\blue{E'}",swap]\ar[dd,"\red{E}"]\ar[ddl,"\blue{C}"]&Y_2'\ar[uur,"\blue{F}",swap]\ar[ddr,"\blue{E'} "]&\\
&&&&\\
Q\ar[uuuu,dashed,-]&Y_1\ar[l,"D"]\ar[dd,"\red{E}"]&Y_1'\ar[ddl,"\blue{C}"]\ar[ddr,"\blue{E'}",swap]&Y_1''\ar[dd,"\red{E} "]\ar[r,"\ell"]&Q\ar[uuuu,dashed,"g",swap]\\
&&&&\\
&\bP^2\ar[rr,dashed,"f_{66}",swap]\ar[uul,dashed,-]&&\bP^2\ar[ruu,dashed,"\tau_1",swap]&
\end{tikzcd}};
\end{tikzpicture}
\end{center}

The Bertini involution $\beta\in\mathrm{Aut}(T)$ acts on the set $\{E,E',\ell,C,F,D,S,S'\}$, and it does not preserve any of them. It is thus a rotation of order $2$, and it exchanges the rational curves $\ell,D$.
So, $\beta$ is the birational map corresponding to the path of arrows from the lower left $\bP^2$ to the upper right $\bP^2$, that is,
$$
    \beta
    = \varepsilon
    \circ\tau_2^{-1}
    \circ g
    \circ \tau_1
    \circ f_{66}
    \quad\text{for some}\quad
    \varepsilon\in\mathrm{PGL}_3(\bF_q).
$$
By Proposition~\ref{prop:PGLpermute}, $\varepsilon$ induces even parity on $\bP^2(\bF_q)$. By Theorem~\ref{thm:delPezzo}, the automorphism $\beta$ induces an even permutation on $T(\bF_q)$, and by Theorem~\ref{thm:parityBirMapIntro}, it induces an even permutation on $\bP^2(\bF_q)$. 
The map $\tau_2^{-1}\circ g\circ\tau_1$ is a generator of $\BCr_2(\bF_q)$ of the form (\ref{gs:2}), and we showed in Section~\ref{subsubsect:symmetric_generators} that it induces an even permutation on $\bP^2(\bF_q)$.
As a consequence, $f_{66}$ induces an even permutation on $\bP^2(\bF_q)$.
\end{proof}

The Bertini involution acting on the commutative diagram in the above proof is a tool used in \cite{LS21} to show that the Cremona group of rank $2$ over an arbitrary perfect field is generated by involutions, where it is called \emph{central symmetry} \cite{LS21}*{Corollary~4.4}.

\begin{proof}[Proof of Theorem~\ref{mainthm:intro}]
By Corollary~\ref{cor:conicBundle_v2}, the results proved in \S\ref{subsubsect:parityGenerator} and \S\ref{subsubsect:symmetric_generators}, and Proposition~\ref{prop:standardquinticseven}, it follows that all generators of $\BCr(\bF_q)$ induce even permutations on $\bP^2(\bF_q)$.
\end{proof}

\section{Basic properties on the bijective Cremona group}
\label{sect:basicProperties}

In this section, we prove that the group $\BCr_2(k)$ is not finitely generated in most situations and is of infinite index as a subgroup of $\Cr_2(k)$. We also show that $\BCr_2(k)$ is not a normal subgroup of $\Cr_2(k)$, and discuss whether the kernel of the homomorphism $\BCr_n(k)\to\Sym(\bP^n(k))$ is a normal subgroup of $\Cr_n(k)$ or not.

%----------Non-finite generation
\subsection{Non-finite generation}
\label{subsect:non-fg}

The Cremona group $\Cr_2(k)$ itself is not finitely-generated over any field $k$. (See \cite{Can12}*{Proposition~3.3} and \cite{Can18}*{Proposition~3.6}.) Here we prove that the same property holds for $\BCr_2(k)$ under the situations described below.

\begin{prop}
\label{prop:basicprop_fg}
Let $k$ be a field and let $k^s$ be a separable closure. The group $\BCr_2(k)$ is not finitely generated provided that
\begin{enumerate}[label=\textup{(\arabic*)}]
    \item\label{fg_uncountable}
    the field $k$ is uncountable,
    \item\label{fg_finite-to-closure}
    the degree $[k^s:k]$ is finite, or
    \item\label{fg_infinite-to-closure}
    the degree $[k^s:k]$ is infinite and $k$ admits a separable quadratic extension $T/k$.
\end{enumerate}
\end{prop}

We will prove the three statements in Proposition~\ref{prop:basicprop_fg} separately. The proofs for \ref{fg_uncountable} and \ref{fg_finite-to-closure} will come first as they are relatively shorter comparing to \ref{fg_infinite-to-closure}.

\begin{proof}[Proof of Proposition~\ref{prop:basicprop_fg}~\ref{fg_uncountable}]
If $k$ is uncountable, then $\operatorname{PGL}_{3}(k)\subset \operatorname{BCr}_2(k)$ is uncountable, thus $\operatorname{BCr}_2(k)$ cannot be a finitely generated group. (This proof was pointed out to us by Zinovy Reichstein.)
\end{proof}

\begin{proof}[Proof of Proposition~\ref{prop:basicprop_fg}~\ref{fg_finite-to-closure}]
Let $k_0$ be the prime field of $k$, which is either $\bQ$ or $\bF_{p}$ depending on the characteristic. For each $f\in\BCr_2(k)$, let $\Bs(f)\subset \bP^2$ denote the base scheme of $f$, and let $k_f/k_0$ be the minimal field extension over which every geometric point of $\Bs(f)$ and $\Bs(f^{-1})$ (including the infinitely near ones) is defined. Note that $f$ is defined over $k_f$ by definition, and $k_f$ may not contain $k$ in general.

Assume that $\BCr_2(k)$ is generated by a finite subset $\Gamma$ and let $k_{\Gamma}$ be the composite of $k_{f}$ for all $f\in\Gamma$. Since every $g\in\BCr_2(k)$ is a composition of elements of $\Gamma$, we have $k_g\subset k_\Gamma$. For every $a\in k$, the map
$
    g\colon [x:y:z]\mapsto[x+ay:y:z]
$
belongs to $\PGL_3(k)$, and thus $\BCr_2(k)$. Hence $k_g=k_0(a) \subset k_\Gamma$. This implies $k\subset k_\Gamma$ as $a\in k$ is arbitrary. Now we obtain a tower of field extensions
$$
    k_0\subset k\subset k_\Gamma
$$
where $k_\Gamma$ is finitely-generated over $k_0$ and $[k_\Gamma:k]$ is finite. By the Artin--Tate lemma \cite{AT51}*{Theorem~1}, $k$ is finitely-generated over $k_0$. Hence $[k:k_0]$ is finite. (See, e.g., \cite{AM69}*{Proposition~7.9}.) As $[k^s:k]$ is finite by hypothesis, we conclude that $[k^s:k_0]$ is finite, contradiction.
\end{proof}

When $k^s/k$ is an infinite extension, our strategy is to construct a sequence of elements in $\BCr_2(k)$ whose indeterminacy loci contain points of arbitrarily large degrees. The construction requires careful selections of the candidates for the indeterminacy points in $\bP^2$. Let us start with a few lemmas that help us deal with the positioning problem.

\begin{lemma}
\label{lem:4-points}
Suppose that $k$ is a field with $[k^s:k]=\infty$ and let $T/k$ be a separable quadratic extension. Then there exists four points $\{a_1, a_2, b_1, b_2\}$ in $\mathbb{P}^2(T)$ such that $\{a_1, a_2\}$ and $\{b_1, b_2\}$ form $\Gal(T/k)$-orbits, and no three of them are collinear.
\end{lemma}

\begin{proof}
Since $T/k$ is separable, there exists a point in $\bP^2$ of degree $2$ that is reduced: we may take $a_1 = [a:1:0]$ and $a_2 = [a':1:0]$, where $a,a'\in T\setminus k$ are the distinct roots of an irreducible quadratic polynomial over $k$. Take $\beta\subset\bP^2$ to be any $k$-line not spanned by $a_1$ and $a_2$. Since $\beta\cong\bP^1$ over $k$, we can find a pair of Galois-conjugate points $\{b_1,b_2\}$ on $\beta$ in a similar way as before. Then $\{a_1,a_2\}$ and $\{b_1,b_2\}$ satisfy the requirements.
\end{proof}

As a consequence of Lemma~\ref{lem:4-points}, there exists a unique conic $C_x$ through $\{a_1,a_2,b_1,b_2,x\}$ for every $x\in\bP^2\setminus\{a_1, a_2, b_1, b_2\}$,
which degenerates if and only if $x$ lies on the line spanned by any two of the four points \cite{BKT08}*{Theorem~1}. All but three of these conics are smooth,
and the three degenerate ones are
\begin{equation}
\label{eqn:degconicfibers}
\begin{aligned}
	C_0 = \mathrm{span}(a_1,a_2)\cup\mathrm{span}(b_1,b_2),\\
    C_1 = \mathrm{span}(a_1,b_1)\cup\mathrm{span}(a_2,b_2),\\
    C_2 = \mathrm{span}(a_1,b_2)\cup\mathrm{span}(a_2,b_1).
\end{aligned}
\end{equation}
Note that these curves are defined over $k$.

\begin{lemma}
\label{lem:evenOrbit}
Retain the notation from Lemma~\ref{lem:4-points}. Let $\ell_1\subset\bP^2$ be a line over $T$ passing through $a_1$, but not $a_2,b_1,b_2$, and let $\ell_2$ be its $\Gal(T/k)$-conjugate. Let $K_0/k$ be a non-trivial Galois extension different from $T$ and let $K=K_0T$ be the composite field. Then there exists a closed point $x\in\ell_1$ defined over $K/k$ but not over any proper subfield, such that
\begin{enumerate}[label=\textup{(\arabic*)}]
\item\label{evenOrb:conj}
Let $r = [K:T]$. Then $r$ of the $\Gal(K/k)$-conjugates of $x$ lie on $\ell_1$ (resp. $\ell_2$).

\item\label{evenOrb:smConic}
Let $x = x_1,...,x_{2r}$ be the $\Gal(K/k)$-conjugates of $x$.
For each $1\leq i\leq 2r$, the unique conic passing through $\{a_1, a_2, b_1, b_2, x_i\}$ is smooth.

\item\label{evenOrb:distinct}
If $x_i, x_j$ are any two distinct conjugates of $x$,
the six points $a_1$, $a_2$, $b_1$, $b_2$, $x_i$, $x_j$ do not lie on a conic.
\end{enumerate}
\end{lemma}

\begin{proof}
Consider the $\bP^1$ that parametrizes the conics passing through $a_1,a_2, b_1, b_2$. By the primitive element theorem, $K=k(z)$ for some $z\in K$, which can be seen as a $K$-point $z\in\bP^1(K) = K\cup\{\rm pt\}$. Let $\{z=z_1, ..., z_{2r}\}$ be the Galois orbit of $z$ in the base $\mathbb{P}^1$, and let $F_1, ..., F_{2r}$ be the conics in $\bP^2$ corresponding to these orbit points. Here we index the points in a way that the action of $\Gal(K/T)$ preserves the parities of the indices. In particular, the conic $F_i$ with odd $i$ (resp. even $i$) intersects $\ell_1$ (resp. $\ell_2$) at $a_1$ (resp. at $a_2$), and it cannot be tangent to $\ell_1$ (resp. to $\ell_2$) since otherwise it would be defined over $T$.

Let $x_i$ be the residual intersection of $F_i$ with $\ell_1$ (resp. with $\ell_2$) for odd $i$ (resp. for even $i$) and let $x=x_1$. By construction, these points are all distinct, form an orbit under the action of $\Gal(K/k)$, and equally distribute on $l_1$ and $l_2$, which proves \ref{evenOrb:conj}. Property~\ref{evenOrb:smConic} holds since each $F_i$ is defined over $K$ but not over any proper subfield, while the three degenerate conics $C_0$, $C_1$, $C_2$ are defined over $k$. Finally, if the set $\{a_1, a_2, b_1, b_2, x_i, x_j\}$ where $i\neq j$ lies on a conic $C$, then $C=F_i=F_j$, which contradicts the construction. This proves \ref{evenOrb:distinct}.
\end{proof}

\begin{lemma}
\label{lem:largebase}
Retain the notation from Lemma~\ref{lem:evenOrbit}. Then there exists $f\in\BCr_2(k)$ whose indeterminacy locus contains a point of degree $[K:k]$ over $k$.
\end{lemma}

\begin{proof}
The construction is accomplished via the following steps:
\begin{enumerate}[label=(\arabic*)]
\item Pick four points $a_1,a_2,b_1,b_2\in\bP^2$ as in Lemma~\ref{lem:4-points}.
blow-up $\bP^2$ along $\{a_1,a_2,b_1,b_2\}$ to obtain a conic bundle $\mathcal{C}\rightarrow\bP^1$ fibered in the conics passing through $\{a_1,a_2,b_1, b_2\}$.
Recall that only three of the fibers are degenerate, namely, $C_0$, $C_1$, $C_2$ defined in \eqref{eqn:degconicfibers}.
The exceptional divisors $A_1,A_2,B_1,B_2$ over $a_1,a_2,b_1,b_2$, respectively, form four sections of the bundle. Moreover, the $\Gal(K/k)$-action exchanges the irreducible components of the two singular fibres $C_1$ and $C_2$.

\item Let $x$ be the point obtained in Lemma~\ref{lem:evenOrbit} and consider it as a point on $\mathcal{C}$. 
blow-up $\mathcal{C}$ along the $\Gal(K/k)$-orbit of $x$ to obtain a map $X\to\mathcal{C}$. The strict transform of the fibers of $\mathcal{C}\to\bP^1$ containing $x$ is a $\Gal(K/k)$-orbit of $(-1)$-curves $F_1,\dots,F_{2r}$ by Lemma~\ref{lem:evenOrbit}\ref{evenOrb:smConic}. Using Castelnuovo's contractibility criterion in positive characteristics \cite{Bad01}*{Theorem~3.30}, blow down $F_1, ..., F_{2r}$ to get $X\to\mathcal{C}'$, and $\mathcal{C}'$ is a conic fibration over $\bP^1$.
The induced birational map $\phi\colon \mathcal{C}\dashrightarrow\mathcal{C}'$ preserves the conic fibrations.

\item The birational map $\phi$ is regular around the singular fibers of $\mathcal{C}\to\bP^1$, so $\phi(C_0),\phi(C_1),\phi(C_2)$ are the singular fibres of $\mathcal{C}'$ and the $\Gal(K/k)$-action exchanges the irreducible components of $\phi(C_1)$ and $\phi(C_2)$. Hence $K_{\mathcal{C}'}^2=5$. Sine $\mathcal{C}'$ has a $k$-point, it follows from \cite{Schneider20}*{Lemma 6.5} that there is a birational morphism $\mathcal{C}'\to\bP^2$ contracting a $\Gal(K/k)$-orbit $O$ of four points. Since $\mathcal{C}'\to\bP^1$ has three singular fibres, one of which has $\Gal(K/k)$-invariant components, it follows that $O$ is the union of two $\Gal(K/k)$-orbits $\{a_1',a_2'\}$ and $\{b_1',b_2'\}$.
\end{enumerate}

The desired Cremona map $f$ is then obtained from the composition
\begin{equation}
\label{eqn:PCXCP}
\vcenter{\vbox{
\xymatrix@C=12pt@R=12pt{
	&& X\ar[dl]_{\text{contracting }E_i\text{'s}}\ar[dr]^{\text{contracting }F_i\text{'s}} &&\\
	& \mathcal{C}\ar[dl]\ar@{-->}[rr]^-\phi && \mathcal{C}'\ar[dr] &\\
    \bP^2\ar@{-->}[rrrr]^{f} &&&& \bP^2.
}}}
\end{equation}
The map $f$ belongs to $\Cr_2(k)$ since it is composed from maps defined over $k$.
As for the indeterminacy loci, we have
$$
	\Bs(f) = \{a_1, a_2, b_1, b_2, x_1, \dots, x_{2r}\}
	,\qquad
	\Bs(f^{-1}) = \{a_1', a_2', b_1', b_2', y_1, \dots, y_{2r}\}
$$
where $y_1, \dots, y_{2r}$ are the images of $F_1, ..., F_{2r}$ in the final $\bP^2$. This shows that $f\in\BCr_2(k)$, and $\Bs(f)$ contains the $\Gal(K/k)$-orbit $\{x_1, ..., x_{2r}\}$ of size $2r = [K:k]$.
\end{proof}

\begin{proof}[Proof of Proposition~\ref{prop:basicprop_fg}~\ref{fg_infinite-to-closure}]
Let $k'_f/k$ be the be the minimal field extension over which every geometric point of $\Bs(f)$ and $\Bs(f^{-1})$ (including the infinitely near ones) is defined. If $\BCr_2(k)$ is finitely-generated by $f_1, f_2, ..., f_r$, then for each $f\in\BCr_2(k)$, $k'_f$ would be contained in the composite of $k'_{f_1}, ..., k'_{f_r}$, and so
$$
	[k'_f:k] \leq \prod_{i=1}^{r} [k'_{f_i}:k],
$$
which implies that the set of integers $\{[k'_f:k] \ {:} \ f\in\BCr_2(k)\}$ is bounded. The assumption $[k^s:k]=\infty$ guarantees that $k$ admits a Galois extension $K_0/k$ such that $K=K_0T$ has arbitrarily large degree $d$ over $k$. By Lemma~\ref{lem:largebase}, there exists $h\in\BCr_2(k)$ whose indeterminacy locus contains a point of degree $d$ over $k$ and hence $d\leq [k_h':k]$, contradiction. 
\end{proof}

%----------The infinite index
\subsection{The infinite index}
\label{subsect:infty-index}

The construction of the Cremona maps in Lemma~\ref{lem:largebase} can be used to show that $\BCr_2(k)$ is of infinite index as a subgroup of $\Cr_2(k)$. Before proving this statement, let us remark that the transformation between conic bundles $\mathcal{C}\dashrightarrow\mathcal{C}'$ in the proof of Lemma~\ref{lem:largebase} is a Sarkisov link of type~II. The discovery of the induced Cremona maps can date back to 1877 by Ruffini, whose homaloidal type, as computed in the following lemma, is documented in \cite{Hud24}*{page~234}.

\begin{lemma}
\label{lem:Ruffini}
Consider the Cremona map (\ref{eqn:PCXCP}). Let $M\in\Pic(X)$ be the pullback of a line class from the right $\bP^2$. Then
$$
	M = (2n+1)L - 2\sum_{i=1}^nE_i - n(A_1 + A_2 + B_1 + B_2)
$$
where $n = 2r$ is the cardinality of the large Galois orbit.
\end{lemma}

\begin{proof}
The fiber class $F$ corresponds to a conic in the right $\bP^2$ passing through $a_1', a_2', b_1', b_2'$,
so the class in $\Pic(X)$ corresponding to a conic from the right $\bP^2$ equals
\begin{gather*}
	2M = F + A_1' + A_2' + B_1' + B_2'
    = F + A_1 + A_2 + B_1 + B_2 + 2nF - 4\sum_{i=1}^{n}E_i\\
    = (2n+1)(2L - A_1 - A_2 - B_1 - B_2) - 4\sum_{i=1}^{n}E_i\\
    = (4n+2)L - 2n(A_1 + A_2 + B_1 + B_2) - 4\sum_{i=1}^{n}E_i.
\end{gather*}
Divide both sides by 2 to get the result.
\end{proof}

\begin{prop}
\label{prop:baseprop_ii}
Let $k$ be any field. Then $\BCr_2(k)\subset\Cr_2(k)$ is a subgroup of infinite index.
\end{prop}

\begin{proof}
First assume that $k$ is infinite. Let us construct inductively an infinite sequence of maps $f_1, f_2, f_3, ... $ in $\Cr_2(k)$ as follows: Let $f_1$ be the identity map. Suppose that $f_i$ is constructed and let $U\subset\bP^2$ be the open subset such that $f_i|_U$ is an isomorphism. As $k$ is infinite, we can take three non-collinear points
$
	\{a,b,c\}\subset U(k).
$
Define $f_{i+1}\colonequals\tau\circ f_i$ where $\tau$ is the quadratic transformation with $\Bs(\tau)=\{f_i(a),f_i(b),f_i(c)\}$. Then we have
$$
	|\Bs(f_{i+1})(k)|\geq |\Bs(f_{i})|+3.
$$
Note that the left cosets $f_1 \BCr_2(k)$, $f_{2} \BCr_2(k), \dots$ are all pairwise disjoint because the elements in $\BCr_2(k)$ cannot increase the indeterminacy points of $f_{i}$ in $\mathbb{P}^2(k)$.

Now assume that $k=\bF_{q}$ is a finite field. The same idea as in the proof of Lemma~\ref{lem:4-points} produces four points $a_1,a_2,b_1,b_2\in\bP^2(\bF_q)$ such that no three are collinear. The main construction of the Cremona map carried out in Lemma~\ref{lem:largebase} still works, and for each even integer $n=2r$, we get a map $f_{r}\in\Cr_2(\bF_q)$ such that $\Bs(f_r)$ supports at $a_1, a_2, b_1, b_2$ with multiplicity $2r$ (Lemma~\ref{lem:Ruffini}). We obtain an infinite sequence $\{f_1,f_2, f_3,  \ldots\}$ of elements in $\Cr_2(\bF_q)$ such that the left cosets $f_1 \BCr_2(\bF_q)$, $f_{2} \BCr_2(\bF_q), \dots$ are all pairwise disjoint. Indeed, for any $g\in\BCr_2(k)$ the multiplicity of $f_rg$ at $a_1,a_2,b_1,b_2$ is equal to $2r$.
\end{proof}

%----------On the non-normality
\subsection{On the non-normality}
\label{subsect:non-normal}

Over an algebraically closed field $k$, Blanc \cite{Bla10}*{Theorem~4.2} proved that $\Cr_2(k)$ has no non-trivial closed normal subgroup with respect to its natural topology. On the other hand, Cantat and Lamy \cite{CL13} proved that $\Cr_2(k)$ is not simple as an abstract group, and Lonjou generalized this result to any field $k$ \cite{L15}. Here we prove that $\BCr_2(k)$ is not a normal subgroup of $\Cr_2(k)$. For the kernel of the homomorphism $\BCr_n(k)\to\Sym(\bP^n(k))$, we prove that it is not a normal subgroup of $\Cr_n(k)$ when $k$ is finite and that it is trivial when $k$ is infinite.

\begin{prop}
\label{prop:baseprop_ns}
For any field $k$, the group $\BCr_2(k)$ is not a normal subgroup of $\Cr_2(k)$.
\end{prop}

\begin{proof} 
Let $f\in\Cr_2(k)$ be the standard quadratic involution $f : [x: y: z]\mapsto [yz: zx: xy]$ and $g\in\PGL_3(k)\subset\BCr_2(k)$ be any map sending $[1:0:0]$ to $[1:1:1]$.
Then $f^{-1}gf$ contracts the line $\{x=0\}$ to the point
$$
	f^{-1}gf([0: y: z]) = f^{-1}g([1: 0: 0]) = f^{-1}([1: 1: 1]) = [1: 1: 1].
$$
Therefore, $(f^{-1}gf)^{-1} = f^{-1}g^{-1}f$ possesses a $k$-point in its indeterminacy locus, and thus cannot be an element of $\BCr_2(k)$.
\end{proof}

\begin{prop}
\label{prop:not_normal}
Let $k$ be a finite field. Then the kernel of $\BCr_n(k)\rightarrow \Sym(\bP^n(k))$, where $n\geq2$, is not a normal subgroup of $\Cr_n(k)$.
\end{prop}

\begin{proof}
Let $N$ denote the kernel of $\BCr_n(k)\rightarrow \Sym(\bP^2(k))$.
Suppose, to the contrary, that $N$ is a normal subgroup of $\Cr_n(k)$.
Let $\ell\in k[x_2,\dots,x_n]$ be a linear homogeneous polynomial. Consider the birational map
\[
    f\colon[x_0:\dots:x_n]\mapsto
    [x_0^2:x_1\ell:x_0x_2:\dots:x_0x_n]
\]
and its inverse
\[
    f^{-1}\colon [x_0:\dots:x_n]\mapsto
    [\ell x_0:x_1x_0:x_2\ell:\dots:x_n\ell].
\]
A straightforward computation shows that both $f$ and $f^{-1}$ contract two and only two hypersurfaces, namely, the hyperplanes $\{x_0=0\}$ and $\{\ell=0\}$.
Moreover, $f$ and $f^{-1}$, respectively, contracts the union $\{x_0=0\}\cup\{\ell=0\}$ onto
\[
    \Bs(f^{-1}) = \{\ell=x_0=0\}\cup\{\ell=x_1=0\},\quad
    \Bs(f) = \{x_0=x_1=0\}\cup\{x_0=\ell=0\}.
\]

For every $g\in N$, we claim that
\begin{equation}
\label{eqn:preserveHypers}
    g(\{x_0=0\}\cup\{\ell=0\})
    = \{x_0=0\}\cup\{\ell=0\}.
\end{equation}
First note that $g(\{x_0=0\}\cup\{\ell=0\})$ is a hypersurface due to the facts that $g$ is bijective and that $\{x_0=0\}\cup\{\ell=0\}$ contains $k$-points.
Since $N$ is normal, we have $fg=hf$ for some $h\in N$.
Suppose that $g(\{x_0=0\}\cup\{\ell=0\})$ is not contained in $\{x_0=0\}\cup\{\ell=0\}$.
Then $fg(\{x_0=0\}\cup\{\ell=0\})$ is a hypersurface while $hf(\{x_0=0\}\cup\{\ell=0\})$ is not, contradiction.
Therefore, we have
\[
    g(\{x_0=0\}\cup\{\ell=0\})
    \subset
    \{x_0=0\}\cup\{\ell=0\}.
\]
The same argument with $g$ replaced by $g^{-1}$ implies that
\[
    \{x_0=0\}\cup\{\ell=0\}
    \subset
    g(\{x_0=0\}\cup\{\ell=0\}).
\]
Hence (\ref{eqn:preserveHypers}) follows.
By applying the same argument with $f$ replaced by $\alpha f\alpha^{-1}$ for any $\alpha\in\Aut(\bP^n)$, we conclude that (\ref{eqn:preserveHypers}) holds for any union of two distinct rational hyperplanes.
This implies that $g$ preserves any rational hyperplane of $\bP^n$. 

Write $g\in N$ as
$
    g([x_0:\dots:x_n]) = [g_0:\dots:g_n]
$
where $g_i\in k[x_0,\dots,x_n]$ are homogeneous polynomials without a common factor.
As $g$ preserves each coordinate hyperplane $\{x_i=0\}$,
we have $g_i=x_ig_i'$ for some $g_i'\in k[x_0,\dots,x_n]$.
The fact that $g^{-1}$ also preserves each $\{x_i=0\}$ then implies that
\[
    g^{-1}(\{x_i=0\}) = \{g_i = 0\} = \{x_ig_i'=0\} = \{x_i=0\}
\]
hence $x_ig_i' = a_ix_i$ for some $a_i\in k^*$.
Therefore $g_i' = a_i\in k^*$ for all $i$ and so $g$ is linear.
Since $g\in N$, it fixes $|\bP^n(k)|=q^n+q^{n-1}+\cdots+q+ 1\geq n+2$ points in $\bP^n$, and thus equal to the identity map.
We conclude that $N=\{\mathrm{Id}\}$, which is a contradiction because $\BCr_n(k)$ is infinite by Lemma~\ref{lem:largebase} and $N$ is never trivial as it is of finite index in $\BCr_n(k)$.
\end{proof}

\begin{prop}
If $k$ is an infinite field, then $\BCr_n(k)\to\Sym(\bP^n(k))$, where $n\geq 1$, is injective.
\end{prop}

\begin{proof}
Every element in the kernel of $\BCr_n(k)\to\Sym(\bP^n(k))$ fixes $\bP^n(k)$, which is a Zariski dense subset of $\bP^n$. This forces such an element to be the identity map.
\end{proof}

\bigskip
\bibliography{BijectiveCremona_bib}

\begin{bibdiv}
\begin{biblist}

\bib{AM69}{book}{
      author={Atiyah, M.~F.},
      author={Macdonald, I.~G.},
       title={Introduction to commutative algebra},
   publisher={Addison-Wesley Publishing Co., Reading, Mass.-London-Don Mills,
  Ont.},
        date={1969},
      review={\MR{0242802}},
}

\bib{AT51}{article}{
      author={Artin, Emil},
      author={Tate, John~T.},
       title={A note on finite ring extensions},
        date={1951},
        ISSN={0025-5645},
     journal={J. Math. Soc. Japan},
      volume={3},
       pages={74\ndash 77},
         url={https://doi.org/10.2969/jmsj/00310074},
      review={\MR{44509}},
}

\bib{Bad01}{book}{
      author={B\u{a}descu, Lucian},
       title={Algebraic surfaces},
      series={Universitext},
   publisher={Springer-Verlag, New York},
        date={2001},
        ISBN={0-387-98668-5},
         url={https://doi.org/10.1007/978-1-4757-3512-3},
        note={Translated from the 1981 Romanian original by Vladimir Ma\c{s}ek
  and revised by the author},
      review={\MR{1805816}},
}

\bib{Magma}{article}{
      author={Bosma, Wieb},
      author={Cannon, John},
      author={Playoust, Catherine},
       title={The {M}agma algebra system. {I}. {T}he user language},
        date={1997},
        ISSN={0747-7171},
     journal={J. Symbolic Comput.},
      volume={24},
      number={3-4},
       pages={235\ndash 265},
         url={http://dx.doi.org/10.1006/jsco.1996.0125},
        note={Computational algebra and number theory (London, 1993)},
      review={\MR{MR1484478}},
}

\bib{Bha81}{article}{
      author={Bhattacharya, Prabir},
       title={On groups containing the projective special linear group},
        date={1981},
        ISSN={0003-889X},
     journal={Arch. Math. (Basel)},
      volume={37},
      number={4},
       pages={295\ndash 299},
         url={https://doi.org/10.1007/BF01234360},
      review={\MR{639332}},
}

\bib{BKT08}{article}{
      author={Bashelor, Andrew},
      author={Ksir, Amy},
      author={Traves, Will},
       title={Enumerative algebraic geometry of conics},
        date={2008},
        ISSN={0002-9890},
     journal={Amer. Math. Monthly},
      volume={115},
      number={8},
       pages={701\ndash 728},
         url={https://doi.org/10.1080/00029890.2008.11920584},
      review={\MR{2456094}},
}

\bib{Bla10}{article}{
      author={Blanc, J\'er\'emy},
       title={Groupes de {C}remona, connexit\'e et simplicit\'e},
        date={2010},
        ISSN={0012-9593},
     journal={Ann. Sci. \'Ec. Norm. Sup\'er. (4)},
      volume={43},
      number={2},
       pages={357\ndash 364},
         url={https://doi.org/10.24033/asens.2123},
      review={\MR{2662668}},
}

\bib{BM14}{article}{
      author={Blanc, J\'er\'emy},
      author={Mangolte, Fr\'ed\'eric},
       title={Cremona groups of real surfaces},
        date={2014},
     journal={Automorphisms in birational and affine geometry, Springer Proc.
  Math. Stat},
      volume={79},
       pages={35–58},
}

\bib{Can09}{article}{
      author={Cantat, Serge},
       title={Birational permutations},
        date={2009},
        ISSN={1631-073X},
     journal={C. R. Math. Acad. Sci. Paris},
      volume={347},
      number={21-22},
       pages={1289\ndash 1294},
         url={https://doi.org/10.1016/j.crma.2009.09.019},
      review={\MR{2561040}},
}

\bib{Can12}{article}{
      author={Cantat, Segre},
       title={Generators for the cremona group},
        date={2012},
        note={Online access:
  \href{https://perso.univ-rennes1.fr/serge.cantat/Articles/hudson-pan-derksen.pdf}{\\https://perso.univ-rennes1.fr/serge.cantat/Articles/hudson-pan-derksen.pdf}},
}

\bib{Can18}{incollection}{
      author={Cantat, Segre},
       title={The {C}remona group},
        date={2018},
   booktitle={Algebraic geometry---{S}alt {L}ake {C}ity 2015. {P}art 1},
      series={Proc. Sympos. Pure Math.},
      volume={97},
   publisher={Amer. Math. Soc., Providence, RI},
       pages={101\ndash 142},
}

\bib{CL13}{article}{
      author={Cantat, Serge},
      author={Lamy, St\'ephane},
       title={Normal subgroups in the {C}remona group},
        date={2013},
        ISSN={0001-5962},
     journal={Acta Math.},
      volume={210},
      number={1},
       pages={31\ndash 94},
         url={https://doi.org/10.1007/s11511-013-0090-1},
        note={With an appendix by Yves de Cornulier},
      review={\MR{3037611}},
}

\bib{Coh83}{article}{
      author={Cohen, Stephen~D.},
       title={Primitive roots in the quadratic extension of a finite field},
        date={1983},
        ISSN={0024-6107},
     journal={J. London Math. Soc. (2)},
      volume={27},
      number={2},
       pages={221\ndash 228},
         url={https://doi.org/10.1112/jlms/s2-27.2.221},
      review={\MR{692527}},
}

\bib{DI09}{incollection}{
      author={Dolgachev, Igor~V.},
      author={Iskovskikh, Vasily~A.},
       title={Finite subgroups of the plane {C}remona group},
        date={2009},
   booktitle={Algebra, arithmetic, and geometry: in honor of {Y}u. {I}.
  {M}anin. {V}ol. {I}},
      series={Progr. Math.},
      volume={269},
   publisher={Birkh\"{a}user Boston, Inc., Boston, MA},
       pages={443\ndash 548},
         url={https://doi.org/10.1007/978-0-8176-4745-2_11},
      review={\MR{2641179}},
}

\bib{GLU21}{article}{
      author={Genevois, Anthony},
      author={Lonjou, Anne},
      author={Urech, Christian},
       title={Cremona groups over finite fields, {N}eretin groups, and
  non-positively curved cube complexes},
        date={2021},
        note={Preprint,
  \href{https://arxiv.org/abs/2110.14605v1}{arXiv:2110.14605v1}},
}

\bib{GS17}{book}{
      author={Gille, Philippe},
      author={Szamuely, Tam\'{a}s},
       title={Central simple algebras and {G}alois cohomology},
      series={Cambridge Studies in Advanced Mathematics},
   publisher={Cambridge University Press, Cambridge},
        date={2017},
      volume={165},
        ISBN={978-1-316-60988-0; 978-1-107-15637-1},
      review={\MR{3727161}},
}

\bib{Hud24}{article}{
      author={Hudson, Hilda~P.},
       title={Plane {H}omaloidal {F}amilies of {G}eneral {D}egree},
        date={1924},
        ISSN={0024-6115},
     journal={Proc. London Math. Soc. (2)},
      volume={22},
       pages={223\ndash 247},
         url={https://doi.org/10.1112/plms/s2-22.1.223},
      review={\MR{1575704}},
}

\bib{Isk79}{article}{
      author={Iskovskih, V.~A.},
       title={Minimal models of rational surfaces over arbitrary fields},
        date={1979},
        ISSN={0373-2436},
     journal={Izv. Akad. Nauk SSSR Ser. Mat.},
      volume={43},
      number={1},
       pages={19\ndash 43, 237},
      review={\MR{525940}},
}

\bib{Isk91}{article}{
      author={Iskovskikh, V.~A.},
       title={Generators of the two-dimensional {C}remona group over a
  nonclosed field},
        date={1991},
        ISSN={0371-9685},
     journal={Trudy Mat. Inst. Steklov.},
      volume={200},
       pages={157\ndash 170},
      review={\MR{1143365}},
}

\bib{Isk96}{article}{
      author={Iskovskikh, V.~A.},
       title={Factorization of birational mappings of rational surfaces from
  the point of view of {M}ori theory},
        date={1996},
        ISSN={0042-1316},
     journal={Uspekhi Mat. Nauk},
      volume={51},
      number={4(310)},
       pages={3\ndash 72},
         url={https://doi.org/10.1070/RM1996v051n04ABEH002962},
      review={\MR{1422227}},
}

\bib{KM74}{article}{
      author={Kantor, W.~M.},
      author={McDonough, T.~P.},
       title={On the maximality of {${\rm PSL}(d+1,\,q),\,d\geq 2$}},
        date={1974},
        ISSN={0024-6107},
     journal={J. London Math. Soc. (2)},
      volume={8},
       pages={426},
         url={https://doi.org/10.1112/jlms/s2-8.3.426},
      review={\MR{0376896}},
}

\bib{Kol07}{book}{
      author={Koll\'{a}r, J\'{a}nos},
       title={Lectures on resolution of singularities},
      series={Annals of Mathematics Studies},
   publisher={Princeton University Press, Princeton, NJ},
        date={2007},
      volume={166},
        ISBN={978-0-691-12923-5; 0-691-12923-1},
      review={\MR{2289519}},
}

\bib{kollar1999rational}{book}{
      author={Kollar, J.},
       title={Rational curves on algebraic varieties},
      series={Ergebnisse der Mathematik und ihrer Grenzgebiete. 3. Folge / A
  Series of Modern Surveys in Mathematics},
   publisher={Springer Berlin Heidelberg},
        date={1999},
        ISBN={9783540601685},
         url={https://books.google.fr/books?id=oqW3GabJLjgC},
}

\bib{Lis75}{article}{
      author={List, R.},
       title={On permutation groups containing {${\rm PSL}_{n}(q)$} as a
  subgroup},
        date={1975},
        ISSN={0046-5755},
     journal={Geometriae Dedicata},
      volume={4},
      number={2/3/4},
       pages={373\ndash 375},
         url={https://doi.org/10.1007/BF00148769},
      review={\MR{0409614}},
}

\bib{L15}{article}{
      author={Lonjou, Anne},
       title={Non simplicité du groupe de cremona sur tout corps},
        date={2016},
     journal={Ann. Inst. Fourier (Grenoble)},
      volume={66},
       pages={2021–2046},
}

\bib{LS21}{article}{
      author={Lamy, St\'ephane},
      author={Schneider, Julia},
       title={Generating the {C}remona groups by involutions},
        date={2021},
        note={Preprint,
  \href{https://arxiv.org/abs/2110.02851v1}{arXiv:2110.02851v1}},
}

\bib{Man86}{book}{
      author={Manin, Yu.~I.},
       title={Cubic forms},
     edition={Second},
      series={North-Holland Mathematical Library},
   publisher={North-Holland Publishing Co., Amsterdam},
        date={1986},
      volume={4},
        ISBN={0-444-87823-8},
        note={Algebra, geometry, arithmetic, Translated from the Russian by M.
  Hazewinkel},
      review={\MR{833513}},
}

\bib{Pog74}{article}{
      author={Pogorelov, B.~A.},
       title={Maximal subgroups of symmetric groups that are defined on
  projective spaces over finite fields},
        date={1974},
        ISSN={0025-567X},
     journal={Mat. Zametki},
      volume={16},
       pages={91\ndash 100},
      review={\MR{0357562}},
}

\bib{Poo17}{book}{
      author={Poonen, Bjorn},
       title={Rational points on varieties},
      series={Graduate Studies in Mathematics},
   publisher={American Mathematical Society, Providence, RI},
        date={2017},
      volume={186},
        ISBN={978-1-4704-3773-2},
      review={\MR{3729254}},
}

\bib{Schneider20}{book}{
      author={Schneider, Julia},
       title={Relations in the cremona group over perfect fields},
   publisher={Annales de l'Institut Fourier (to appear)},
        date={2020},
}

\bib{Sil09}{book}{
      author={Silverman, Joseph~H.},
       title={The arithmetic of elliptic curves},
     edition={Second},
      series={Graduate Texts in Mathematics},
   publisher={Springer, Dordrecht},
        date={2009},
      volume={106},
        ISBN={978-0-387-09493-9},
         url={https://doi.org/10.1007/978-0-387-09494-6},
      review={\MR{2514094}},
}

\bib{Wat89}{article}{
      author={Waterhouse, William~C.},
       title={Two generators for the general linear groups over finite fields},
        date={1989},
        ISSN={0308-1087},
     journal={Linear and Multilinear Algebra},
      volume={24},
      number={4},
       pages={227\ndash 230},
         url={https://doi.org/10.1080/03081088908817916},
      review={\MR{1007258}},
}

\bib{Wei56}{inproceedings}{
      author={Weil, Andr\'{e}},
       title={Abstract versus classical algebraic geometry},
        date={1956},
   booktitle={Proceedings of the {I}nternational {C}ongress of
  {M}athematicians, 1954, {A}msterdam, vol. {III}},
   publisher={Erven P. Noordhoff N.V., Groningen; North-Holland Publishing Co.,
  Amsterdam},
       pages={550\ndash 558},
      review={\MR{0092196}},
}

\end{biblist}
\end{bibdiv}
\bibliographystyle{alpha}

\ContactInfo
\end{document}